\documentclass[12pt]{amsart}
\usepackage{amscd,amssymb,amsthm}

\DeclareMathOperator{\inj}{inj}
\DeclareMathOperator{\diam}{diam}
\DeclareMathOperator{\Ric}{Ric}
\DeclareMathOperator{\Rm}{Rm}
\newcommand{\RDT}{\mathrm{RDT}}
\newcommand{\Lie}{\mathcal L}
\newcommand{\eps}{\varepsilon}

\newcommand{\ol}{\overline}
\newcommand{\RR}{\mathbb R}
\newcommand{\ZZ}{\mathbb Z}
\newcommand{\TT}{\mathbb T}

\newcommand{\Nil}{\operatorname{Nil}}
\newcommand{\pt}{\operatorname{pt}}
\newcommand{\SL}{\operatorname{SL}}
\newcommand{\Sol}{\operatorname{Sol}}
\newcommand{\Tr}{\operatorname{Tr}}
\newcommand{\Id}{\operatorname{Id}}
\newcommand{\BG}{\operatorname{BG}}
\newcommand{\PBG}{\operatorname{PBG}}
\newcommand{\proofstep}[1]{%
  \par\medskip
  \noindent\textbf{#1}\par\nobreak\smallskip
}

\numberwithin{equation}{section}
\theoremstyle{plain}
\newtheorem{theorem}{Theorem}[section]
\newtheorem{proposition}[theorem]{Proposition}
\newtheorem{lemma}[theorem]{Lemma}
\newtheorem{corollary}[theorem]{Corollary}

\theoremstyle{definition}
\newtheorem{definition}[theorem]{Definition}

\theoremstyle{remark}
\newtheorem{remark}[theorem]{Remark}
\usepackage{graphicx}
\usepackage{float}
\usepackage{xcolor}
\definecolor{InternalLinkBlue}{RGB}{0,61,130}
\usepackage{tikz}
\usetikzlibrary{arrows.meta,positioning,calc,fit,shapes.geometric,decorations.pathreplacing}
\usepackage{array,longtable}
\usepackage[
  colorlinks=true,
  linkcolor=InternalLinkBlue,
  citecolor=InternalLinkBlue,
  urlcolor=InternalLinkBlue,
  linktoc=all,
  bookmarksopen=true,
  bookmarksnumbered=true
]{hyperref}
\hypersetup{
  pdftitle={Under Ricci flow, a 3-torus goes flat},
  pdfauthor={John Lott},
  pdfsubject={Long-time behavior of Ricci flow in the Euclidean, Nil, and Sol geometries},
  pdfkeywords={Ricci flow on three-manifolds, Type-III Ricci flow, torus bundles, collapsing, Gromov-Hausdorff convergence, homogeneous Ricci solitons}
}
\newsavebox{\dependencytreebox}
\newcounter{dependencylink}
\newcommand{\DependencyTreeLink}[5]{%
  \stepcounter{dependencylink}%
  \pgfmathsetlengthmacro{\deplinkwidth}{#3*\wd\dependencytreebox}%
  \pgfmathsetlengthmacro{\deplinkheight}{#4*\ht\dependencytreebox}%
  \edef\deplinkid{\arabic{dependencylink}}%
  \coordinate (depbottom\deplinkid) at ($(depimage.south west)!#1!(depimage.south east)$);%
  \coordinate (deptop\deplinkid) at ($(depimage.north west)!#1!(depimage.north east)$);%
  \node[inner sep=0pt,outer sep=0pt] at ($(depbottom\deplinkid)!#2!(deptop\deplinkid)$) {%
    \hyperref[#5]{\makebox[\deplinkwidth][c]{\rule{0pt}{\deplinkheight}}}};%
}

\newcommand{\DependencyTreeI}{%
  \begingroup
  \sbox{\dependencytreebox}{\includegraphics[width=\textwidth,height=.88\textheight,keepaspectratio]{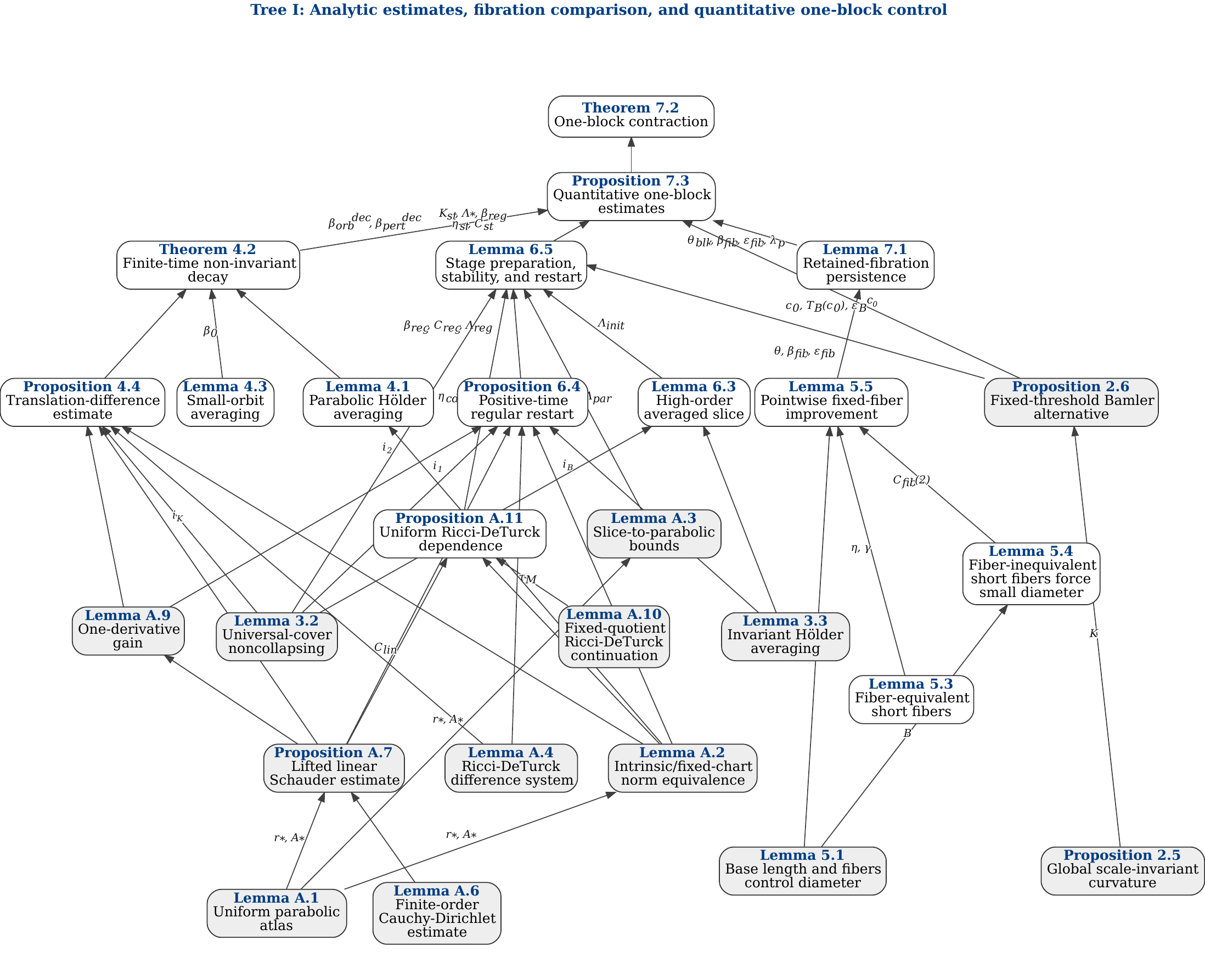}}%
  \setcounter{dependencylink}{0}%
  \begin{tikzpicture}[baseline=(depimage.base)]
    \node[inner sep=0pt,outer sep=0pt] (depimage) {\usebox{\dependencytreebox}};
    \DependencyTreeLink{0.5239565}{0.8809805}{0.1328427}{0.0395589}{thm:one-block-contraction}
    \DependencyTreeLink{0.5239565}{0.7992996}{0.1346639}{0.0460000}{prop:one-block-quantitative}
    \DependencyTreeLink{0.1729131}{0.7292879}{0.1465050}{0.0460000}{thm:noninvariant-decay}
    \DependencyTreeLink{0.3057399}{0.5892646}{0.1046070}{0.0460000}{lem:averaging-parabolic-holder}
    \DependencyTreeLink{0.4243360}{0.7292879}{0.1210017}{0.0460000}{lem:stage-stability}
    \DependencyTreeLink{0.7184534}{0.7292879}{0.1100719}{0.0460000}{lem:retained-fibration-persistence}
    \DependencyTreeLink{0.6899907}{0.5892646}{0.1228235}{0.0460000}{lem:pointwise-fixed-fiber}
    \DependencyTreeLink{0.8892339}{0.5892646}{0.1392182}{0.0460000}{input:bamler-fixed-threshold}
    \DependencyTreeLink{0.0685483}{0.5892646}{0.1319315}{0.0460000}{prop:uniform-translation-difference}
    \DependencyTreeLink{0.1871446}{0.5892646}{0.0781938}{0.0460000}{lem:small-orbit-averaging}
    \DependencyTreeLink{0.4338234}{0.5892646}{0.1046070}{0.0460000}{prop:positive-time-regular-restart}
    \DependencyTreeLink{0.5761388}{0.5892646}{0.0900342}{0.0460000}{lem:high-order-averaged-slice}
    \DependencyTreeLink{0.3816415}{0.4550759}{0.1383076}{0.0460000}{prop:uniform-lifted-rdt-stability}
    \DependencyTreeLink{0.5429313}{0.4550759}{0.1073394}{0.0460000}{lem:slice-to-parabolic-geometry}
    \DependencyTreeLink{0.5097250}{0.3500584}{0.0891236}{0.0588821}{lem:fixed-quotient-rdt-continuation}
    \DependencyTreeLink{0.1064991}{0.3558930}{0.0900342}{0.0460000}{lem:local-rdt-one-derivative-gain}
    \DependencyTreeLink{0.2298391}{0.3500584}{0.0973209}{0.0460000}{lem:universal-cover-inj}
    \DependencyTreeLink{0.6520400}{0.3500584}{0.1027858}{0.0460000}{lem:averaging-invariant-holder}
    \DependencyTreeLink{0.2772769}{0.2158696}{0.1128043}{0.0460000}{prop:uniform-lifted-linear-schauder}
    \DependencyTreeLink{0.4243356}{0.2158696}{0.1064288}{0.0460000}{lem:rdt-difference-system}
    \DependencyTreeLink{0.5666510}{0.2158696}{0.1191805}{0.0460000}{lem:intrinsic-fixed-chart-equivalence}
    \DependencyTreeLink{0.2298387}{0.0676750}{0.1118937}{0.0460000}{lem:uniform-parabolic-atlas}
    \DependencyTreeLink{0.3626659}{0.0676750}{0.1027858}{0.0588821}{lem:finite-order-cauchy-dirichlet}
    \DependencyTreeLink{0.6662711}{0.1108521}{0.1337533}{0.0460000}{lem:base-length-controls-diameter}
    \DependencyTreeLink{0.7564045}{0.2858813}{0.1000534}{0.0460000}{lem:isotopic-fibers-force-fixed}
    \DependencyTreeLink{0.8560222}{0.4142354}{0.1100719}{0.0588821}{lem:different-fibrations-small-diameter}
    \DependencyTreeLink{0.9319284}{0.1108521}{0.1310209}{0.0460000}{input:bamler-typeIII}
  \end{tikzpicture}%
  \endgroup
}

\newcommand{\DependencyTreeII}{%
  \begingroup
  \sbox{\dependencytreebox}{\includegraphics[width=\textwidth,height=.88\textheight,keepaspectratio]{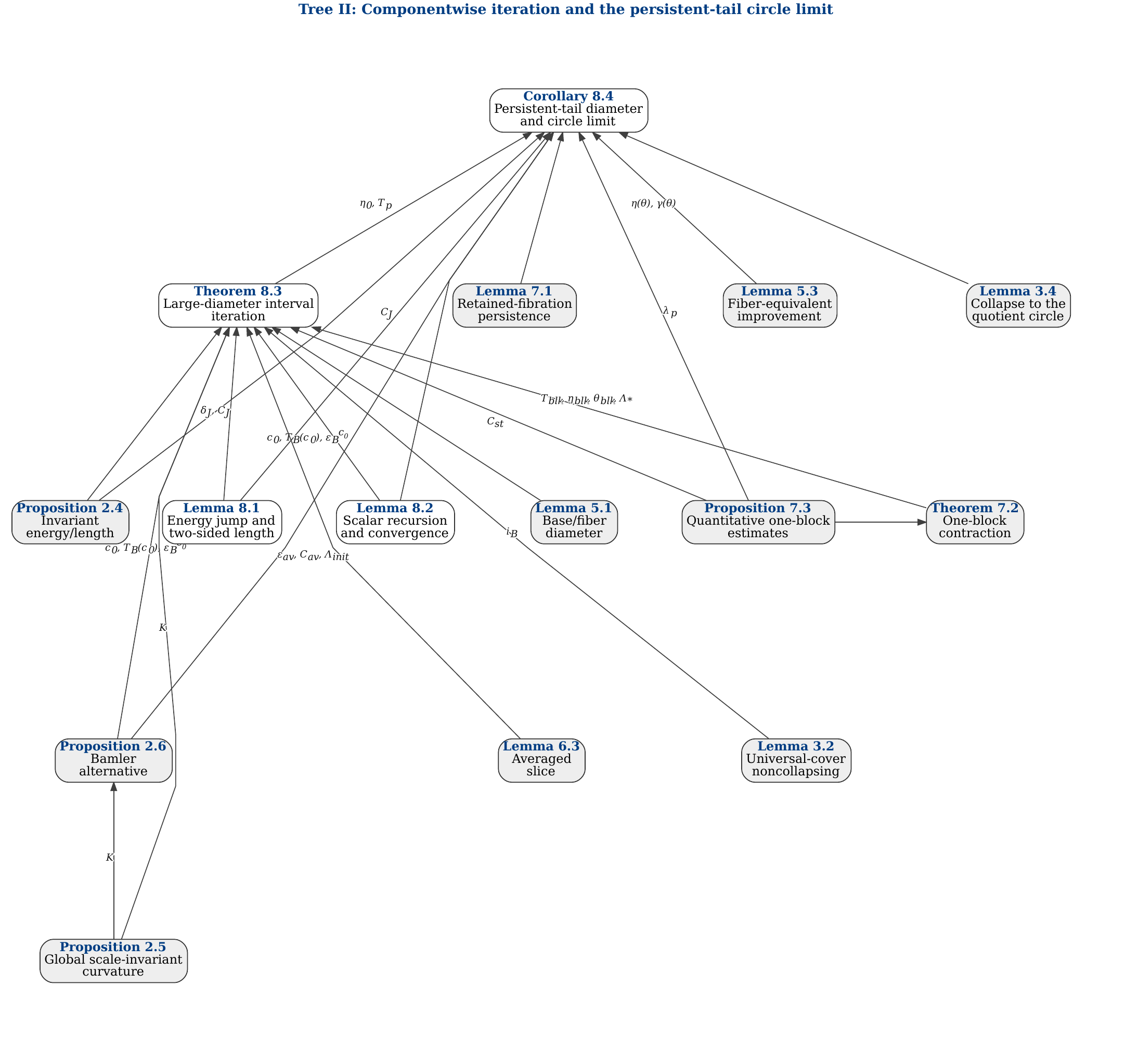}}%
  \setcounter{dependencylink}{0}%
  \begin{tikzpicture}[baseline=(depimage.base)]
    \node[inner sep=0pt,outer sep=0pt] (depimage) {\usebox{\dependencytreebox}};
    \DependencyTreeLink{0.5000000}{0.8961301}{0.1342354}{0.0382428}{cor:persistent-large-tail-diameter}
    \DependencyTreeLink{0.2095239}{0.7128307}{0.1351495}{0.0382428}{thm:large-component-iteration}
    \DependencyTreeLink{0.4523811}{0.7128307}{0.1049781}{0.0382428}{lem:retained-fibration-persistence}
    \DependencyTreeLink{0.0619046}{0.5091652}{0.0994923}{0.0382428}{input:invariant-energy-length}
    \DependencyTreeLink{0.1952379}{0.5091652}{0.1013211}{0.0382428}{lem:transition-jump}
    \DependencyTreeLink{0.3476189}{0.5091652}{0.1004071}{0.0382428}{lem:scalar-quotient-recursion}
    \DependencyTreeLink{0.5047621}{0.5091652}{0.0738926}{0.0382428}{lem:base-length-controls-diameter}
    \DependencyTreeLink{0.6857139}{0.7128307}{0.0967495}{0.0382428}{lem:isotopic-fibers-force-fixed}
    \DependencyTreeLink{0.8952404}{0.7128307}{0.0885209}{0.0382428}{lem:elementary-circle-collapse}
    \DependencyTreeLink{0.0999999}{0.2851321}{0.0994923}{0.0382428}{input:bamler-fixed-threshold}
    \DependencyTreeLink{0.8571404}{0.5091652}{0.0830352}{0.0382428}{thm:one-block-contraction}
    \DependencyTreeLink{0.6666668}{0.5091652}{0.1296638}{0.0382428}{prop:one-block-quantitative}
    \DependencyTreeLink{0.4761907}{0.2851321}{0.0738926}{0.0382428}{lem:high-order-averaged-slice}
    \DependencyTreeLink{0.7000000}{0.2851321}{0.0930926}{0.0382428}{lem:universal-cover-inj}
    \DependencyTreeLink{0.1000001}{0.0967439}{0.1250928}{0.0382428}{input:bamler-typeIII}
  \end{tikzpicture}%
  \endgroup
}

\newcommand{\DependencyTreeIII}{%
  \begingroup
  \sbox{\dependencytreebox}{\includegraphics[width=\textwidth,height=.88\textheight,keepaspectratio]{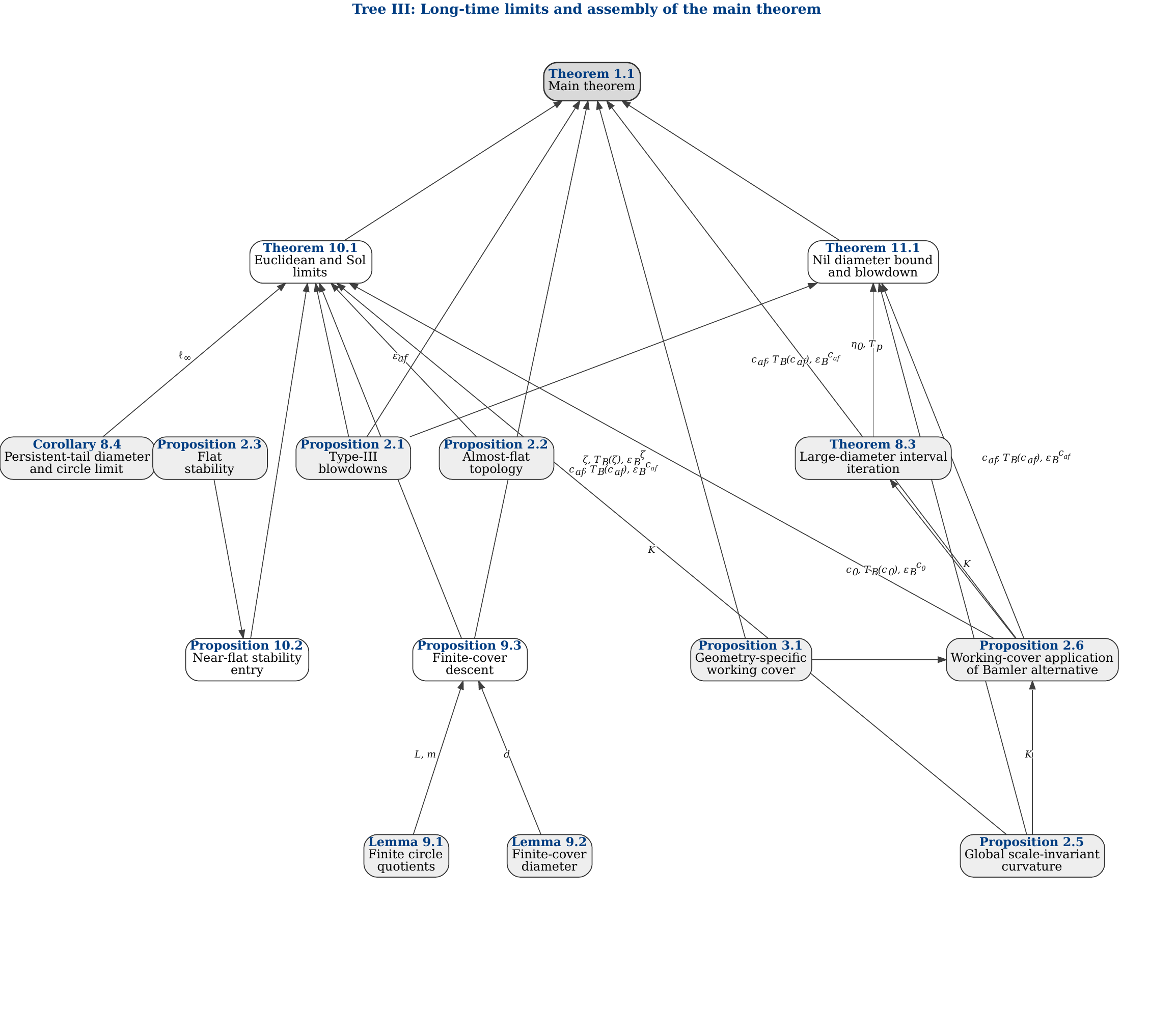}}%
  \setcounter{dependencylink}{0}%
  \begin{tikzpicture}[baseline=(depimage.base)]
    \node[inner sep=0pt,outer sep=0pt] (depimage) {\usebox{\dependencytreebox}};
    \DependencyTreeLink{0.5036098}{0.9191176}{0.0798843}{0.0359496}{thm:main}
    \DependencyTreeLink{0.2644403}{0.7405462}{0.1003497}{0.0394479}{thm:euclidean-sol-long-time-limits}
    \DependencyTreeLink{0.7427819}{0.7405462}{0.1072806}{0.0394479}{thm:nil-diameter-blowdown}
    \DependencyTreeLink{0.0658845}{0.5462182}{0.1272087}{0.0394479}{cor:persistent-large-tail-diameter}
    \DependencyTreeLink{0.1787000}{0.5462182}{0.0942843}{0.0394479}{input:flat-stability}
    \DependencyTreeLink{0.3005412}{0.5462182}{0.0942843}{0.0394479}{input:lott-blowdown}
    \DependencyTreeLink{0.4223823}{0.5462182}{0.0942843}{0.0394479}{input:almost-flat-topology}
    \DependencyTreeLink{0.7427799}{0.5462182}{0.1280749}{0.0394479}{thm:large-component-iteration}
    \DependencyTreeLink{0.2102890}{0.3466386}{0.1012159}{0.0394479}{prop:near-flat-stability-entry}
    \DependencyTreeLink{0.3998194}{0.3466386}{0.0942843}{0.0394479}{prop:finite-cover-descent}
    \DependencyTreeLink{0.6389888}{0.3466386}{0.0994828}{0.0394479}{prop:finite-torus-bundle-covers}
    \DependencyTreeLink{0.8781610}{0.3466386}{0.1410713}{0.0394479}{input:bamler-fixed-threshold}
    \DependencyTreeLink{0.3456681}{0.1523083}{0.0700245}{0.0394479}{lem:finite-quotient-circle-collapse}
    \DependencyTreeLink{0.4675093}{0.1523083}{0.0700245}{0.0394479}{lem:finite-cover-diameter}
    \DependencyTreeLink{0.8781590}{0.1523083}{0.1185446}{0.0394479}{input:bamler-typeIII}
  \end{tikzpicture}%
  \endgroup
}

\begin{document}

\title[]{Under Ricci flow, a 3-torus goes flat}

\author{John Lott}
\address{Department of Mathematics\\
University of California, Berkeley\\
Berkeley, CA  94720-3840\\
USA} \email{lott@berkeley.edu}

\date{September 13, 2026}

\begin{abstract}
We determine the long-time behavior of the canonical Ricci flow on a closed
three-manifold of Thurston type \(\RR^3\), \(\Nil\), or \(\Sol\), starting from
an arbitrary initial metric.  In the Euclidean case, the metric \(g(t)\)
converges exponentially fast to a flat metric.  The Gromov--Hausdorff limit of
\((M,t^{-1}g(t))\) is a point in the Euclidean and Nil cases, and a metric circle or a compact interval of positive length in the Sol case.  In all three cases, the blowdown flows lifted to the
universal cover, \(s^{-1}\widetilde g(s\tau)\), converge, in the pointed
Cheeger--Hamilton sense, to an explicit homogeneous expanding Ricci soliton
solution.
\end{abstract}

\maketitle
\vspace{-1.5ex}
\tableofcontents

\section{Introduction}

\subsection{Main result and geometric setting}

Given an initial Riemannian metric on a closed three-manifold, there is a
canonical Ricci flow through singularities
\cite{Bamler-Kleiner,Kleiner-Lott}.  One can ask whether its long-time behavior
realizes Thurston's geometric decomposition.  This is still not completely
understood.

After some finite time, the canonical Ricci flow is known to be either empty or
smooth; if the original manifold is aspherical, then the manifold that remains
is diffeomorphic to it \cite{Bamler,Bamler-Kleiner}.  Hence, to study the
long-time behavior in the cases considered here, it is enough to consider
immortal Ricci flows, meaning smooth solutions defined for all
\(t\in[0,\infty)\).  The title emphasizes the simplest consequence of the main
theorem below: every immortal Ricci flow on a three-dimensional torus converges
to a flat metric as \(t\to\infty\).

\begin{theorem}[Main theorem]\label{thm:main}
Let \(M\) be a closed connected three-manifold which admits a locally
homogeneous metric modeled on one of \(\RR^3\), \(\Nil\), or \(\Sol\), and let
\(g(t)\), \(0\leq t<\infty\), be any immortal Ricci flow on \(M\).  In all
three cases, for every basepoint \(\widetilde m\in\widetilde M\) in the
universal cover, the lifted blowdown flows
\((\widetilde M,\widetilde m,s^{-1}\widetilde g(s\tau))\), \(\tau>0\),
converge as \(s\to\infty\) in the pointed Cheeger--Hamilton sense to a
homogeneous expanding Ricci soliton solution.  The compact
and lifted limits are as follows.
\begin{enumerate}
\item If \(M\) has type \(\RR^3\), then, for every \(m\geq0\), \(g(t)\)
converges exponentially fast in the \(C^m\)-norm of any fixed smooth
background metric, as \(t\to\infty\), to a flat metric on \(M\).
The Gromov--Hausdorff limit of
\((M,t^{-1}g(t))\), as \(t\to\infty\), is a point.
The limit of the lifted blowdown flows is the flat expanding soliton solution
\[
        (\RR^3,0,g_{\mathrm{Euc}}(\tau)),
        \qquad
        g_{\mathrm{Euc}}(\tau)=dx^2+dy^2+dz^2,
        \quad \tau>0.
\]
\item If \(M\) has type \(\Nil\), then the Gromov--Hausdorff limit of
\((M,t^{-1}g(t))\), as \(t\to\infty\), is a point.  The limit of the lifted blowdown flows is the Nil
expanding soliton solution
\[
        \left(\RR^3,0,
        \frac{1}{3\tau^{1/3}}
        \left(dx+\frac12 y\,dz-\frac12 z\,dy\right)^2
        +\tau^{1/3}(dy^2+dz^2)
        \right).
\]
\item If \(M\) has type \(\Sol\), then the Gromov--Hausdorff limit of
\((M,t^{-1}g(t))\), as \(t\to\infty\), is a metric circle or a compact
interval of positive length.  The limit of the lifted blowdown flows is the Sol expanding soliton solution
\[
        \left(\RR^3,0,e^{-2z}dx^2+e^{2z}dy^2+4\tau\,dz^2\right).
\]
\end{enumerate}
\end{theorem}

Throughout the universal-cover conclusions, convergence as \(s\to\infty\) is
understood in the sense used in \cite{Lott10}: for every sequence
\(s_j\rightarrow\infty\), the flows
\((\widetilde{M},\widetilde{m},s_j^{-1}\widetilde{g}(s_j\tau))\)
converge as \(j\rightarrow\infty\), in the pointed Cheeger--Hamilton sense,
to the displayed pointed Ricci flow.

\subsection{Previous work and the missing diameter estimate}

The analogous result for Ricci flow on \(\TT^2\) was proved by Hamilton
\cite{Hamilton}.  In that case, since Ricci flow preserves the conformal
structure and the area, one knows exactly what flat metric one gets in the
limit.  In contrast, in part~(1) of Theorem~\ref{thm:main}, there is no
\emph{a priori} way to know what the limiting flat metric will be.

Each manifold \(M\) in Theorem~\ref{thm:main} has a finite regular
orientable cover that fibers over a circle with torus fiber and whose
monodromy lies in \(\SL(2,\ZZ)\).  The cover may be chosen to be \(\TT^3\)
with identity monodromy in the Euclidean case, a \(\Nil\) torus bundle with
nontrivial unipotent monodromy in the Nil case, and a \(\Sol\) torus bundle
with hyperbolic monodromy having positive eigenvalues in the Sol case.  We may therefore
reduce the analytic part of the proof to these three cases and then
descend the conclusions
through the fixed finite cover.

There are three main geometric inputs.  First, the compact and universal-cover
blowdown conclusions used here, together with analogous results for the other
aspherical Thurston geometries, were proved by the author \cite{Lott10} under
the Type-III curvature estimate
\(|\sec_{g(t)}|=O(t^{-1})\) and the scale-invariant diameter bound
\(\diam(M,g(t))=O(\sqrt t)\).  Second, the needed estimates for flows with
local \(\TT^2\)-symmetry were established by the author and
Nata\v{s}a \v{S}e\v{s}um \cite{LottSesum}.  Third, Bamler proved the Type-III curvature estimate and
a large-time description by increasingly accurate locally
\(\TT^2\)-invariant time-slice models
\cite{Bamler,BamlerA,BamlerB,BamlerC,BamlerD}.

More precisely, fix \(K\) so that \(|\sec_{g(t)}|\leq K/t\), choose a
geometric cutoff \(c_{\mathrm{af}}=c_{\mathrm{af}}(K)>0\), and divide each
sufficiently late time slice into two cases:
\[
\begin{array}{ll}
\textit{almost-flat case:}&\diam(M,g(t))<c_{\mathrm{af}}\sqrt t,\\[2pt]
\textit{torus-bundle case:}&\diam(M,g(t))\geq c_{\mathrm{af}}\sqrt t.
\end{array}
\]
The cutoff \(c_{\mathrm{af}}\) is chosen only so that the first inequality makes
\[
        \bigl(\sup_M|\sec_{g(t)}|\bigr)\diam(M,g(t))^2
\]
smaller than the three-dimensional almost-flat threshold.  In the second
case, \(t^{-1}g(t)\) is increasingly close in high \(C^k\)-norms to an invariant metric for a possibly time-dependent torus-bundle
structure, and the torus fibers have diameter \(o(1)\) with respect to
\(t^{-1}g(t)\).

The principal missing geometric estimate is
\[
        \diam(M,g(t))=O(\sqrt t).
\]
The almost-flat case already gives this estimate.  If the flow had an exact local
\(\TT^2\)-symmetry in the other case, then the diameter bound would follow
from \cite{LottSesum}.  The difficulty is that the error in Bamler's
approximation may tend to zero arbitrarily slowly.  When the comparison is
restarted on successive time intervals, these errors can accumulate and swamp
the desired diameter bound.  Moreover, Bamler's fibration may change with time.
Thus one cannot simply start a locally \(\TT^2\)-invariant Ricci flow at each time and
compare it with the original flow for an indefinitely long interval.

\subsection{The blockwise argument}

We use three terms consistently.  A \emph{block} is a physical-time interval
\([S,100S]\) on which \(\diam(M,g(t))\geq\delta_0\sqrt t\) for a fixed
threshold \(\delta_0>0\); parabolic rescaling by \(S^{-1}\) identifies it with
the normalized interval \([1,100]\).  A \emph{stage} is the Ricci--DeTurck
comparison on this normalized interval between the rescaled original flow and
an invariant Ricci-flow background.  A \emph{restart} is the
factor-\(100\) endpoint rescaling and averaging that produces the initial
metrics for the next stage.

We first prove a finite-time decay theorem for the non-invariant part of a
metric.  Let \(h(t)\), \(1\leq t\leq100\), be an invariant Ricci flow with bounded curvature and a quotient circle of
uniformly positive length.  If \(g_{\RDT}\) is a Ricci--DeTurck flow relative
to \(h\), sufficiently close to \(h\), and the local \(\TT^2\)-orbits are
sufficiently small, then we show that the size of the non-invariant part of
\(g_{\RDT}\) on
\([10,100]\) is an arbitrarily small multiple of its size at time \(1\).  
Intuitively, at
the linearized level, the non-invariant fluctuations correspond to
torus modes whose Lichnerowicz eigenvalues become very negative if the torus
fibers are small.
These negative eigenvalues then cause a fast decay for the non-invariant part of
the metric.  The actual proof uses an
averaging inequality on short torus orbits and parabolic Schauder estimates.
The Schauder estimates are all carried out on the universal cover, which
is uniformly noncollapsed under the geometric hypotheses. In this way,
the decay estimates are uniform with respect to collapsing fibers.
Using the Ricci--DeTurck gauge is essential for the Schauder estimates.

We iterate this estimate on intervals \([T_k,T_{k+1}]\), where
\(T_{k+1}=100T_k\), 
using one fixed fibration. In order to apply the finite-time
decay theorem, we need to know that the torus fibers remain uniformly small
on blocks. The argument for this uses the time-dependent fibration supplied
by Bamler.

This is incorporated into the central block-contraction theorem, which proves the
conceptual implication
\[
        \text{admissible block start}
        \quad\longmapsto\quad
        \text{admissible restart},
        \qquad
        b_{\mathrm{new}}\leq\tfrac12 b_{\mathrm{old}},
\]
where \(b\) is the non-invariant error.
To apply it at the first stage, we choose a sufficiently late time and a
fibration supplied by Bamler.  We parabolically rescale the original flow,
average the initial metric with respect to the chosen fibration, evolve the
average by Ricci flow, and put the original rescaled flow in Ricci--DeTurck
gauge relative to that invariant solution.  At the end of the stage, we
rescale, average again, and continue with the same fixed fibration.  In physical time this produces a continuous,
piecewise smooth pullback \(g'(t)\) of \(g(t)\) and a piecewise Ricci-flow
comparison family \(h(t)\), which may jump at switching times; see
Figure~\ref{fig:three-metric-families}.

For the invariant comparison flow \(h(t)\), the estimates of
\cite{LottSesum} control the increase of the quotient-circle length on each
smooth stage.  The jumps of \(h\) between stages produce a summable sequence
of length errors.  A scalar recursion then shows that the length of the
quotient circle is \(O(\sqrt t)\).  On a persistent large-diameter tail, the
quantitative lower bound for the quotient length forces the normalized
stage-start lengths to converge and the energies to tend to the self-similar
value \(2\).  A two-sided estimate at each restart and the smooth-stage
inequalities propagate this convergence throughout every block; comparison
with Bamler's shrinking fibers then gives a full metric-circle limit on the
working cover.  The fiber-diameter bounds yield the same diameter estimate for
\(h(t)\), and hence for \(g(t)\).

\subsection{The three geometries and the remaining cases}

The block construction and quotient-length recursion are applied separately
on each maximal interval of sufficiently late times for which the normalized
diameter \(t^{-1/2}\diam(M,g(t))\) stays above the fixed threshold; the
argument also controls a possible terminal partial block.  The same block and
quotient-length machinery is used in the Euclidean, Nil, and Sol cases.  The
Thurston type
matters only in determining whether the normalized diameter eventually stays
above the threshold or sometimes drops below it.  For a \(\Sol\) manifold, the
latter alternative would produce an almost-flat slice and is topologically
impossible.  Hence the iteration applies for all sufficiently large times and
gives both the required diameter estimate and the full metric-circle limit on
the working cover.  The Sol universal-cover blowdown comes from
\cite{Lott10}; finite-cover descent then gives part~(3) of
Theorem~\ref{thm:main}.

In the Euclidean case, if one sufficiently late time slice is
sufficiently almost flat, then the stability results of \cite{GIK,LottSesum}
give exponential
convergence to a flat metric.  Otherwise, the torus-bundle case persists on a
tail.  The resulting diameter bound, combined with \cite{Lott10}, then forces
the flow into the almost-flat case.  This proves part~(1).

For a \(\Nil\) manifold, we apply the same iteration separately on each
connected component of
\[
        \{t:t^{-1/2}\diam(M,g(t))\geq c_{\mathrm{af}}\}.
\]
The complementary times are already almost flat and satisfy the required
diameter bound, while Type-III metric distortion controls the partial blocks
at the ends of each component.  Once \(\diam(M,g(t))=O(\sqrt t)\) is known,
\cite{Lott10} proves part~(2).

In a broader context, one may speculate that an immortal Ricci flow on a
compact three-manifold \(M\) has diameter growth \(O(\sqrt t)\) if and only if
\(M\) admits a locally homogeneous Riemannian metric
\cite[Remarks~1.4--1.5]{Lott10}.  The ``only if'' implication follows by
combining Bamler's Type-III estimate with
\cite[Proposition~3.5]{Lott10}.  The ``if'' implication is known for \(H^3\)
from Perelman's work, and for \(\RR^3\), \(\Nil\), and \(\Sol\) from the
present paper.  The remaining geometries are \(H^2\times\RR\) and
\(\widetilde{\SL(2,\RR)}\).  The methods here do not extend
directly to these cases.  Bamler's results do not provide an approximation by an
\(S^1\)-invariant metric on a circle bundle over a higher-genus surface
\cite[Theorem~1.4(e)]{Bamler}.  Moreover, even for an \(S^1\)-invariant Ricci
flow on such a circle bundle, the long-time behavior is not known in general.
Section~2 of \cite{LottSesum} treats the special case of a warped-product
metric, but curvature of the circle bundle introduces additional
complications.

\subsection{Organization and conventions}

The structure of the paper is the following.
Section~\ref{sec:external-inputs} states the long-time inputs from earlier
work and the fixed-threshold formulation used in the block argument.
Section~\ref{sec:collapse-inputs} selects the working torus-bundle cover and
develops the local bundle geometry.  Section~\ref{sec:finite-time-decay} proves
finite-time decay of the non-invariant part.
Section~\ref{sec:fibration-comparison} compares torus fibrations to bound the
diameters of fibers.  Section~\ref{sec:stage-restart} constructs one Ricci--DeTurck
stage and its regular restart.  Section~\ref{sec:deturck-iteration} proves the
one-block contraction.  Section~\ref{sec:quotient-length} combines that
contraction with the quotient-circle energy and length estimates to obtain a
single iteration theorem for any connected large-diameter interval.
Section~\ref{sec:finite-cover-descent} collects the finite-cover descent facts
used in the final arguments.  Sections~\ref{sec:euclidean-sol} and
\ref{sec:nil-bootstrap} prove the Euclidean, Sol, and Nil conclusions on the
working cover, and Section~\ref{sec:main-theorem-proof} descends them and
completes the proof of Theorem~\ref{thm:main}.
Appendix~\ref{app:parabolic-estimates} contains the uniform lifted parabolic
estimates, and Appendix~\ref{app:source-guide} records the detailed
guide from the external formulations in Section~\ref{sec:external-inputs}
to the cited sources.

For the reader's convenience, Appendix~\ref{app:dependencies} gives
dependency trees of the principal statements in the paper.  Each arrow
runs from a dependency to a result that uses it, and selected edge labels record
the main parameters or quantities carried along that implication.
Appendix~\ref{app:symbols} lists the persistent parameters and notation
alphabetically, first by Greek letter and then by Roman letter.

Unnamed constants denoted by \(C\) may change from line to line.  Named
constants retain the values assigned to them.  All constants depend only on
the parameters shown in the relevant statements, not on the particular
collapsed manifold once the collapse parameter is sufficiently small.

\subsection*{AI Acknowledgment.}
The author developed the proof structure, intermediate results, and proof
sketches.  ChatGPT-5.6 Sol was used to elaborate detailed arguments from this
material and to assist in preparing the figures and the dependency and notation
appendices.  The author assumes responsibility for all mathematical claims and
source attributions in the manuscript.

\section{Long-time inputs from earlier work}\label{sec:external-inputs}
The proof uses several long-time results from the literature as external
inputs.  This section states them in forms that are uniform under passage to
covers and compatible with the later rescaling-and-restart argument.  It also
fixes the thresholds and quantitative formulations that will be used without
further alteration in the later analytic and geometric constructions.  A
detailed guide from these formulations to the cited statements is given in
Appendix~\ref{app:source-guide}.

\subsection{Type-III blowdowns}

The following proposition is the part of \cite[Theorem~1.2]{Lott10} used below.
\begin{proposition}[Blowdowns under Type-III curvature and diameter control]
\label{input:lott-blowdown}
Let \(M\) be a closed connected orientable three-manifold modeled on
\(\RR^3\), \(\Nil\), or \(\Sol\), and let \(g(t)\) be an immortal Ricci
flow.  Suppose that, for all sufficiently large \(t\),
\begin{equation}\label{eq:lott-input-hypotheses}
        |\Rm_{g(t)}|\leq C/t,
        \qquad
        \diam(M,g(t))\leq C\sqrt t.
\end{equation}
Then the compact and universal-cover blowdown conclusions of
\cite[Theorem~1.2]{Lott10} hold in the following form.  The compact normalized
metrics converge to a point in the Euclidean and Nil cases.  In the Sol case,
every Gromov--Hausdorff subsequential limit is a metric circle or a compact
interval.  The lifted blowdown flows
converge in the pointed Cheeger--Hamilton sense to the corresponding
homogeneous expanding soliton solutions.  In the Euclidean case, the limit is
the static Euclidean flow \((\RR^3,g_{\mathrm{Euc}}(\tau))\), where
\(g_{\mathrm{Euc}}(\tau)=dx^2+dy^2+dz^2\) for \(\tau>0\).
\end{proposition}

The new geometric estimate needed in this paper is the second inequality
in \eqref{eq:lott-input-hypotheses}.

\subsection{Almost-flatness and stability near flat metrics}

\phantomsection\label{par:almost-flat-threshold}
Fix a dimensional constant \(\eps_{\mathrm{af}}>0\) below the
three-dimensional Gromov--Ruh threshold.

We use the following almost-flat theorem of Gromov and Ruh
\cite{GromovAlmostFlat,Ruh}.
\begin{proposition}[Almost-flat topological theorem]
\label{input:almost-flat-topology}
If a metric \(g\) on a closed three-manifold satisfies
\begin{equation}
        \bigl(\sup_M|\sec_g|\bigr)\diam(M,g)^2<\eps_{\mathrm{af}},
\end{equation}
then the manifold is an infranilmanifold; in particular, its fundamental group
is virtually nilpotent.
\end{proposition}

The following uniform version of the flat stability theorem of
Guenther--Isenberg--Knopf \cite{GIK} is the external analytic input used in
the Euclidean argument.
\begin{proposition}[Stability near a compact family of flat metrics]
\label{input:flat-stability}
Let \(M\) be a fixed torus and let \(\mathcal F\) be a compact set of
flat metrics on \(M\) (not modulo diffeomorphisms).  For a fixed
\(\rho\in(0,1)\), there is a uniform
\(h^{2+\rho}\)-neighborhood of \(\mathcal F\), measured using one fixed smooth
reference metric on \(M\), such that the Ricci flow starting from any
smooth Riemannian metric in this neighborhood exists for all future time and,
without applying a time-dependent pullback, converges exponentially as a tensor
field on \(M\) to a smooth flat metric.  Here \(h^{2+\rho}\) is the little
H\"older space, namely the closure of smooth metrics in the
\(C^{2,\rho}\)-norm.  The decay rate and the constants in the fixed-background
stability estimates may be chosen uniformly on \(\mathcal F\).  Convergence
holds in every fixed \(C^m\)-norm, with higher-order constants that may also
depend on higher norms of the smooth initial metric.
\end{proposition}

\begin{proof}
For one flat background, \cite[Theorem~3.7]{GIK} gives exponential stability
for the Ricci flow.  In its proof, the corresponding Ricci--DeTurck
convergence is established first, and \cite[Proposition~3.6]{GIK} is then used
to transfer this convergence to the ungauged Ricci flow.  Fix a smooth
reference metric on
the underlying torus and use its little H\"older norm to metrize the
space of \(h^{2+\rho}\)-metrics.  For every \(h\in\mathcal F\), choose an open
stability neighborhood \(U_h\) and the exponential estimate supplied by
\cite{GIK}.  There are \(h_1,\ldots,h_J\in\mathcal F\) such that
\(\mathcal F\subset U:=\bigcup_{j=1}^J U_{h_j}\).  Since \(\mathcal F\) is
compact and \(U\) is open, there is \(\delta>0\) such that the
\(\delta\)-neighborhood of \(\mathcal F\), in this one fixed
\(h^{2+\rho}\)-metric, is contained in \(U\).  Thus every initial metric in
that uniform neighborhood lies in at least one of the finitely many
\(U_{h_j}\).  The minimum of the corresponding exponential rates and the maximum of the
constants, together with the equivalence of the finitely many little H\"older
norms occurring in these estimates, give uniform stability data.  For smooth
initial data, positive-time parabolic estimates bootstrap the exponentially
decaying Ricci--DeTurck perturbation in every fixed \(C^m\)-norm.  The
resulting estimates for the DeTurck vector field give convergence of the
associated diffeomorphisms in the corresponding orders, transferring the
exponential convergence to the ungauged flow.  The decay rate remains uniform
on \(\mathcal F\), while the higher-order constants may also depend on higher
norms of the smooth initial metric.
\end{proof}

The diameter-normalized near-flat compactness statement that places a metric
in this fixed-coordinate stability neighborhood is specific to the Euclidean
argument.  It is formulated and proved in
Subsection~\ref{subsec:euclidean-case}, immediately before it is used.

\subsection{Estimates for invariant torus-bundle flows}

The next proposition states the two invariant-flow estimates from
\cite[Lemmas~3.3 and~3.7]{LottSesum} used in the iteration.
\begin{proposition}[Invariant energy and quotient-length inequalities]
\label{input:invariant-energy-length}
Let \(h(t)\) be an invariant Ricci flow on a torus bundle,
written in the notation of \cite{LottSesum} and
Subsection~\ref{subsec:one-dimensional-estimates}, and set
\[
        E(t)=\max_{S^1}h^{yy}\Tr\bigl((G^{-1}G_y)^2\bigr),
        \qquad
        L(t)=L(S^1,h_{yy}(t)).
\]
Then \cite[Lemmas~3.3 and~3.7]{LottSesum} give
\begin{equation}\label{eq:external-energy-length}
        \frac{d}{dt}\log L(t)\leq\frac14E(t).
\end{equation}
If \(t_1\geq t_0\) and \(E(t_0)>0\), they also give
\begin{equation}
        \frac1{E(t_1)}-\frac1{E(t_0)}\geq\frac12(t_1-t_0).
\end{equation}
If \(E(t_0)=0\), the maximum-principle inequality underlying the energy
estimate gives \(E(t)=0\) for every \(t\geq t_0\); in that case the
reciprocal inequality is not invoked and \eqref{eq:external-energy-length}
has zero right-hand side.  These inequalities are used only on the smooth invariant stages constructed in
Section~\ref{sec:deturck-iteration}.
\end{proposition}

\subsection{Bamler's Type-III estimate and fixed-threshold alternative}
\label{subsec:bamler-input}

We use Bamler's Type-III estimate in the following form
\cite[Corollary~1.2]{Bamler}.
\begin{proposition}[Global scale-invariant curvature bound]
\label{input:bamler-typeIII}
Let \(M\) be a closed connected three-manifold of type
\(\RR^3\), \(\Nil\), or \(\Sol\), and let \(g(t)\) be an immortal Ricci
flow.  There is a constant \(K<\infty\) such that, for every \(t>0\),
\begin{equation}\label{eq:bamler-typeIII}
        |\sec_{g(t)}|\leq K/t.
\end{equation}
\end{proposition}

The next proposition is a fixed-threshold, tail-uniform reformulation of the
topology-specific alternatives in \cite[Theorem~1.4(b)--(d)]{Bamler}.  The
detailed construction of the invariant comparison metrics, together with the
derivative and fiber estimates used here, is given in
\cite[Section~4.4, especially Proposition~4.9]{BamlerD}, using
\cite{BamlerA,BamlerB,BamlerC}.  The initial normalization in the cited
theorem is removed by a constant parabolic rescaling, as explained in
Appendix~\ref{app:source-guide}; no normalization of \(g(0)\) is assumed below.
\begin{proposition}[Fixed-threshold almost-flat/torus-bundle alternative]
\label{input:bamler-fixed-threshold}
Let \(M\) be \(\TT^3\), a compact \(\Nil\) torus bundle with nontrivial
unipotent monodromy, or a compact \(\Sol\) torus bundle with monodromy
\(H\in\SL(2,\ZZ)\) satisfying \(\Tr H>2\), and let \(g(t)\) be an
immortal Ricci flow on \(M\).  Let \(K\) be
a Type-III constant supplied by Proposition~\ref{input:bamler-typeIII}.
Fix any \(c_0\in(0,1)\) such that
\begin{equation}\label{eq:c0-almost-flat-choice}
        Kc_0^2<\eps_{\mathrm{af}}.
\end{equation}
The parameter \(c_0\) is used only as the diameter threshold separating the two
geometric alternatives below; it is not an analytic smallness parameter and will not
be decreased later to meet compactness or stability tolerances.

There exist a finite time \(T_B(c_0)<\infty\) and a nonincreasing function
\[
        \eps_B^{c_0}:[T_B(c_0),\infty)\longrightarrow(0,\tfrac12],
        \qquad \eps_B^{c_0}(t)\longrightarrow0,
\]
with
\begin{equation}
        \eps_B^{c_0}(t)<c_0,
        \qquad t\geq T_B(c_0),
\end{equation}
such that every \(t\geq T_B(c_0)\) belongs to exactly one of the following
diameter-defined alternatives.
\begin{enumerate}
\item\label{item:bamler-small-slice}
If
\begin{equation}
        \diam(M,g(t))<c_0\sqrt t,
\end{equation}
then
\begin{equation}\label{eq:bamler-small-is-almost-flat}
 \bigl(\sup_M|\sec_{g(t)}|\bigr)\diam(M,g(t))^2
 <Kc_0^2<\eps_{\mathrm{af}}.
\end{equation}
Thus the slice is quantitatively almost flat.

\item\label{item:bamler-large-slice}
If
\begin{equation}\label{eq:bamler-large-diameter}
        \diam(M,g(t))\geq c_0\sqrt t,
\end{equation}
then there exist a torus fibration \(\pi_t^B:M\to S_t^1\) and an invariant metric \(g_t^B\), in the sense described in
Subsection~\ref{subsec:torus-bundle-conventions}, such that
\begin{equation}\label{eq:bamler-comparison-closeness}
        (1+\eps_B^{c_0}(t))^{-1}g_t^B\leq g(t)
        \leq(1+\eps_B^{c_0}(t))g_t^B.
\end{equation}
Moreover, for every integer \(m\) with
\(0\leq m<\bigl(\eps_B^{c_0}(t)\bigr)^{-1}\), the metric
\(t^{-1}g_t^B\) is \(\eps_B^{c_0}(t)\)-close to \(t^{-1}g(t)\) in
\(C^m\), relative to \(t^{-1}g(t)\), and every fiber has
\(g(t)\)-diameter at most \(\eps_B^{c_0}(t)\sqrt t\).
\end{enumerate}
Labels~(\ref{item:bamler-small-slice}) and~(\ref{item:bamler-large-slice})
refer only to the diameter inequalities.  A torus comparison
may also exist in the first alternative, but it is not used there.  In the
second alternative the fibration, quotient circle, and comparison metric may
depend on \(t\), with no asserted compatibility or uniqueness.
\end{proposition}

\medskip
\noindent\textbf{Fixed global application.}
For the global almost-flat/torus-bundle decomposition, fix once and for all a
cutoff \(c_{\mathrm{af}}=c_{\mathrm{af}}(K)\in(0,1)\) satisfying
\[
        Kc_{\mathrm{af}}^2<\eps_{\mathrm{af}},
\]
and write \(T_B(c_{\mathrm{af}})\) and
\(\eps_B^{c_{\mathrm{af}}}\) for the corresponding tail data.  Any
additional application of Proposition~\ref{input:bamler-fixed-threshold} is
labeled by its own cutoff and has its own tail data.  The Euclidean argument
will make one such additional application.

\medskip
\noindent\textbf{Working-cover application.}
All later applications of Proposition~\ref{input:bamler-fixed-threshold}
are made directly on the single working cover fixed after
Proposition~\ref{prop:finite-torus-bundle-covers}; no additional finite cover
is taken.

\section{Finite covers and torus-bundle geometry}\label{sec:collapse-inputs}
This section
establishes preliminary results about the topology, local torus
action, lifted noncollapsing, and averaging formalism.  
Its conclusions allow the later parabolic estimates to be
uniform despite collapse of the compact quotient.

The section is organized as follows.  The first subsection, \emph{Finite
torus-bundle covers and monodromy}, passes to a finite regular orientable
torus-bundle cover that, according to the geometry, is \(\TT^3\) with identity
monodromy, a \(\Nil\) bundle with nontrivial unipotent monodromy, or a \(\Sol\)
bundle with hyperbolic monodromy having positive eigenvalues.  The second subsection,
\emph{Torus-bundle conventions}, interprets the bundle as a twisted principal
\(\TT^2\)-bundle and fixes the local translation actions and the meaning of
local \(\TT^2\)-invariance.  The third subsection, \emph{Universal-cover
noncollapsing}, proves a collapse-independent injectivity-radius lower bound
on the universal cover from an invariant bilipschitz comparison, a uniform
volume lower bound, and the Cheeger--Gromov--Taylor estimate.  The final
subsection, \emph{Averaging and elementary collapse}, defines the global
fiberwise averaging operator and the non-invariant part, proves fixed-time
contraction of averaging in invariant H\"older norms, and quantifies collapse
to the metric quotient circle when the invariant fibers are short.

We use distinct notation for the three covers that occur in the proof.  A
fixed finite regular cover of the original manifold is denoted
\(\pi_{\mathrm{fin}}:M_{\mathrm{fin}}\to M\), with pulled-back metric
\(g_{\mathrm{fin}}\).  If \(M\) is a torus bundle, the infinite cyclic cover
obtained by unwrapping its base circle is denoted
\(\pi_{\mathrm{cyc}}:M_{\mathrm{cyc}}\to M\); a pulled-back tensor is written
with the subscript \({\mathrm{cyc}}\).  The universal cover is denoted
\(\widetilde M\), and universal-cover lifts carry tildes.  In particular, the
compact finite cover and the noncompact cyclic cover are never
represented by the same symbol.

\subsection{Finite torus-bundle covers and monodromy}\label{subsec:topological-preliminaries}

For a matrix \(H\in\SL(2,\ZZ)\), we use the standard terminology that
\(H\) is \emph{elliptic} if it has finite order, \emph{parabolic} if it has
infinite order and \(|\Tr H|=2\), and \emph{hyperbolic} if \(|\Tr H|>2\).
These are the mutually exclusive possibilities in \(\SL(2,\ZZ)\), and
replacing \(H\) by a positive power preserves its type.

\begin{proposition}[Finite torus-bundle covers of the three geometries]
\label{prop:finite-torus-bundle-covers}
Let \(M\) be a closed connected three-manifold modeled on \(\RR^3\),
\(\Nil\), or \(\Sol\).  Then \(M\) has a finite regular orientable cover
\(M_{\mathrm{fin}}\) which is the total space of a \(\TT^2\)-bundle over
\(S^1\), with monodromy \(H\in\SL(2,\ZZ)\).  The cover may be chosen so
that the following hold:
\begin{enumerate}
\item if \(M\) is modeled on \(\RR^3\), then
      \(M_{\mathrm{fin}}\cong\TT^3\) and \(H=I\);
\item if \(M\) is modeled on \(\Nil\), then \(H\) is nontrivial unipotent;
\item if \(M\) is modeled on \(\Sol\), then \(H\) is hyperbolic with
      positive eigenvalues, equivalently \(\Tr H>2\).
\end{enumerate}
In particular, the monodromy is elliptic, parabolic, or hyperbolic according
as the original geometry is \(\RR^3\), \(\Nil\), or \(\Sol\).
\end{proposition}

\begin{proof}
For a compact Euclidean three-manifold, the translation subgroup gives a
finite cover by \(\TT^3\).  A compact Nil manifold has a finite cover which
is a torus bundle over a circle with nontrivial unipotent monodromy, while a
compact Sol manifold has a cover of degree at most four which is a torus
bundle with hyperbolic monodromy.  These facts, together with the monodromy
classification, are recorded in \cite[Sections~4--5, especially
Theorems~4.16, 5.3, and~5.5]{Scott}.  Pulling back first by the degree-two
cover of the base changes the monodromy from \(H_0\) to \(H=H_0^2\) and
places it in \(\SL(2,\ZZ)\) without changing the associated geometry.
In the Sol case this also makes both eigenvalues positive and gives
\(\Tr H=(\Tr H_0)^2-2\det H_0>2\).  Finally,
we take the cover corresponding to the core of this subgroup in
\(\pi_1(M)\), meaning the intersection of all its conjugates, or equivalently
the largest normal subgroup of \(\pi_1(M)\) contained in it.  This is a finite
regular refinement.  To see directly that the torus-bundle structure and its
type persist under the refinement, write the bundle group as
\[
        \Gamma_H=\ZZ^2\rtimes_H\ZZ.
\]
For any finite-index subgroup \(\Gamma'\leq\Gamma_H\), the intersection
\(L=\Gamma'\cap\ZZ^2\) is a finite-index lattice in \(\ZZ^2\), and the image
of \(\Gamma'\) under the projection to \(\ZZ\) is \(m\ZZ\) for some
\(m\geq1\).  Choosing an element of \(\Gamma'\) over \(m\) identifies
\(\Gamma'\) with \(L\rtimes_{H^m}\ZZ\); conjugation preserves \(L\) and is
the restriction of \(H^m\).  Hence the corresponding connected finite cover
is again a toroidal mapping torus.  The stronger geometry-specific conclusion is
also preserved.  If \(H=I\), then the refined monodromy is again the identity
and the total space is a three-torus.  If \(H\) is nontrivial unipotent, then
\(H^m-I\) is nonzero on \(\mathbb Q^2\), so its restriction to the finite-index
lattice \(L\) is nonzero; hence the refined monodromy remains nontrivial
unipotent.  In the Sol case, the eigenvalues of \(H^m\) remain positive and
off the unit circle, so the refined monodromy still has trace greater than
\(2\).  In terms of the original monodromy, this power is \(H_0^{2m}\),
regardless of the parity of \(m\).  This applies in particular to the core
refinement.
\end{proof}

\paragraph{\normalfont\itshape Working-cover convention.}
\phantomsection\label{par:working-cover-convention}
Apply Proposition~\ref{prop:finite-torus-bundle-covers} once to the original
manifold, fix the resulting regular orientable torus-bundle cover, and lift the
flow to it.  Take the identity cover only when all the geometry-specific
conditions of that proposition already hold; in the Sol case this requires
\(\Tr H>2\).  The phrase \emph{working cover} always refers to this one fixed
cover and its lifted flow.  It is \(\TT^3\) with identity monodromy in the
Euclidean case, a \(\Nil\) torus bundle with nontrivial unipotent monodromy
in the Nil case, and a \(\Sol\) torus bundle with \(\Tr H>2\) in the Sol
case.  Thus it lies exactly in the scope of
Proposition~\ref{input:bamler-fixed-threshold}.  The base double cover and
regular refinement are part of this initial choice; no further cover is
needed to apply Bamler's theorem.

Apply that proposition with the fixed global choice \(c_0=c_{\mathrm{af}}\)
directly to the flow on this cover, and carry out all block constructions on
the same cover.  The curvature constant is unchanged by pullback, while
\(T_B(c_{\mathrm{af}})\) and \(\eps_B^{c_{\mathrm{af}}}\) are the tail
data obtained for the lifted flow.  Any later auxiliary cutoff is applied to
this same flow and receives separately labeled tail data.  In the local
constructions below, the working cover is denoted simply by \(M\); the
notation \(M_{\mathrm{fin}}\) is restored only when it must be distinguished
from the original manifold in the final assembly.  The cyclic cover used for
torus translations is always \(M_{\mathrm{cyc}}\), and its universal cover is
\(\widetilde M\).  All descent statements are collected in
Section~\ref{sec:finite-cover-descent}.

\subsection{Torus-bundle conventions}
\label{subsec:torus-bundle-conventions}
Under the \hyperref[par:working-cover-convention]{working-cover convention},
fix the bundle projection
\(\pi:M\to S^1\) on the already selected torus-bundle cover.  Its monodromy
is represented by a fixed matrix \(H\in\SL(2,\ZZ)\), with the terminology
fixed in Subsection~\ref{subsec:topological-preliminaries}.  No further
preliminary passage to a cover is made in the local constructions below.

We view such a bundle as a \emph{twisted principal \(\TT^2\)-bundle}.  Choose
fiber orientations and an atlas over intervals in the base in which the fibers
are principal \(\TT^2\)-spaces.  On overlaps, the fiber translations are
conjugated by constant elements of \(\SL(2,\ZZ)\) coming from the monodromy.
Equivalently, after choosing a coordinate on the base, the bundle is the
mapping torus of an affine automorphism of \(\TT^2\) whose linear part is
\(H\).  The local principal translations give local \(\TT^2\)-actions.  A
tensor field, and in particular a metric, is called \emph{invariant} if, in
each such chart, it is invariant under these local fiber translations.  This condition is independent of the chart, because the
transition maps identify the local translation actions by torus automorphisms.

\subsection{Universal-cover noncollapsing}
The next lemma is tailored to the torus-bundle structures that occur throughout
the paper. 

\begin{lemma}[Universal-cover noncollapsing for torus bundles]
\label{lem:universal-cover-inj}
Fix \(B\geq1\).  There is a number \(i_B>0\) with the following property.  Let
\(M\) be the total space of a \(\TT^2\)-bundle over \(S^1\), let \(h\) be an
invariant metric for that bundle structure, and let \(g\) be
another metric satisfying
\[
        B^{-1}h\leq g\leq Bh.
\]
If \(|\sec_g|\leq\Lambda\) for some \(\Lambda>0\), then the universal cover
satisfies
\[
        \inf_{\widetilde x\in\widetilde M}
        \inj_{\widetilde g}(\widetilde x)
        \geq i_B\Lambda^{-1/2}.
\]
\end{lemma}

\begin{proof}
By scaling, it is enough to prove the result for \(\Lambda=1\).  Lift the bundle
and both metrics to the universal cover.  The base becomes \(\RR\), the local
torus translations become a free isometric \(\RR^2\)-action for
\(\widetilde h\), and horizontal transport gives a global identification
\[
        \widetilde M\cong\RR^2\times\RR.
\]
Use \(\widetilde h\)-arclength \(s\) on the base and horizontally transported
fiber coordinates.  Then
\[
        \widetilde h=ds^2+G_{ij}(s)\,dx^i dx^j,
\]
with \(G(s)\) positive definite.  This formula is global; the monodromy affects
only the identification after one period downstairs.

Consider \(\widetilde x=(0,0)\), and for \(r>0\) define
\[
        \Omega_r=
        \left\{(x,s): |s|<r,\; x^TG(s)x<r^2\right\}.
\]
To reach a point of \(\Omega_r\), first move horizontally from \((0,0)\) to
\((0,s)\), then move along a straight segment in the \(\RR^2\)-orbit.  Both
pieces have length less than \(r\), so
\(\Omega_r\subset B_{\widetilde h}(\widetilde x,2r)\).  For each fixed \(s\),
the fiber section is a radius-\(r\) disk for the metric \(G(s)\), hence has
area \(\pi r^2\).  Therefore
\[
        \operatorname{vol}_{\widetilde h}(\Omega_r)=2\pi r^3.
\]

The bilipschitz comparison gives
\(d_{\widetilde g}\leq B^{1/2}d_{\widetilde h}\) and the volume-measure inequality
\(d\mu_{\widetilde g}\geq B^{-3/2}d\mu_{\widetilde h}\).  Taking \(r=(2\sqrt B)^{-1}\) yields
\[
\begin{split}
        \operatorname{vol}_{\widetilde g}
        B_{\widetilde g}(\widetilde x,1)
        &\geq B^{-3/2}\operatorname{vol}_{\widetilde h}(\Omega_r)\\
        &=B^{-3/2}\,2\pi(2\sqrt B)^{-3}
          =\frac{\pi}{4B^3}.
\end{split}
\]
The same lower bound holds at every basepoint.  Since
\(|\sec_{\widetilde g}|\leq1\), the Cheeger--Gromov--Taylor estimate
\cite[Section~4]{CGT} gives
\[
        \inf_{\widetilde x\in\widetilde M}
        \inj_{\widetilde g}(\widetilde x)\geq i_B>0.
\]
Scaling back gives the stated lower bound \(i_B\Lambda^{-1/2}\).
\end{proof}

\subsection{Averaging and elementary collapse}
\label{subsec:torus-averaging-collapse}
We next establish the averaging estimate and the elementary consequence that small invariant
fibers make the metric space close to its quotient circle.

For a tensor field \(U\) on \(M\), write \(\ol U\) for its average over the
chosen local \(\TT^2\)-actions and write \(U^\perp=U-\ol U\).  Normalized Haar
measure is preserved by every automorphism of \(\TT^2\), in particular by the
\(\SL(2,\ZZ)\) changes of fiber coordinates on overlaps.  Hence the local
averages agree on overlaps and define a global tensor.  In particular, for a
Riemannian metric \(g\) we write \(g^\perp=g-\ol g\).  If a tensor \(U\) is pulled back to the cyclic cover
\(M_{\mathrm{cyc}}\) obtained by unwrapping the base circle, we denote the
pullback by \(U_{\mathrm{cyc}}\).  Its further pullback to the universal
cover is denoted by \(\widetilde U\).

\begin{lemma}[Averaging in invariant H\"older norms]
\label{lem:averaging-invariant-holder}
Let \(h\) be an invariant metric for the fixed torus
structure, and let \(\mathcal A\) denote the corresponding averaging operator.
For every integer \(k\geq0\), every \(\sigma\in(0,1)\), and every tensor field
\(U\),
\begin{equation}\label{eq:averaging-holder-contraction}
        \|\mathcal A U\|_{C^{k,\sigma}(M,h)}
        \leq \|U\|_{C^{k,\sigma}(M,h)}.
\end{equation}
Averaging also commutes with constant metric scaling.
\end{lemma}

\begin{proof}
Pass to \(M_{\mathrm{cyc}}\), on which the fiber translations are the global
maps \(R_c\), \(c\in\TT^2\).  Each \(R_c\) is an isometry of
\(h_{\mathrm{cyc}}\).  Pullback by \(R_c\) therefore commutes with the
Levi-Civita connection, preserves pointwise tensor norms and spatial distances,
and intertwines the parallel transports used in the intrinsic H\"older
seminorms.  Minkowski's inequality and integration against normalized Haar
measure give \eqref{eq:averaging-holder-contraction} on the cyclic cover.  The
integral is invariant under the deck transformation, because the monodromy acts
on \(\TT^2\) by a Haar-measure-preserving automorphism, and hence the estimate
descends to \(M\).  Constant scaling changes neither the maps \(R_c\) nor
normalized Haar measure.
\end{proof}

For an invariant metric \(h\), let
\[
        d_{\TT^2}(h)=\sup_{y\in S^1}\diam_h(F_y),
        \qquad F_y=\pi^{-1}(y),
\]
where the diameter is computed in the ambient metric space \((M,h)\).  We also
write \(d^{\mathrm{int}}_{\TT^2}(h)\) for the corresponding supremum of the
intrinsic diameters of the flat torus fibers with their induced metrics.  Then
\(d_{\TT^2}(h)\leq d^{\mathrm{int}}_{\TT^2}(h)\).

\begin{lemma}[Collapse to the quotient circle]
\label{lem:elementary-circle-collapse}
Let \(\pi:(M,h)\to(S^1,h_{S^1})\) be the metric quotient of an invariant
torus-bundle metric, and suppose
\[
        \sup_{y\in S^1}\diam_h\bigl(\pi^{-1}(y)\bigr)\leq\delta,
\]
where the fiber diameters are ambient diameters in \((M,h)\).  Then, for all
\(x,x'\in M\),
\begin{equation}\label{eq:elementary-circle-collapse-distance}
 0\leq d_h(x,x')-
 d_{h_{S^1}}\bigl(\pi(x),\pi(x')\bigr)\leq2\delta.
\end{equation}
Consequently, \(\pi\) has distortion at most \(2\delta\) and
\[
        d_{\mathrm{GH}}\bigl((M,h),(S^1,h_{S^1})\bigr)\leq\delta.
\]
\end{lemma}

\begin{proof}
The quotient map is distance nonincreasing, which gives the first inequality
in \eqref{eq:elementary-circle-collapse-distance}.  Given \(\eta>0\), choose
\(y\in\pi^{-1}(\pi(x))\) and \(y'\in\pi^{-1}(\pi(x'))\) so that
\[
 d_h(y,y')\leq d_{h_{S^1}}\bigl(\pi(x),\pi(x')\bigr)+\eta.
\]
The ambient fiber-diameter bound gives \(d_h(x,y)\leq\delta\) and
\(d_h(x',y')\leq\delta\).  The triangle inequality, followed by
\(\eta\downarrow0\), proves the second inequality.  Since \(\pi\) is onto, a
surjective map of distortion at most \(2\delta\) gives the stated
Gromov--Hausdorff bound.
\end{proof}

\section{Finite-time decay of the non-invariant part}\label{sec:finite-time-decay}
This section develops the finite-time decay estimate behind the
iteration.  It shows that, on a fixed rescaled time interval, the non-invariant
part of a sufficiently small Ricci--DeTurck perturbation of a locally
\(\TT^2\)-invariant Ricci flow background contracts to any prescribed fraction
of its initial
size, provided the local torus orbits are short and the quotient circle does
not collapse.
The estimate is proved on lifted coordinate systems, so its constants are
independent of collapse on the compact manifold and can be reused at every
stage.

Here is the proof mechanism.  Set
\[
        a_0=\bigl\|g_{\RDT}^\perp(1)\bigr\|_{C^{N,\sigma}(M,h(1))}.
\]
The case \(a_0=0\) follows at once from invariance and uniqueness.  When
\(a_0>0\), pass to the cyclic cover and write
\(u=\pi_{\mathrm{cyc}}^*g_{\RDT}\).  For every torus translation \(R_a\),
the normalized difference
\[
        y_a=a_0^{-1}(u-R_a^*u)
\]
satisfies a homogeneous linear parabolic system with scalar principal part.
A maximum principle argument gives a collapse-independent \(C^0\)-bound
for \(y_a\).  Lifting to the universal cover and applying interior Schauder
estimates then gives a uniform \(C^1\)-bound on the time interval
\([10,100]\).  Averaging the translation differences recovers
\(a_0^{-1}g_{\RDT}^\perp\).  Since this tensor has zero orbit average, the
short-orbit estimate makes its \(C^0\)-norm small, proving the \(C^0\)
contraction.  Higher derivative contraction follows from interpolation between
the small \(C^0\)-norm and higher derivative bounds.

The section is organized as follows.  First, we define the
lifted bounded-geometry conditions \(\BG\) and \(\PBG\), give the
parabolic averaging estimate, fix the Ricci--DeTurck equation and the
intrinsic parabolic spaces \(X_s^m\) and \(Y_s^m\), and then state the
finite-time decay theorem.  The first subsection, \emph{Short-orbit
averaging}, controls the geometry of the torus orbits, compares their
ambient and intrinsic diameters, and proves that a zero-average tensor has
small \(C^0\)-norm when the orbits are short.  The second subsection,
\emph{Translation differences and the Ricci--DeTurck system}, normalizes the
differences between a solution and its torus translates, derives a linear
tensor system with scalar principal part, obtains a uniform \(C^0\)-bound on
a compact associated bundle, and upgrades it by lifted Schauder estimates.
The final subsection, \emph{The lifted estimate and completion of the proof},
averages the translation differences to recover the non-invariant part and
combines the short-orbit \(C^0\)-estimate with parabolic interpolation to
obtain any prescribed contraction on \([10,100]\).  The detailed
collapse-independent Schauder estimates used in this argument are
proved in Appendix~\ref{app:parabolic-estimates}.

\paragraph{\normalfont\itshape Bounded lifted geometry notation.}
Let \(m\geq2\) and \(\Lambda\geq1\).  For a metric \(s\) on a compact manifold, the
notation
\[
        \BG_m(s;\Lambda)
\]
means that, on the universal cover,
\begin{equation}
        \inj_{\widetilde s}\geq \Lambda^{-1},
        \qquad
        |\widetilde\nabla^j\Rm_{\widetilde s}|\leq \Lambda
        \quad(0\leq j\leq m+3).
\end{equation}
In applications, the condition is imposed after the parabolic rescaling that makes the stage
start time equal to one.  For a Ricci flow \(s(t)\), \(a\leq t\leq b\), the
notation
\[
        \PBG_m(s;\Lambda;[a,b])
\]
means that, on the universal cover,
\begin{equation}\label{eq:bounded-lifted-geometry}
        \inj_{\widetilde s(t)}\geq \Lambda^{-1},
        \qquad
        |\widetilde\nabla^j\Rm_{\widetilde s(t)}|\leq \Lambda
        \quad(0\leq j\leq m+3),\qquad a\leq t\leq b,
\end{equation}
and that the fixed-coordinate, uniformly locally finite parabolic atlas
constructed in Lemma~\ref{lem:uniform-parabolic-atlas} may be chosen with chart
radius at least \(\Lambda^{-1}\) and all coefficient, overlap, and cutoff bounds at
most \(\Lambda\).  Enlarging \(\Lambda\) by a function of the curvature and injectivity
bounds provides this atlas automatically.  Thus \(\BG\) encodes the finite-order
geometry at one stage start, while \(\PBG\) encodes the corresponding geometry
and coordinate control on a whole parabolic interval.  The tensor H\"older
norms below are the intrinsic lifted norms associated with the reference
metric, as specified in Paragraph~\ref{par:lifted-norm-convention}; in
particular, they are defined for every smooth reference flow, without a
\(\PBG\) hypothesis.  Whenever \(\PBG_m\) holds, these intrinsic norms are
uniformly equivalent to the fixed-coordinate uniformly local norms furnished
by Lemma~\ref{lem:uniform-parabolic-atlas}.  Thus the coordinate estimates used
below have constants independent of collapse of the compact quotient.

Let \(h(t)\) be an invariant Ricci flow.  A family
\(g_{\RDT}(t)\) satisfies the Ricci--DeTurck equation relative to \(h(t)\) if
\begin{equation}\label{eq:rdt}
        \frac{\partial}{\partial t}g_{\RDT}(t)
        =-2\Ric(g_{\RDT}(t))+\Lie_{W(t)}g_{\RDT}(t),
\end{equation}
where
\begin{equation}
        W^k(t)=g_{\RDT}(t)^{ij}
        \bigl(\Gamma^k_{ij}(g_{\RDT}(t))-\Gamma^k_{ij}(h(t))\bigr).
\end{equation}
The background flow \(h(t)\) itself satisfies \eqref{eq:rdt} relative to
\(h(t)\).

For a reference Ricci flow \(s(t)\) on \([a,b]\), an integer \(m\geq2\),
and \(\sigma\in(0,1)\), set
\begin{equation}\label{eq:XY-spaces}
\begin{split}
 X_s^m[a,b]&=C^{m+\sigma,(m+\sigma)/2}(M\times[a,b],s),\\
 Y_s^m[a,b]&=C^{m-2+\sigma,(m-2+\sigma)/2}(M\times[a,b],s),
\end{split}
\end{equation}
where these symbols denote the intrinsic lifted parabolic norms of
Paragraph~\ref{par:lifted-norm-convention}.

\begin{lemma}[Averaging in parabolic H\"older norms]
\label{lem:averaging-parabolic-holder}
Let \(h(t)\), \(a\leq t\leq b\), be a smooth family of invariant metrics for the fixed torus structure, and let
\(\mathcal A\) be the corresponding averaging operator.  For every integer
\(m\geq2\) and every time-dependent tensor field \(U\),
\begin{equation}\label{eq:averaging-parabolic-holder-contraction}
 \|\mathcal A U\|_{X_h^m[a,b]}\leq\|U\|_{X_h^m[a,b]},
 \qquad
 \|\mathcal A U\|_{Y_h^m[a,b]}\leq\|U\|_{Y_h^m[a,b]}.
\end{equation}
If \(\PBG_m(h;\Lambda;[a,b])\) holds, the analogous estimates in the
fixed-chart uniformly local norms hold with a constant depending only on
\(m,\sigma,\Lambda\), and \(b-a\); in those norms the constant need not equal
one.  Averaging commutes with constant metric scaling.
\end{lemma}

\begin{proof}
On \(M_{\mathrm{cyc}}\), the translations \(R_c\) are time-independent and
satisfy \(R_c^*h_{\mathrm{cyc}}(t)=h_{\mathrm{cyc}}(t)\).  Hence their pullbacks
commute with the spatial Levi-Civita connection and the metric-compatible time
derivative \(\mathcal D_t^{h_{\mathrm{cyc}}}\), and preserve the parallel
transports in the intrinsic H\"older seminorms of
Paragraph~\ref{par:lifted-norm-convention}.  Minkowski's inequality and
integration against normalized Haar measure therefore give
\eqref{eq:averaging-parabolic-holder-contraction}.  Deck invariance lets the
estimates descend to \(M\).  Under \(\PBG_m\), the fixed-chart estimates follow
from Lemma~\ref{lem:intrinsic-fixed-chart-equivalence}.  The scaling assertion
is immediate.
\end{proof}

We will prove the following theorem in this section.  For the
Ricci--DeTurck metric appearing in its statement, set
\[
        g_\RDT^\perp=g_\RDT-\overline{g_\RDT},
\]
where the average is taken in the fixed local \(\TT^2\)-structure.

\medskip
\noindent\emph{Informal content.}
On the normalized interval \([1,100]\), sufficiently short torus orbits
suppress the non-invariant part by any prescribed factor, uniformly in the
amount of collapse of the compact quotient.
\medskip

\begin{theorem}[Finite-time decay of the non-invariant part]\label{thm:noninvariant-decay}
Given \(0<\alpha<1\), \(K\geq0\), \(\lambda>0\), an integer \(N\geq2\), and
\(\sigma\in(0,1)\), there are constants
\(\beta_{\mathrm{orb}}>0\) and \(\beta_{\mathrm{pert}}>0\) with the following
property.  Let \(h(t)\) be an invariant Ricci-flow solution on \(M\), defined
for \(t\in[1,100]\), such that \(|\sec_{h(t)}|\leq K\) and such that, for every
\(t\in[1,100]\), the quotient circle for the chosen local \(\TT^2\)-structure
has length at least \(\lambda\), while every local \(\TT^2\)-orbit has
\(h(t)\)-diameter less than \(\beta_{\mathrm{orb}}\).  Let \(g_\RDT(t)\) be a
solution of the Ricci--DeTurck equation relative to \(h(t)\), defined for
\(t\in[1,100]\), and assume
\[
        \|g_\RDT-h\|_{X_h^N[1,100]}<\beta_{\mathrm{pert}}.
\]
Then
\begin{equation}\label{eq:main-estimate}
        \|g_\RDT^\perp\|_{X_h^N[10,100]}
        \leq \alpha\,\bigl\|g_\RDT^\perp(1)\bigr\|_{C^{N,\sigma}(M,h(1))}.
\end{equation}
\end{theorem}

The orbit-diameter hypothesis refers to the same fixed local \(\TT^2\)-structure
used in the averaging; Lemma~\ref{lem:pointwise-fixed-fiber} provides it in the
geometric applications.

\subsection{Short-orbit averaging}

Once \(K\) is fixed, put
\[
        i_K=i_1\max\{K,1\}^{-1/2},
\]
where \(i_1\) is the constant from
Lemma~\ref{lem:universal-cover-inj} for \(B=1\).  In the applications below,
Lemma~\ref{lem:universal-cover-inj}, Shi's estimates, and
Lemma~\ref{lem:uniform-parabolic-atlas} produce the lifted charts used in
Appendix~\ref{app:parabolic-estimates}.  The cyclic cover used in the proof
of the decay theorem serves only to make the local torus translations global.

The next fixed-time estimate says that a tensor with zero orbit average is controlled by its derivatives when the orbits are short.

\begin{lemma}[Small-orbit averaging estimate]
\label{lem:small-orbit-averaging}
Fix \(K<\infty\) and \(\lambda>0\).  There are constants \(\beta_0>0\) and
\(C<\infty\), depending only on \(K\) and \(\lambda\), with the following
property.  Let \(h\) be an invariant metric satisfying
\(|\sec_h|\leq K\), whose quotient circle has length at least \(\lambda\), and
whose local \(\TT^2\)-orbits have diameter at most \(d\leq\beta_0\).  If \(U\)
is a smooth symmetric two-tensor with zero average over the local
\(\TT^2\)-orbits, then
\begin{equation}\label{eq:small-orbit-C0}
        \|U\|_{C^0(M,h)}
        \leq C d\,\|U\|_{C^1(M,h)}.
\end{equation}
\end{lemma}

\begin{proof}
\proofstep{Step 1. Control the variation of the orbit metrics.}

Unwrap the base circle, obtaining
\(\pi_{\mathrm{cyc}}:M_{\mathrm{cyc}}\to M\).  The local torus translations
become a global isometric \(\TT^2\)-action \(R_a\), \(a\in\TT^2\).  In a
global horizontal gauge, with arclength \(s\) on the lifted base,
\[
        h_{\mathrm{cyc}}=ds^2+G_{ij}(s)\,dx^i dx^j.
\]
The orbit shape operator is
\(A=\frac12G^{-1}G_s\).  Along a horizontal geodesic it satisfies
\[
        A_s+A^2+R_{\partial_s}=0,
        \qquad
        R_{\partial_s}(V)=R(V,\partial_s)\partial_s.
\]
Let \(\lambda_{\min}\) be the smallest eigenvalue of \(A\).  At points of
differentiability, and in the barrier sense elsewhere,
\[
        \lambda_{\min}'\leq K-\lambda_{\min}^2.
\]
If \(\lambda_{\min}(s_0)<-\sqrt K\), comparison with
\(y'=K-y^2\) forces blow-up to \(-\infty\) in finite forward time.  That is
impossible because the vertical Jacobi tensor is nonsingular, equivalently
because \(G(s)\) stays smooth and positive definite for all \(s\in\RR\).
Thus \(\lambda_{\min}\geq-\sqrt K\).  Reversing \(s\) replaces \(A\) by
\(-A\), so \(\lambda_{\max}\leq\sqrt K\).  Hence
\begin{equation}\label{eq:fiber-shape-bound}
        \|A(s)\|\leq C_K
        \qquad\text{for every }s\in\RR.
\end{equation}
Consequently, horizontal transport through base distance \(r\) changes
vertical lengths by at most the factor \(e^{C_Kr}\).

\proofstep{Step 2. Compare intrinsic and ambient orbit diameter.}

Decrease \(\beta_0\) so that \(\beta_0<\lambda/10\).  Since
\(d\leq\beta_0\), we then have \(d<\lambda/10\).  Take two points in one orbit
and choose \(0<\varepsilon<\lambda-d\).  Join them by an ambient path of length
at most \(d+\varepsilon\).  Lift the path to the cyclic cover after choosing a
lift of its initial point.  Its endpoint must
lie in the same lifted orbit as its initial point.  Indeed, an endpoint in
a nontrivial deck translate of that orbit would make the projected path wind
nontrivially around the quotient circle, whose shortest noncontractible loop
has length at least \(\lambda\); this contradicts \(d+\varepsilon<\lambda\).  The base
coordinate of the lifted path therefore stays within \(d+\varepsilon\) of that
lifted orbit.  Project the path horizontally back to this orbit.  By
\eqref{eq:fiber-shape-bound}, the projected path has length at most
\(e^{C_K(d+\varepsilon)}(d+\varepsilon)\).  Letting
\(\varepsilon\downarrow0\) gives
\begin{equation}
        d^{\mathrm{int}}_{\TT^2}(h)\leq C_Kd.
\end{equation}

\proofstep{Step 3. Control the differential of a translation.}

Fix \(a\in\TT^2\) and \(x\in M_{\mathrm{cyc}}\), and join \(x\) to \(R_ax\) by an intrinsic
minimizing geodesic in the flat orbit.  Its length is at most \(C_Kd\).
Translation-invariant orthonormal vertical frames are parallel for the
intrinsic flat metric.  Their ambient covariant derivatives and the derivative of the
invariant horizontal unit normal are bounded by the second fundamental form,
hence by \eqref{eq:fiber-shape-bound}.  Let \(E_1,E_2\) be an invariant orthonormal vertical frame, and let
\(N\) be the
invariant horizontal unit normal along the orbit geodesic.  The differential
\(dR_a\) carries the frame \((E_1,E_2,N)\) at \(x\) to the same invariant
frame at \(R_ax\).  Integrating the bounds for
\(\nabla_{\dot\gamma}E_i\) and \(\nabla_{\dot\gamma}N\) along the geodesic and
comparing with ambient parallel transport \(\mathcal P_{x,R_ax}\) gives
\begin{equation}\label{eq:isometry-differential-control}
        \bigl\|dR_a-\mathcal P_{x,R_ax}\bigr\|\leq C_Kd.
\end{equation}
Here and below, \(U\) denotes the pullback to \(M_{\mathrm{cyc}}\).
Comparing tensor values first by parallel transport and then by
\eqref{eq:isometry-differential-control}, we obtain
\[
 \bigl|U(x)-R_a^*U(x)\bigr|
 \leq C_Kd\left(\|\nabla U\|_{C^0}+\|U\|_{C^0}\right).
\]

\proofstep{Step 4. Average.}

Because the orbit average of \(U\) is zero,
\[
        U=\int_{\TT^2}(U-R_a^*U)\,da.
\]
Integrating the previous estimate with normalized Haar measure and taking the
supremum gives
\[
        \|U\|_{C^0}
        \leq C_Kd\left(\|\nabla U\|_{C^0}+\|U\|_{C^0}\right).
\]
Decrease \(\beta_0\) once more so that \(C_K\beta_0\leq1/2\).  Since
\(d\leq\beta_0\), we have \(C_Kd\leq1/2\), and the last term can be absorbed.
This is \eqref{eq:small-orbit-C0}.
\end{proof}

\subsection{Translation differences and the Ricci--DeTurck system}

We give a preview of the rest of the argument.  The two covers in the proof
have distinct roles:
\[
        \widetilde M
        \xrightarrow{\text{universal cover}}
        M_{\mathrm{cyc}}
        \xrightarrow{\text{cyclic cover}}
        M.
\]
\begin{center}
\small
\begin{tabular}{@{}>{$}l<{$}@{\quad}p{0.76\textwidth}@{}}
M_{\mathrm{cyc}} & Global torus translations and the compact
associated-bundle maximum-principle argument.\\[2pt]
\widetilde M & Bounded geometry, fixed-radius harmonic-coordinate cylinders,
and all parabolic Schauder estimates.
\end{tabular}
\end{center}

\paragraph{\normalfont\itshape Outline of the proof.}
The following table summarizes the translation-difference argument and the
final proof of Theorem~\ref{thm:noninvariant-decay}.
\begin{center}
\small
\renewcommand{\arraystretch}{1.12}
\begin{tabular}{@{}>{\raggedright\arraybackslash}p{0.10\textwidth}>{\raggedright\arraybackslash}p{0.39\textwidth}>{\raggedright\arraybackslash}p{0.43\textwidth}@{}}
\textbf{Step} & \textbf{Hypothesis or operation} & \textbf{Conclusion}\\ \hline
1 & Compare \(u\) with each translate \(R_a^*u\). & A homogeneous tensor system with scalar principal part.\\
2 & Apply the maximum principle on the compact associated bundle. & A collapse-independent \(C^0\)-bound for the normalized differences.\\
3 & Lift to \(\widetilde M\) and regularize the coefficients. & Uniform finite-order coefficient bounds in fixed-radius charts.\\
4 & Apply the lifted Schauder estimate. & A uniform \(C^{N+1+\sigma,(N+1+\sigma)/2}\)-bound on \([9,100]\).\\
5 & Average over translations and interpolate on \([10,100]\). & The prescribed contraction of the non-invariant part.\\
\end{tabular}
\end{center}
The proposition below proves Steps~1--4; Step~5 is the concluding proof of the
decay theorem.

\begin{proposition}[Uniform lifted estimate for translation differences]
\label{prop:uniform-translation-difference}
Under the hypotheses of Theorem~\ref{thm:noninvariant-decay}, put
\[
        a_0=\bigl\|g_\RDT^\perp(1)\bigr\|_{C^{N,\sigma}(M,h(1))},
\]
and assume that \(a_0>0\).  Let
\(\pi_{\mathrm{cyc}}:M_{\mathrm{cyc}}\to M\) be the cyclic cover and let
\(R_a\), \(a\in\TT^2\), denote its global torus action.  Write
\(h_{\mathrm{cyc}}\) and \((g_\RDT)_{\mathrm{cyc}}\) for the lifted metrics,
and set
\[
        u=(g_\RDT)_{\mathrm{cyc}},\qquad
        u_a=R_a^*(g_\RDT)_{\mathrm{cyc}},\qquad
        z_a=u-u_a,\qquad
        y_a=a_0^{-1}z_a.
\]
If tildes denote the further lifts to the universal cover, then
\[
 \sup_{a\in\TT^2}
 \|\widetilde y_a\|_{C^{N+1+\sigma,(N+1+\sigma)/2}_{\mathrm{uloc}}
       (\widetilde M\times[9,100],\widetilde h)}
 \leq C.
\]
Here the subscript \(\mathrm{uloc}\) denotes the uniformly local fixed-chart
norm defined in Paragraph~\ref{par:lifted-norm-convention}.  The constant
\(C\) depends only on the fixed data in
Theorem~\ref{thm:noninvariant-decay}, and not on the collapsed quotient.
\end{proposition}

\begin{proof}
\proofstep{Step 1. Apply the general Ricci--DeTurck difference lemma.}

Let
\[
        u_s=u_a+s(u-u_a),\qquad 0\leq s\leq1.
\]
Choose \(\beta_{\mathrm{pert}}\) below a constant depending only on the
fixed data so that each \(u_s\) is a metric and
\(2^{-1}h_{\mathrm{cyc}}\leq u_s\leq2h_{\mathrm{cyc}}\).  Apply
Lemma~\ref{lem:rdt-difference-system} to \(u_a\) and \(u\), with background
\(h_{\mathrm{cyc}}\), and denote its coefficients by
\[
        P_a^{pq}=\int_0^1u_s^{pq}\,ds,\qquad
        B_a^p=\int_0^1\mathcal B_s^p\,ds,\qquad
        C_a=\int_0^1\mathcal C_s\,ds.
\]
Dividing the resulting equation for \(z_a\) by \(a_0\) gives the homogeneous
local system
\begin{equation}\label{eq:finite-translate-equation}
        \partial_t y_a-P_a^{pq}\nabla_p\nabla_qy_a
        =B_a^p\nabla_p y_a+C_a y_a.
\end{equation}
The second-order part acts separately on every tensor component by the same
operator \(P_a^{pq}\nabla_p\nabla_q\); \(B_a^p\) and \(C_a\) are
endomorphism-valued and contain all coupling between components.  Uniformly in \(a\),
\[
        c\,h_{\mathrm{cyc}}^{pq}\xi_p\xi_q
        \leq P_a^{pq}\xi_p\xi_q
        \leq C\,h_{\mathrm{cyc}}^{pq}\xi_p\xi_q,
        \qquad
        |P_a-h_{\mathrm{cyc}}^{-1}|\leq C\beta_{\mathrm{pert}}.
\]
Since \(N\geq2\), \eqref{eq:linearized-coefficient-structure} also gives, on
the full cylinder,
\begin{equation}\label{eq:coefficient-zero-order-bounds}
        |B_a|\leq C\beta_{\mathrm{pert}},
        \qquad |C_a|\leq C(K+\beta_{\mathrm{pert}}).
\end{equation}
No derivatives of \(\Rm(h_{\mathrm{cyc}})\) are needed for these bounds.
Indeed, writing \(u_s=h_{\mathrm{cyc}}+e_s\), we have
\(\nabla^{h_{\mathrm{cyc}}}u_s=\nabla^{h_{\mathrm{cyc}}}e_s\) and
\((\nabla^{h_{\mathrm{cyc}}})^2u_s=(\nabla^{h_{\mathrm{cyc}}})^2e_s\), because
\(\nabla^{h_{\mathrm{cyc}}}h_{\mathrm{cyc}}=0\).  Because
\(R_a^*h_{\mathrm{cyc}}=h_{\mathrm{cyc}}\), the tensor \(e_s\) inherits the
assumed \(X_h^N\)-bound uniformly in \(a\) and \(s\).  Thus the derivative
terms in \(B_a\) and \(C_a\) are controlled by that bound, while the remaining
background term is algebraic in \(\Rm(h_{\mathrm{cyc}})\) and is controlled by
\(K\).

The cyclic cover may be collapsed, and no bounded-geometry chart on
\(M_{\mathrm{cyc}}\) is used at this stage.  Equation
\eqref{eq:finite-translate-equation}, the pointwise ellipticity bound, and
\eqref{eq:coefficient-zero-order-bounds} are the only coefficient information
needed before we lift the system to the universal cover in Step~3.  All
higher coefficient estimates will be proved there.

At \(t=1\), the invariant parts cancel:
\[
        z_a(1)=(g_\RDT)_{\mathrm{cyc}}^\perp(1)
                 -R_a^*(g_\RDT)_{\mathrm{cyc}}^\perp(1).
\]
Therefore
\begin{equation}\label{eq:finite-translate-initial}
        \|y_a(1)\|_{C^{N,\sigma}(M_{\mathrm{cyc}},h_{\mathrm{cyc}}(1))}
        \leq2.
\end{equation}
By the norm convention stated immediately after the parabolic H\"older norm
definition in Subsection~\ref{subsec:lifted-parabolic-estimates}, the norm in
\eqref{eq:finite-translate-initial} is computed after lifting to the universal
cover; it does not presuppose a fixed-radius atlas on \(M_{\mathrm{cyc}}\).

\proofstep{Step 2. Obtain a uniform \(C^0\)-bound.}

Let \(\phi_a=|y_a|_{h_{\mathrm{cyc}}(t)}^2\).  Since
\(\partial_t h_{\mathrm{cyc}}=-2\Ric(h_{\mathrm{cyc}})\), differentiating the norm and
using \eqref{eq:finite-translate-equation} gives
\[
\begin{split}
        (\partial_t-P_a^{pq}\nabla_p\nabla_q)\phi_a
        &\leq-c|\nabla y_a|_{h_{\mathrm{cyc}}}^2
              +C|y_a|_{h_{\mathrm{cyc}}}|\nabla y_a|_{h_{\mathrm{cyc}}}
              +C|y_a|_{h_{\mathrm{cyc}}}^2\\
        &\leq C\phi_a.
\end{split}
\]
The cyclic cover is noncompact, but the maximum principle can be applied on a
compact associated bundle.  Let \(\gamma\) be its deck generator.  If \(H\) is
the monodromy, then
\(\gamma R_a\gamma^{-1}=R_{Ha}\), and the lifted metrics are deck-invariant.
Hence
\[
        \gamma^*y_{Ha}=y_a,
        \qquad
        \phi_{Ha}(\gamma x_{\mathrm{cyc}},t)=\phi_a(x_{\mathrm{cyc}},t).
\]
Thus \((x_{\mathrm{cyc}},a)\mapsto\phi_a(x_{\mathrm{cyc}},t)\) descends to
\[
        \mathcal Z=(M_{\mathrm{cyc}}\times\TT^2)/
        \bigl((x_{\mathrm{cyc}},a)\sim(\gamma x_{\mathrm{cyc}},Ha)\bigr),
\]
which is compact.  It is the \(\TT^2\)-bundle over
\(M_{\mathrm{cyc}}/\langle\gamma\rangle=M\) associated to the
monodromy action \(a\mapsto Ha\).  At a spatial maximum on
\(\mathcal Z\), with \(a\) fixed,
the \(M_{\mathrm{cyc}}\)-gradient vanishes and the \(P_a\)-trace of the
\(M_{\mathrm{cyc}}\)-Hessian is nonpositive.  Applying the scalar maximum principle
to \(e^{-Ct}\phi\) and using \eqref{eq:finite-translate-initial} gives
\begin{equation}\label{eq:finite-C0-bound}
        \sup_{a\in\TT^2}
        \|y_a\|_{C^0(M_{\mathrm{cyc}}\times[1,100],h_{\mathrm{cyc}})}\leq C.
\end{equation}

\proofstep{Step 3. Lift to the universal cover and prove the finite-order Schauder estimate.}

The goal is the uniform estimate \eqref{eq:finite-high-bound} for the
translation differences on \(\widetilde M\times[9,100]\).  There are two
parts.  First we obtain exactly the finite amount of interior regularity of
\(\widetilde u\) needed to control the coefficients of the equation for
\(\widetilde y_a\).  We then estimate that linear tensor system through order
\(N+1+\sigma\).

The scalar estimates and their reduction to finite tensor systems with scalar
principal part are recorded in
Remark~\ref{rem:lieberman-scalar-versus-system} and
Proposition~\ref{prop:uniform-lifted-linear-schauder}.  We apply that
proposition after obtaining the finite coefficient regularity required for the
translation-difference equation.

Let
\[
        \pi_\infty:\widetilde M\longrightarrow M_{\mathrm{cyc}}
\]
be the universal covering map.  We decorate all lifted quantities by tildes.
In particular, \(\widetilde y_a\) satisfies
\begin{equation}\label{eq:finite-translate-equation-universal}
 \partial_t\widetilde y_a
 -\widetilde P_a^{pq}\widetilde\nabla_p\widetilde\nabla_q\widetilde y_a
 =\widetilde B_a^p\widetilde\nabla_p\widetilde y_a
  +\widetilde C_a\widetilde y_a
\end{equation}
on \(\widetilde M\times[1,100]\), where \(\widetilde\nabla\) is the
connection of \(\widetilde h(t)\).  The \(C^0\)-estimate
\eqref{eq:finite-C0-bound} lifts unchanged to this equation.

Lemma~\ref{lem:universal-cover-inj}, with \(B=1\) and
\(\Lambda=\max\{K,1\}\), gives a time-uniform injectivity-radius lower bound
for \(\widetilde h(t)\).  No curvature-derivative bound at the initial time is
assumed here.  Since \(\widetilde h(t)\) is a complete Ricci flow with
uniformly bounded curvature on \([1,100]\), Shi's derivative estimates
\cite{Shi}, applied with the fixed positive time gap from \(1\) to \(7\), give
\[
        \sup_{\widetilde M\times[7,100]}
        |\widetilde\nabla^j\Rm_{\widetilde h}|
        \leq C_j(K,N),
        \qquad 0\leq j\leq N+4.
\]
Together with the preceding injectivity-radius estimate and
Lemma~\ref{lem:uniform-parabolic-atlas}, these bounds give, for a constant
\(\Lambda_h\) depending only on the fixed data,
\[
        \PBG_{N+1}(h;\Lambda_h;[7,100]).
\]
By Lemma~\ref{lem:intrinsic-fixed-chart-equivalence}, the assumed intrinsic
\(X_h^N[1,100]\)-bound controls, with a constant depending only on the fixed
data, the order-\(N\) fixed-chart norm on \([7,100]\) used below.  Let \(r_0\)
be the corresponding atlas radius from
Lemma~\ref{lem:uniform-parabolic-atlas}, reduced so that
\((4r_0)^2\leq1\).  For a center
\((\widetilde x_0,t_0)\in\widetilde M\times[8,100]\), write
\[
 Q_\rho(\widetilde x_0,t_0)
 =B_{\widetilde h(t_0)}(\widetilde x_0,\rho)
   \times[t_0-\rho^2,t_0]
\]
for the corresponding one-sided backward cylinder in the fixed spatial
coordinates of that lemma.  The choice \((4r_0)^2\leq1\) ensures that every
\(Q_{4r_0}(\widetilde x_0,t_0)\) with \(t_0\in[8,100]\) is contained in
\(\widetilde M\times[7,100]\).

\medskip
\noindent\emph{Claim 1 (finite coefficient regularity).}
Uniformly in \(a\), the coefficients of
\eqref{eq:finite-translate-equation-universal} satisfy
\begin{equation}\label{eq:coefficient-interior-regularity}
\begin{split}
        &\|\widetilde P_a\|_{C^{N+1+\sigma,(N+1+\sigma)/2}
          (\widetilde M\times[8,100],\widetilde h)}\\
        &\quad+\|\widetilde B_a\|_{C^{N+\sigma,(N+\sigma)/2}
          (\widetilde M\times[8,100],\widetilde h)}\\
        &\quad+\|\widetilde C_a\|_{C^{N-1+\sigma,(N-1+\sigma)/2}
          (\widetilde M\times[8,100],\widetilde h)}
        \leq C.
\end{split}
\end{equation}

To prove the claim, use a fixed local frame for symmetric two-tensors in one
of the preceding charts.  The lifted Ricci--DeTurck equation has the form
\begin{equation}
 \partial_t\widetilde u^A
 -\widetilde u^{pq}\partial_p\partial_q\widetilde u^A
 =F^A\bigl(x,t,\widetilde u,\widetilde u^{-1},\partial\widetilde u;
     \widetilde h,\widetilde h^{-1},
     \partial\widetilde h,\partial^2\widetilde h\bigr),
\end{equation}
where the right-hand side contains no second derivatives of
\(\widetilde u\).  The same equation holds for \(\widetilde u_a\).  The assumed
\(X_h^N[1,100]\)-closeness, lifted to the universal cover, gives a uniform
\(C^{N+\sigma,(N+\sigma)/2}\)-bound for both metrics and makes their principal
matrices uniformly elliptic.  The \(\PBG_{N+1}\)-bound gives the required
\(C^{N+1+\sigma,(N+1+\sigma)/2}\)-control of the reference metric in the
fixed coordinates.  The same atlas, supplied by
Lemma~\ref{lem:uniform-parabolic-atlas}, bounds
\(\|\widetilde h^{-1}\|_{C^0}\) uniformly there.  Thus the reference
nondegeneracy hypothesis in the derivative-gain lemma is also satisfied.

Apply Lemma~\ref{lem:local-rdt-one-derivative-gain} with \(k=N\), outer
cylinder \(Q_{4r_0}\), and inner cylinder \(Q_{2r_0}\).  It gives
\begin{equation}
 \|\widetilde u\|_{C^{N+1+\sigma,(N+1+\sigma)/2}(Q_{2r_0})}
 \leq C.
\end{equation}
The same bound holds for \(\widetilde u_a\), with the same constant, because
\(R_a^*h_{\mathrm{cyc}}=h_{\mathrm{cyc}}\).  The derivative count in that
lemma is exactly the one needed here: differentiating \(N-1\) times uses
\(\widetilde u\) through order \(N\) and \(\widetilde h\) through order
\(N+1\).  The formulas \eqref{eq:linearized-coefficient-structure} now give
\eqref{eq:coefficient-interior-regularity}: \(P_a\) is algebraic in the
metrics and their inverses, \(B_a\) uses one derivative, and \(C_a\) uses at
most two derivatives.  The norms on the noncompact universal cover are the
suprema of the corresponding local norms in this fixed-radius atlas.

\medskip
\noindent\emph{Claim 2 (translation-difference estimate).}
By Claim~1, the lifted system
\eqref{eq:finite-translate-equation-universal} satisfies the hypotheses of
Proposition~\ref{prop:uniform-lifted-linear-schauder} with \(m=N+1\) on
\([8,100]\), uniformly in \(a\).  Its forcing term is zero.  The uniformly
local interior estimate \eqref{eq:uniform-interior-linear-schauder}, with one
unit of room at the initial side, and the lift of the \(C^0\)-bound
\eqref{eq:finite-C0-bound} therefore give
\begin{equation}\label{eq:finite-high-bound}
 \sup_{a\in\TT^2}
 \|\widetilde y_a\|_{C^{N+1+\sigma,(N+1+\sigma)/2}_{\mathrm{uloc}}
       (\widetilde M\times[9,100],\widetilde h)}
 \leq C.
\end{equation}
The estimate is one-sided backward at the top time and is uniform in the
translation parameter.  This is the required universal-cover H\"older bound.
\end{proof}

\subsection{The lifted estimate and completion of the proof}

\begin{proof}[Proof of Theorem~\ref{thm:noninvariant-decay}]
Set
\[
        a_0=\bigl\|g_\RDT^\perp(1)\bigr\|_{C^{N,\sigma}(M,h(1))}.
\]
If \(a_0=0\), the initial metric is invariant.  On the cyclic
cover that unwraps the base, the local translations form a global
\(\TT^2\)-action, and the Ricci--DeTurck equation with invariant background is
equivariant under that action.  Uniqueness, or equivalently the zero-data
maximum principle used below, preserves the symmetry.  Hence
\(g_\RDT^\perp\equiv0\), and the estimate is immediate.  Assume from now on
that \(a_0>0\).
Use the notation of Proposition~\ref{prop:uniform-translation-difference}.  That proposition gives the uniform estimate \eqref{eq:finite-high-bound}.

Averaging the translation differences recovers the non-invariant part:
\[
\begin{aligned}
        a_0^{-1}(g_\RDT)_{\mathrm{cyc}}^\perp
        &=a_0^{-1}\left((g_\RDT)_{\mathrm{cyc}}
          -\int_{\TT^2}R_a^*(g_\RDT)_{\mathrm{cyc}}\,da\right)  \\
        &=\int_{\TT^2}y_a\,da.
\end{aligned}
\]
Because \(\gamma\) conjugates \(R_a\) to \(R_{Ha}\), and the automorphism
\(a\mapsto Ha\) preserves normalized Haar measure, this integral is
deck-invariant.  It descends to
\[
        w=a_0^{-1}g_\RDT^\perp.
\]
After lifting this identity to \(\widetilde M\), we have
\[
        \widetilde w=\int_{\TT^2}\widetilde y_a\,da.
\]
By the same Minkowski argument as in
Lemma~\ref{lem:averaging-parabolic-holder}, integration is bounded in each
local H\"older norm of the fixed universal-cover atlas.  Hence
\eqref{eq:finite-high-bound} gives
\begin{equation}\label{eq:finite-averaged-high-bound}
        \|w\|_{X_h^{N+1}[9,100]}\leq C.
\end{equation}
The tensor \(w\) has zero orbit average.

Let
\[
        d=\sup_{t\in[1,100]}d_{\TT^2}(h(t))
        \leq\beta_{\mathrm{orb}}.
\]
Applying Lemma~\ref{lem:small-orbit-averaging} to \(w(t)\) for each
\(t\in[9,100]\) and then using
\eqref{eq:finite-averaged-high-bound} gives
\begin{equation}
        \|w\|_{C^0(M\times[9,100],h)}
        \leq Cd\leq C\beta_{\mathrm{orb}}.
\end{equation}
Apply the parabolic interpolation inequality
\cite[Proposition~4.2, p.~49]{Lieberman} in the fixed-radius universal-cover
charts, take the uniformly local supremum, and then use the
intrinsic/fixed-chart norm equivalence.  With one unit of room at the initial
side and backward cylinders at the final side, this gives some
\(\mu=\mu(N,\sigma)\in(0,1)\) such that
\begin{equation}
\begin{aligned}
        \|w\|_{X_h^N[10,100]}
        &\leq C\|w\|_{C^0(M\times[9,100],h)}^\mu
        \|w\|_{X_h^{N+1}[9,100]}^{1-\mu}  \\
        &\leq C\beta_{\mathrm{orb}}^\mu.
\end{aligned}
\end{equation}
First choose \(\beta_{\mathrm{pert}}\) below all perturbative thresholds
used in Proposition~\ref{prop:uniform-translation-difference}.  With that
choice fixed, choose \(\beta_{\mathrm{orb}}\) below the threshold in
Lemma~\ref{lem:small-orbit-averaging} and so that
\(C\beta_{\mathrm{orb}}^\mu\leq\alpha\).  Since
\(g_\RDT^\perp=a_0w\), this is precisely \eqref{eq:main-estimate}.
\end{proof}

Theorem~\ref{thm:noninvariant-decay} is the analytic contraction estimate used
below: once a fixed fibration has uniformly short fibers and its quotient
circle has a uniform positive lower bound, the non-invariant part can be made
to decay by any prescribed factor on the latter part of a finite rescaled
stage.

\section{Comparison of torus fibrations}
\label{sec:fibration-comparison}

The preceding decay estimate for the non-invariant part is predicated on the
assumption that the fibers of the background invariant metric are short.
In order to apply this in the subsequent iterative process, we will need
to know that this property is preserved.  We do this by comparing the
fixed torus fibration retained by the
iteration with the possibly different fibration supplied by Bamler at each
time.  This section proves the geometric dichotomy that achieves the
comparison: equivalent short-fiber fibrations transfer smallness to one
another, whereas two inequivalent short-fiber fibrations force the entire
manifold to have small diameter.  On a large-diameter block, the second
alternative is excluded, so the retained fibration persists and its quotient
circle remains quantitatively nondegenerate.

The section is organized as follows.  The first subsection, \emph{Fiber
equivalence and diameter}, bounds the total diameter in terms of a
quotient-circle length and a fiber diameter, and defines fiber equivalence
through the integral cohomology class, or equivalently the common kernel of
the induced maps to \(\ZZ\).  The second subsection, \emph{Fiber-equivalent
small-fiber fibrations}, uses that common kernel to show that horizontal
projection of a short comparison fiber onto a fixed fiber has degree
\(\pm1\), thereby transferring the small-fiber estimate to the fixed
fibration.  The final subsection, \emph{Fiber-inequivalent fibrations and
fixed-fiber improvement}, first proves that two inequivalent short-fiber
fibrations force the whole manifold to have small diameter, and then combines
the two alternatives into a pointwise bootstrap: at definite total diameter,
the comparison fibration is fiber-equivalent to the retained one, the fixed
fibers become strictly shorter, and the fixed quotient-circle length acquires a
positive lower bound.  These conclusions are later inserted into a
first-hitting argument on each large block.

\subsection{Fiber equivalence and diameter}

We begin with the fibration-compatibility facts used later in the DeTurck
iteration.  These statements are purely topological and metric; they do not use
the Ricci flow.  Throughout this section, when a torus fibration
\(\pi:M\to S^1\) has been fixed, a \emph{fixed fiber} means a fiber of this
chosen map \(\pi\), as opposed to a different time-slice fibration.  If \(r\)
is a locally \(\TT^2\)-invariant metric for this fixed fibration, its
\emph{fixed quotient length} means the length of the metric quotient circle of
\(\pi\) with respect to \(r\).

\begin{lemma}[Base length and fiber diameter control diameter]
\label{lem:base-length-controls-diameter}
Let \(h\) be an invariant metric for a torus
fibration \(\pi:M\to S^1\), with quotient metric \(h_{S^1}\).  Let \(g\) be
another metric satisfying \(g\leq B h\); no invariance is assumed for \(g\).  If every \(\pi\)-fiber has
\(g\)-diameter at most \(d_f\), then
\[
        \diam(M,g)\leq B^{1/2}L(S^1,h_{S^1})/2+d_f.
\]
In particular, if \(d_{\TT^2}(h)\leq d\), then
\[
        \diam(M,g)\leq B^{1/2}\bigl(L(S^1,h_{S^1})/2+d\bigr).
\]
\end{lemma}

\begin{proof}
Take two points \(x,x'\in M\).  Join their projections by a shortest base
arc and lift that arc horizontally for \(h\), starting at \(x\).  Because
\(g\leq Bh\), the lift has \(g\)-length at most
\(B^{1/2}L(S^1,h_{S^1})/2\).  Its endpoint \(z\) lies in the fiber containing
\(x'\), so the ambient fiber-diameter bound gives \(d_g(z,x')\leq d_f\).
The triangle inequality therefore gives
\[
        d_g(x,x')\leq d_g(x,z)+d_g(z,x')
        \leq B^{1/2}L(S^1,h_{S^1})/2+d_f.
\]
The path used to estimate \(d_g(z,x')\) need not stay in the fiber.  Taking
the supremum over \(x,x'\) proves the first estimate.  If every fiber has
\(h\)-diameter at most \(d\), then it has \(g\)-diameter at most \(B^{1/2}d\);
substituting this into the first estimate gives the last one.
\end{proof}

\noindent We next fix the convention for comparing torus fibrations and prove the
two compatibility lemmas used in the fixed-fiber bootstrap.  The first says that
a fiber-equivalent small-fiber fibration makes the fixed fibers small.  The
second says that a fiber-inequivalent small-fiber fibration forces the whole
diameter to be small.

We shall compare fibrations by their induced cohomology classes rather than by identifying their target circles.

\begin{definition}[Fiber-equivalent torus fibrations]
\label{def:fiber-equivalent}
Let
\[
        \pi_i:M\to S^1_i,\qquad i=0,1,
\]
be torus fibrations.  The target circles are not part of the fixed data and may
be different.  Choose a basepoint \(x\in M\) and, for each \(i=0,1\), an identification
\(\pi_1(S_i^1,\pi_i(x))\cong\ZZ\), and let
\[
        \alpha_{\pi_i,x}:\pi_1(M,x)\longrightarrow\ZZ,
\]
be the surjective homomorphism induced by \(\pi_i\).  Changing the
identification changes \(\alpha_{\pi_i,x}\) by a sign.  Equivalently, \(\pi_i\)
determines a cohomology class \([\pi_i]\in H^1(M;\ZZ)\), well-defined up to
sign.  We say that \(\pi_0\) and \(\pi_1\) are \emph{fiber-equivalent}
if
\[
        [\pi_0]=\pm[\pi_1]\in H^1(M;\ZZ).
\]
Equivalently, after choosing a common basepoint,
\[
        \ker\alpha_{\pi_0,x}=\ker\alpha_{\pi_1,x}\subset\pi_1(M,x).
\]
For a different basepoint, the two kernels are changed by the same conjugacy
convention.  Thus, in fiber-subgroup language, the definition says that the
images of the two fiber fundamental groups agree, up to this usual conjugacy
ambiguity.
\end{definition}

If a bundle equivalence is given by a diffeomorphism of \(M\) isotopic to the
identity and a diffeomorphism of the target circles, then the fibrations are
fiber-equivalent; the sign records the orientation of the base map.  An
arbitrary self-diffeomorphism of \(M\) need not preserve the cohomology class.
Conversely, fiber-equivalence specifies only the underlying cohomology class, or
equivalently the common fiber kernel, and not a chosen bundle equivalence.
Because maps to \(S^1\) classify integral cohomology classes in degree one, the
relation
\[
        [\pi_0]=\pm[\pi_1]\in H^1(M;\ZZ),
\]
means that, after choosing a degree \(\pm1\) map
\(\varphi:S^1_0\to S^1_1\), the two fibrations represent the same homotopy
class of maps from \(M\) to a circle.  In other words, there is a homotopy-commutative diagram of spaces
\[
\begin{CD}
M @>{\pi_0}>> S^1_0\\
@V{\Id_M}VV @VV{\varphi}V\\
M @>{\pi_1}>> S^1_1.
\end{CD}
\]
Passing to fundamental groups gives the agreement of the kernels of the two
maps to \(\ZZ\), and that common kernel is the only topological input used in
the next lemma.

\subsection{Fiber-equivalent small-fiber fibrations}

\begin{lemma}[Fiber-equivalent small fibers force small fixed fibers]
\label{lem:isotopic-fibers-force-fixed}
Fix \(\lambda>0\).  For every \(\theta>0\) there are \(\eta>0\) and \(\gamma>0\)
with the following property.  Let \(r\) be an invariant metric
for a torus fibration \(\pi:M\to S^1\), and suppose that the quotient circle has
length at least \(\lambda\).  Let \(q\) be an invariant metric for a
possibly different torus fibration \(\pi_B:M\to S^1_B\).  Assume that
\[
        \|q-r\|_{C^1(M,r)}<\gamma,
\]
that \(\pi_B\) is fiber-equivalent to \(\pi\), and that every \(\pi_B\)-fiber
has \(q\)-diameter at most \(\eta\).  Then
\[
        d_{\TT^2}(r)<\theta.
\]
\end{lemma}

\begin{proof}
\proofstep{Step 1. Place a comparison fiber over a short base arc.}
See Figure~\ref{fig:same-fibration}. 
Choose \(\eta\) and \(\gamma\) so small that every \(\pi_B\)-fiber has
\(r\)-diameter less than \(2\eta\) and
\[
        4\eta<\min\{\lambda/10,\theta/10\}.
\]
This is possible because \(q\) and \(r\) are bilipschitz when \(\gamma\) is
small.  Since \(\pi:(M,r)\to S^1\) is distance nonincreasing, the image of a
connected \(\pi_B\)-fiber has diameter less than \(2\eta\).  The quotient
circle has length at least \(\lambda\), so that image lies in a base arc of
length less than
\(4\eta<\min\{\lambda/10,\theta/10\}\).

Fix a fiber \(F=\pi^{-1}(y_0)\), choose \(x\in F\), and let
\(F_B=\pi_B^{-1}(\pi_B(x))\).  Choose an arc \(J\subset S^1\) containing
\(\pi(F_B)\), with
\[
        \operatorname{length}(J)<\min\{\lambda/10,\theta/10\}.
\]
Because \(x\in F_B\), the point \(y_0=\pi(x)\) lies in \(J\).  Horizontal
transport over \(J\) identifies \(\pi^{-1}(J)\) with \(F\times J\).  Let
\(P:\pi^{-1}(J)\to F\) be the resulting projection; it fixes \(x\).

\proofstep{Step 2. Prove that \(P|_{F_B}\) is onto.}

Use \(x\) as basepoint.  Since \(\pi_B\) and \(\pi\) are fiber-equivalent, their kernels in
\(\pi_1(M,x)\) agree by Definition~\ref{def:fiber-equivalent}.
The inclusions of \(F\) and \(F_B\) identify their fundamental groups with this
common kernel.  Because \(J\) is an arc, \(P\) is the endpoint of a deformation
retraction of \(\pi^{-1}(J)\) onto \(F\) that fixes \(x\).  Therefore, with
\(i_F:F\hookrightarrow M\) and \(i_B:F_B\hookrightarrow M\),
\[
        (i_F)_*\circ (P|_{F_B})_*=(i_B)_*:
        \pi_1(F_B,x)\longrightarrow\pi_1(M,x).
\]
The map \((i_F)_*\) is injective, and the two inclusion maps have the same
image.  Hence \((P|_{F_B})_*\) is an isomorphism.  Thus
\(P|_{F_B}:F_B\to F\) has degree \(\pm1\), so it is onto.

\proofstep{Step 3. Bound the fixed fiber.}

Given \(p,q\in F\), choose \(p_B,q_B\in F_B\) with
\(P(p_B)=p\) and \(P(q_B)=q\).  The horizontal segments from \(p_B,q_B\) to
\(p,q\) each have \(r\)-length at most \(\operatorname{length}(J)\), while
\(d_r(p_B,q_B)\leq2\eta\).  Hence
\[
        d_r(p,q)\leq2\operatorname{length}(J)+2\eta.
\]
After decreasing \(\eta\) and \(\gamma\), the right side is less than
\(\theta\).  Taking the supremum over \(p,q\) and then over all fibers \(F\) proves
the lemma.
\end{proof}

\begin{figure}[t]
\centering
\includegraphics[width=0.9\textwidth]{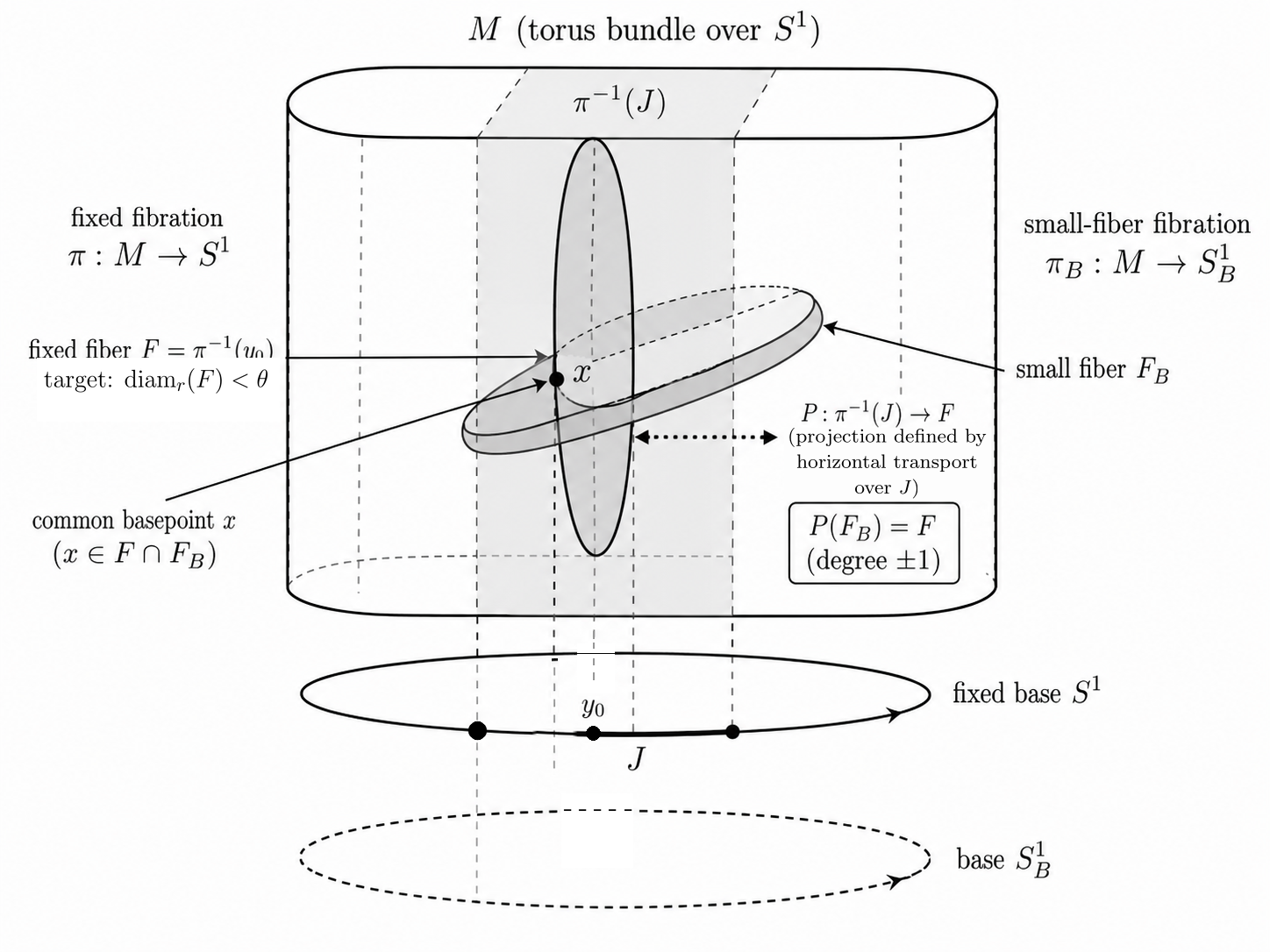}
\caption{Fiber-equivalent short fibers.  The comparison fiber lies over the
short fixed-base arc \(J\).  Horizontal projection in
\(\pi^{-1}(J)\cong F\times J\) maps it onto the fixed fiber with degree
\(\pm1\), transferring the comparison-fiber and base-arc bounds to the
target estimate \(\diam_r(F)<\theta\).}
\label{fig:same-fibration}
\end{figure}

The complementary lemma says that two fiber-inequivalent small-fiber fibrations can occur only when the whole manifold has small diameter.

\subsection{Fiber-inequivalent fibrations and fixed-fiber improvement}

\begin{lemma}[Fiber-inequivalent small fibers force small diameter]
\label{lem:different-fibrations-small-diameter}
Fix \(B\geq1\).  There is a constant \(C_{\mathrm{fib}}=C_{\mathrm{fib}}(B)\) with the
following property.  Let \(h\) be a metric on a closed three-manifold \(M\).
Suppose that \(M\) has two torus fibrations
\[
        \pi_i:M\to S^1_i,\qquad i=0,1,
\]
and invariant comparison metrics \(h_i\), with quotient metrics
\(h_{S^1,i}\), such that
\[
        B^{-1}h_i\leq h\leq B h_i,
        \qquad i=0,1.
\]
Suppose also that, for some \(d_f>0\), every fiber of each \(\pi_i\) has
\(h\)-diameter at most \(d_f\).  If \(\pi_0\) and \(\pi_1\) are fiber-inequivalent, then
\[
        \diam(M,h)\leq C_{\mathrm{fib}}d_f.
\]
\end{lemma}

\begin{proof}
Fix \(x\in M\), choose representatives of the two induced maps to \(\ZZ\), and set
\(K_i=\ker\alpha_{\pi_i,x}\subset\pi_1(M,x)\).  Let \(F_1\) be the
\(\pi_1\)-fiber through \(x\), with inclusion \(i_1:F_1\hookrightarrow M\).
Suppose first that
\[
        \alpha_{\pi_0,x}\circ(i_1)_*:
        \pi_1(F_1,x)\longrightarrow\ZZ,
\]
were zero.  Then \(K_1=(i_1)_*\pi_1(F_1,x)\subset K_0\), and this inclusion
would induce a surjection
\[
        \pi_1(M,x)/K_1\cong\ZZ
        \longrightarrow
        \pi_1(M,x)/K_0\cong\ZZ.
\]
After identifying the two quotient groups with \(\ZZ\), this is a surjective
homomorphism \(\ZZ\to\ZZ\), and hence an isomorphism.  Since its kernel is
\(K_0/K_1\), we would have \(K_0=K_1\), contrary to the assumption that the
fibrations are fiber-inequivalent.  Therefore
\(\alpha_{\pi_0,x}\circ(i_1)_*\neq0\).  Equivalently,
\(\pi_0|_{F_1}:F_1\to S^1_0\) induces a nonzero map on \(H_1\), and is thus
onto.

The quotient map \(\pi_0:(M,h)\to(S^1_0,h_{S^1,0})\) is
\(B^{1/2}\)-Lipschitz: it is distance nonincreasing for \(h_0\), and
\(h_0\leq Bh\).  Since \(\pi_0(F_1)=S^1_0\),
\[
        \diam(S^1_0,h_{S^1,0})
        \leq B^{1/2}\diam_h(F_1)
        \leq B^{1/2}d_f.
\]
Hence \(L(S^1_0,h_{S^1,0})\leq2B^{1/2}d_f\).  Applying
Lemma~\ref{lem:base-length-controls-diameter} to \(\pi_0\) gives
\[
        \diam(M,h)
        \leq B^{1/2}L(S^1_0,h_{S^1,0})/2+d_f
        \leq (B+1)d_f.
\]
Thus one may take \(C_{\mathrm{fib}}=B+1\).
\end{proof}

The next lemma isolates the pointwise geometric input used in the block
argument.  If a metric of definite diameter is close to invariant metrics for
two torus fibrations and both fibrations have sufficiently small fibers, then
the fibrations are fiber-equivalent.  The lemma also improves the
fixed-fiber bound and gives a definite lower bound for the fixed quotient
circle.

\begin{lemma}[Pointwise fixed-fiber improvement]
\label{lem:pointwise-fixed-fiber}
Fix \(\delta_0>0\), \(0<\lambda_{\mathrm{p}}<\delta_0/10\), and
\(\theta_{\max}>0\).  There are constants
\[
        0<\theta<\theta_{\max},\qquad
        \beta_{\mathrm{fib}}>0,\qquad
        \eps_{\mathrm{fib}}>0,
\]
with the following property.

Let \(q\) be a metric on a closed three-manifold \(M\) such that
\begin{equation}\label{eq:pointwise-total-diameter}
        \diam(M,q)\geq\delta_0.
\end{equation}
Let \(\pi:M\to S^1\) be a torus fibration, and let \(r\) be an invariant metric for \(\pi\).  Suppose that
\begin{equation}
        \|q-r\|_{C^1(M,r)}<\beta_{\mathrm{fib}},
        \qquad
        d_{\TT^2}(r)\leq\theta.
\end{equation}
Suppose also that \(M\) has a second torus fibration
\(\pi_B:M\to S_B^1\) and an invariant metric \(h_B\)
for \(\pi_B\) such that
\begin{equation}
        \|h_B-q\|_{C^1(M,q)}<\eps_{\mathrm{fib}},
        \qquad
        \diam_q(F_B)\leq\eps_{\mathrm{fib}},
\end{equation}
for every \(\pi_B\)-fiber \(F_B\).  Then \(\pi_B\) is fiber-equivalent to \(\pi\), and
\begin{equation}
        d_{\TT^2}(r)<\frac{\theta}{2},
        \qquad
        L(S^1,r_{S^1})\geq\lambda_{\mathrm{p}}.
\end{equation}
\end{lemma}

\begin{proof}
Choose preliminary upper bounds for \(\beta_{\mathrm{fib}}\) and
\(\eps_{\mathrm{fib}}\) so small that \(q\), \(r\), and \(h_B\) are pairwise
\(2\)-bilipschitz.  On this neighborhood there is a universal comparison
constant \(C_{\mathrm{cmp}}\) such that
\begin{equation}\label{eq:pointwise-norm-comparison}
\begin{split}
        \|h_B-r\|_{C^1(M,r)}
        &\leq C_{\mathrm{cmp}}\bigl(
        \|h_B-q\|_{C^1(M,q)}+\|q-r\|_{C^1(M,r)}\bigr),\\
        \diam_{h_B}(F_B)&\leq2^{1/2}\diam_q(F_B).
\end{split}
\end{equation}
The first estimate follows from the standard formula for the difference of two
Levi-Civita connections; the second is the bilipschitz comparison.

Let \(C_{\mathrm{fib}}=C_{\mathrm{fib}}(2)\) be the constant in
Lemma~\ref{lem:different-fibrations-small-diameter}.  Choose
\(0<\theta<\theta_{\max}\) so that
\begin{equation}\label{eq:pointwise-theta-choice}
        2^{1/2}\left(\frac{\lambda_{\mathrm{p}}}2+\theta\right)<\delta_0,
        \qquad
        C_{\mathrm{fib}}2^{1/2}\theta<\frac{\delta_0}{2}.
\end{equation}
Apply Lemma~\ref{lem:isotopic-fibers-force-fixed}, with base-length lower
bound \(\lambda_{\mathrm{p}}\) and target fiber diameter \(\theta/2\), to obtain
constants \(\eta>0\) and \(\gamma>0\).  Now choose
\(\eps_{\mathrm{fib}}\) and \(\beta_{\mathrm{fib}}\), still inside the preliminary
bilipschitz neighborhood, so that
\begin{equation}\label{eq:pointwise-error-choice}
\begin{gathered}
        C_{\mathrm{fib}}\eps_{\mathrm{fib}}<\frac{\delta_0}{2},
        \qquad
        2^{1/2}\eps_{\mathrm{fib}}<\eta,\\
        C_{\mathrm{cmp}}(\eps_{\mathrm{fib}}+\beta_{\mathrm{fib}})<\gamma.
\end{gathered}
\end{equation}

We first prove the quotient-length bound.  If
\(L(S^1,r_{S^1})<\lambda_{\mathrm{p}}\), then
Lemma~\ref{lem:base-length-controls-diameter}, applied with \(h=r\),
\(g=q\), and \(B=2\), gives
\[
        \diam(M,q)
        \leq2^{1/2}\left(\frac{\lambda_{\mathrm{p}}}2+\theta\right)
        <\delta_0,
\]
contrary to \eqref{eq:pointwise-total-diameter}.  Hence
\begin{equation}\label{eq:pointwise-quotient-long}
        L(S^1,r_{S^1})\geq\lambda_{\mathrm{p}}.
\end{equation}

The \(\pi\)-fibers have \(q\)-diameter at most \(2^{1/2}\theta\), while
the \(\pi_B\)-fibers have \(q\)-diameter at most \(\eps_{\mathrm{fib}}\).  If
the two fibrations were fiber-inequivalent, set
\[
        d_*:=\max\{2^{1/2}\theta,\eps_{\mathrm{fib}}\}.
\]
Lemma~\ref{lem:different-fibrations-small-diameter}, applied with the common
fiber-diameter bound \(d_f=d_*\), would give
\[
        \diam(M,q)\leq C_{\mathrm{fib}}d_*<\delta_0,
\]
by \eqref{eq:pointwise-theta-choice} and
\eqref{eq:pointwise-error-choice}, again contradicting
\eqref{eq:pointwise-total-diameter}.  Thus \(\pi_B\) is fiber-equivalent to \(\pi\).

Finally, \eqref{eq:pointwise-norm-comparison} and
\eqref{eq:pointwise-error-choice} give
\[
        \|h_B-r\|_{C^1(M,r)}<\gamma,
        \qquad
        \diam_{h_B}(F_B)
        \leq2^{1/2}\eps_{\mathrm{fib}}<\eta.
\]
Together with \eqref{eq:pointwise-quotient-long},
Lemma~\ref{lem:isotopic-fibers-force-fixed} yields
\(d_{\TT^2}(r)<\theta/2\), completing the proof.
\end{proof}

\section{One DeTurck stage and regular restart}
\label{sec:stage-restart}
This section turns the geometric data at the beginning of a large block into
a controlled Ricci--DeTurck evolution and supplies the regularity needed to
recover the analytic stage-start conditions at the next scale.  Starting from
a metric close to its fiberwise average, it constructs a locally
\(\TT^2\)-invariant Ricci-flow background, compares the original flow to it
throughout one stage, and uses positive-time smoothing,
averaging, and parabolic rescaling at the endpoint.  The resulting stage
estimate can be iterated with constants independent of collapse.

The section is organized as follows.  It begins by fixing the metric and gauge
notation and, for Sections~\ref{sec:stage-restart}--\ref{sec:quotient-length},
the derivative order \(N\) and H\"older exponent \(\sigma\).  The first
subsection, \emph{Rescaled blocks and admissible stage starts}, defines a
parabolically rescaled large block, states the analytic and geometric
conditions at an admissible stage start, and gives Bamler's time-slice
comparison after pulling it back by the same diffeomorphism used to define
the rescaled block.  The second subsection,
\emph{Averaged endpoints and positive-time regularization}, explains the
finite regularity needed for a restart, proves that high-order bounded
geometry survives averaging, and states the positive-time smoothing result
that reproduces the bounded-geometry part of the analytic start conditions at
the rescaled endpoint.  The final subsection, \emph{Stage preparation and
stability}, evolves the averaged initial metric by an invariant Ricci flow,
places the rescaled original flow in Ricci--DeTurck gauge relative to that
background, proves uniform continuous dependence through the whole stage, and
verifies that rescaling and averaging at time \(100\) preserve bounded geometry
and control the endpoint error.  Finite-time decay in
Section~\ref{sec:deturck-iteration} contracts that error, while fibration
comparison reproduces the short-fiber condition.  Thus this section proves the
analytic estimate used on every block of the iteration.

\paragraph{\normalfont\itshape Standing regularity convention.}
\phantomsection\label{par:standing-regularity-convention}
Throughout Sections~\ref{sec:stage-restart}--\ref{sec:quotient-length}, fix
once and for all an integer \(N\geq2\) and a H\"older exponent
\(\sigma\in(0,1)\).  Constants in these sections may depend on \(N\) and
\(\sigma\) unless stated otherwise.

\paragraph{\normalfont\itshape Standing fixed-threshold application.}
\phantomsection\label{par:standing-fixed-threshold-application}
Throughout Sections~\ref{sec:stage-restart}--\ref{sec:quotient-length}, fix a
particular application of
Proposition~\ref{input:bamler-fixed-threshold} to the working-cover flow.  Its
cutoff is denoted by \(c_0\), its tail time by \(T_B(c_0)\), and its
nonincreasing comparison error by \(\eps_B^{c_0}\).  Thus
\[
        Kc_0^2<\eps_{\mathrm{af}},
        \qquad \eps_B^{c_0}(t)\longrightarrow0.
\]
Every block, stage, and component statement in these sections is relative to
this one fixed application, and every large-diameter threshold
\(\delta_0\) satisfies \(\delta_0\geq c_0\).  The analytic constants depend
only on the parameters displayed in their statements; only the late start
times may additionally depend on \(T_B(c_0)\) and \(\eps_B^{c_0}\).  Later
arguments use this machinery with \(c_0=c_{\mathrm{af}}\), or, in
the auxiliary Euclidean argument, with \(c_0=\zeta\).

\paragraph{\normalfont\itshape Metric and gauge notation.}
\phantomsection\label{par:metric-gauge-legend}
The following notation is used throughout the block construction.
\begin{center}
\small
\renewcommand{\arraystretch}{1.12}
\begin{tabular}{@{}>{\raggedright\arraybackslash}p{0.16\textwidth}>{\raggedright\arraybackslash}p{0.20\textwidth}>{\raggedright\arraybackslash}p{0.54\textwidth}@{}}
\textbf{Level} & \textbf{Notation} & \textbf{Role} \\ \hline
physical time & \(g(t),g'(t),h(t)\) & the original Ricci flow, its continuous stagewise DeTurck pullback, and the right-continuous piecewise invariant comparison family;\\
one rescaled stage & \(p(\tau),q(\tau),r(\tau)\) & the parabolically rescaled original flow, its Ricci--DeTurck representative, and the invariant Ricci-flow background;\\
auxiliary construction & \(r_{\mathrm D}(\tau)\) & the Ricci--DeTurck solution with background \(p\) and initial data \(r(1)\), used only to construct and control the invariant Ricci flow \(r\);\\
time-slice auxiliary & \(h_B(\tau),\pi_B(\tau)\) & Bamler's invariant comparison metric and its possibly time-dependent torus fibration, pulled back to the coordinates of the current rescaled stage;\\
restart & \(\rho^+,q^+(1),r^+(1)\) & the rescaled invariant endpoint, the rescaled DeTurck endpoint, and the averaged metric used to restart the next stage;\\
fixed structure & \(\pi\) & the torus fibration retained throughout the block iteration.
\end{tabular}
\end{center}
Unsubscripted \(p,q,r\) denote a generic stage, while
\(p_k,q_k,r_k\) denote the corresponding objects on stage \(k\).

Figure~\ref{fig:three-metric-families} shows the relation among \(g\), \(g'\),
and \(h\) after the stages are assembled.

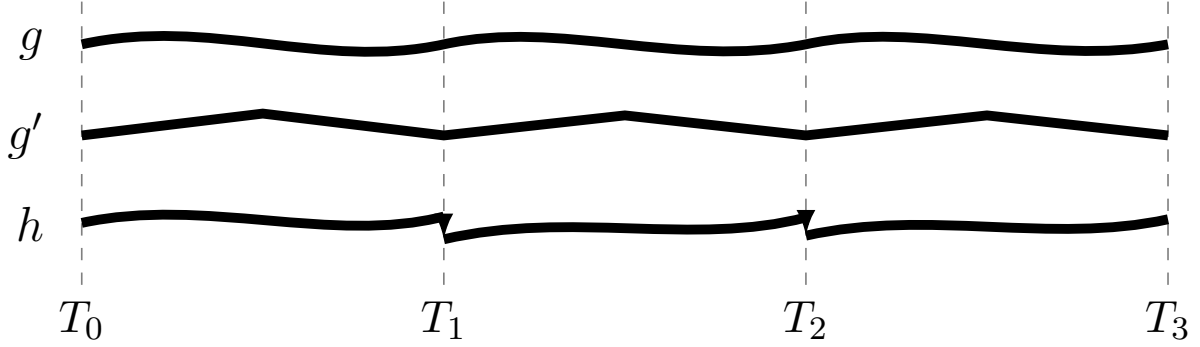
\begin{figure}[htbp]
\centering
\resizebox{0.98\textwidth}{!}{%
\begin{tikzpicture}[x=1cm,y=0.78cm,font=\small]
  \foreach \x/\lab in {0/$T_0$,3.10/$T_1$,6.20/$T_2$,9.30/$T_3$}{
    \draw[dashed,gray] (\x,-.72)--(\x,2.48);
    \node[below] at (\x,-.72) {\lab};
  }
  \node[anchor=east,font=\normalsize] at (-.18,1.92) {$g$};
  \node[anchor=east,font=\normalsize] at (-.18,.92) {$g'$};
  \node[anchor=east,font=\normalsize] at (-.18,-.05) {$h$};

 \draw[line width=2.4pt]
    (0,1.92) .. controls (1.10,2.23) and (2.05,1.63) .. (3.10,1.92)
             .. controls (4.10,2.20) and (5.10,1.65) .. (6.20,1.92)
             .. controls (7.20,2.20) and (8.20,1.66) .. (9.30,1.92);

  \draw[line width=2.4pt]
    (0,.92)--(1.55,1.16)--(3.10,.92)--(4.65,1.14)--(6.20,.92)
             --(7.75,1.14)--(9.30,.92);

  \draw[line width=2.4pt]
    (0,-.04) .. controls (1.05,.25) and (2.15,-.28) .. (3.10,.03);
  \draw[line width=2.4pt]
    (3.10,-.22) .. controls (4.18,.10) and (5.25,-.30) .. (6.20,.02);
  \draw[line width=2.4pt]
    (6.20,-.18) .. controls (7.25,.12) and (8.30,-.28) .. (9.30,.00);
  \draw[-{Latex[length=2.2mm]},line width=1.2pt]
    (3.10,.03)--(3.10,-.22);
  \draw[-{Latex[length=2.2mm]},line width=1.2pt]
    (6.20,.02)--(6.20,-.18);
\end{tikzpicture}%
}
\caption{The original flow \(g\), its continuous stagewise DeTurck pullback
\(g'\), and the invariant comparison family \(h\).  The family \(h\) may
jump at a switching time because the parabolically rescaled endpoint is
averaged before the next stage.  The drawing is schematic and not to scale;
vertical placement only separates the three families and has no metric
meaning.}
\label{fig:three-metric-families}
\end{figure}

\subsection{Rescaled blocks and admissible stage starts}
All unqualified H\"older norms in this section are the intrinsic lifted
norms defined in Paragraph~\ref{par:lifted-norm-convention}, with the displayed
reference metric.  Fixed-coordinate uniformly local norms are used only inside
the Schauder and bootstrapping arguments, after a \(\PBG\) bound has been
established; Lemma~\ref{lem:intrinsic-fixed-chart-equivalence} then converts
between the two conventions with constants independent of collapse.  When the
reference metric is not invariant,
Lemma~\ref{lem:universal-cover-inj} provides the lifted injectivity-radius bound
through the invariant bilipschitz comparison metric in the hypotheses.  Let
\(\{g(t)\}_{t\geq0}\) be the immortal Ricci flow under consideration.

\begin{definition}[Parabolically rescaled large block]
\label{def:parabolically-rescaled-large-block}
Fix \(S>0\), a diffeomorphism \(\Phi\), and \(\delta_0\geq c_0\), where
\(c_0\) is the cutoff in the standing fixed-threshold application.  Set
\begin{equation}
        p(\tau)=\Phi^*\bigl(S^{-1}g(S\tau)\bigr),
        \qquad \frac12\leq\tau\leq101.
\end{equation}
We call \(p(\tau)\) a \emph{parabolically rescaled
\(\delta_0\)-large block starting at \(S\)} if \(S/2\geq T_B(c_0)\) and
\[
        \diam(M,g(S\tau))\geq\delta_0\sqrt{S\tau},
        \qquad 1\leq\tau\leq100.
\]
\end{definition}

\begin{definition}[Admissible stage start]
\label{def:admissible-stage-start}
Fix a torus structure and constants \(\Lambda,\eta,\theta>0\).  We call the
pair \((q_0,r_0)\), where \(r_0=\overline{q_0}\), an \emph{admissible stage
start with constants \(\Lambda,\eta,\theta\)} if \(r_0\) is Riemannian and
\begin{equation}\label{eq:admissible-stage-start}
 \BG_N(r_0;\Lambda),\qquad
 \|q_0-r_0\|_{C^{N,\sigma}(M,r_0)}<\eta,\qquad
 d_{\TT^2}(r_0)<\theta.
\end{equation}
The Riemannian condition and the first two displayed conditions are the
\emph{analytic start conditions}; the last displayed condition is the
\emph{geometric start condition}.  Stage preparation and positive-time
smoothing reproduce bounded geometry and control the endpoint error;
finite-time decay and averaging contract that error, while fibration comparison
reproduces the geometric condition.  No quotient-length condition is part of
an admissible stage start; the required lower bound is reconstructed from the
large-diameter hypothesis inside each block.
\end{definition}

\phantomsection\label{par:rescaled-bamler-comparison}
On a block in the sense of
Definition~\ref{def:parabolically-rescaled-large-block},
\begin{equation}
        |\sec_{p(\tau)}|\leq K/\tau,
        \qquad \frac12\leq\tau\leq101.
\end{equation}
For each \(\tau\in[1,100]\), Bamler's large-diameter case provides, after
pullback by a diffeomorphism, a torus fibration
\(\pi_B(\tau):M\to S_B^1(\tau)\) and an invariant comparison metric
\(h_B(\tau)\) satisfying
\[
 (1+\eps_B^{c_0}(S\tau))^{-1}h_B(\tau)\leq p(\tau)
 \leq(1+\eps_B^{c_0}(S\tau))h_B(\tau),
\]
whose fibers have \(p(\tau)\)-diameter at most
\(\eps_B^{c_0}(S\tau)\sqrt\tau\).  After increasing the starting time, the same
comparison is as small as required in every fixed \(C^k\)-norm used below.
The pair \((\pi_B(\tau),h_B(\tau))\) may vary arbitrarily with \(\tau\), and
all these statements remain true after any further time-dependent pullback of
the block.

\subsection{Averaged endpoints and positive-time regularization}

The next lemma uses the following finite-regularity scheme for the high-order restart:
\begin{equation}
\begin{gathered}
 \vphantom{C^{N+5,\sigma}}\BG_N(q;\Lambda),\qquad m_0\ \text{invariant},\qquad
 \|q-m_0\|_{C^{N+5,\sigma}(M,q)}<\varepsilon_{\mathrm{av}},\\
 \Downarrow\\
 |\widetilde\nabla^j\Rm_{\widetilde m_0}|,
 |\widetilde\nabla^j\Rm_{\widetilde{\overline q}}|\leq C
 \quad(0\leq j\leq N+3),\\
 \Downarrow\\
 \BG_N(m_0;\Lambda'),\qquad \BG_N(\overline q;\Lambda'),\\
 \Downarrow\\
 \text{the bounded-geometry input for the next initial-face estimate is recovered.}
\end{gathered}
\end{equation}
Here \(\varepsilon_{\mathrm{av}}\) is from
Lemma~\ref{lem:high-order-averaged-slice} with \(m=N\).
The loss of two derivatives from the metric to curvature explains the passage
from \(N+5\) metric derivatives to curvature derivatives through order
\(N+3\).  The \(C^{N+5,\sigma}\)-control is not assumed at the old initial
face; it is recovered at positive time by the smoothing proposition below.

\begin{lemma}[High-order regularity of an averaged slice]
\label{lem:high-order-averaged-slice}
Fix an integer \(m\geq2\), \(\sigma\in(0,1)\), and \(\Lambda_0<\infty\).  There
are constants \(\varepsilon_{\mathrm{av}}>0\), \(C_{\mathrm{av}}<\infty\), and
\(\Lambda_1<\infty\), depending only on these data, with the following property.
Let \(q\) and \(m_0\) be metrics on a torus bundle, suppose that \(m_0\) is
invariant, and let \(\overline q\) denote the average of \(q\) with respect to this
torus structure.  If \(\BG_m(q;\Lambda_0)\) holds and
\begin{equation}
 \|q-m_0\|_{C^{m+5,\sigma}(M,q)}<\varepsilon_{\mathrm{av}},
\end{equation}
then \(m_0\) and \(\overline q\) are Riemannian and satisfy
\(\BG_m(m_0;\Lambda_1)\) and \(\BG_m(\overline q;\Lambda_1)\).  Moreover,
\begin{equation}\label{eq:high-order-average-output}
 \|\overline q-m_0\|_{C^{m+5,\sigma}(M,m_0)}
 \leq C_{\mathrm{av}}\|q-m_0\|_{C^{m+5,\sigma}(M,q)}.
\end{equation}
\end{lemma}

\begin{proof}
For sufficiently small \(\varepsilon_{\mathrm{av}}\), the two metrics are
uniformly equivalent and their finite-order tensor H\"older norms are
uniformly equivalent.  The standard formulas comparing two Levi-Civita
connections, followed by the formulas for curvature and its covariant
derivatives, transfer the bounds for
\(\widetilde\nabla^j\Rm_{\widetilde q}\), \(0\leq j\leq m+3\), to
\(\widetilde m_0\).  Since \(m_0\) is invariant,
Lemma~\ref{lem:universal-cover-inj} then gives a uniform injectivity-radius
lower bound for \(\widetilde m_0\).

Because \(m_0\) is fixed by averaging,
\[
        \overline q-m_0=\overline{q-m_0}.
\]
Lemma~\ref{lem:averaging-invariant-holder}, together with equivalence of the
\(q\)- and \(m_0\)-based norms, proves
\eqref{eq:high-order-average-output}.  Decreasing
\(\varepsilon_{\mathrm{av}}\) preserves positive definiteness.  The same
connection and curvature comparison formulas give the required curvature-derivative bounds for \(\overline q\); this metric is invariant, so
Lemma~\ref{lem:universal-cover-inj} gives its lifted injectivity-radius bound.
Enlarging a single constant gives \(\Lambda_1\).
\end{proof}

The next proposition states the positive-time smoothing used at every
switching time.  Informally, its hypotheses, smoothing step, and conclusion are as follows:
\begin{center}
\small
\renewcommand{\arraystretch}{1.12}
\begin{tabular}{@{}>{\raggedright\arraybackslash}p{0.18\textwidth}>{\raggedright\arraybackslash}p{0.72\textwidth}@{}}
\textbf{Hypotheses} & A curvature-controlled invariant flow \(r\) and a Ricci--DeTurck flow \(q\) that is close to \(r\) in \(X_r^N[1,100]\).\\
\textbf{Smoothing} & Interior regularization raises the difference estimate from order \(N\) to order \(N+5\) near \(\tau=100\).\\
\textbf{Conclusion} & Rescale by \(100^{-1}\), average, and obtain a \(C^{N+5,\sigma}\)-small restart with a uniform lifted-geometry bound.\\
\end{tabular}
\end{center}
This argument propagates the high-order initial-face regularity
needed for the next Ricci--DeTurck stage.

\begin{proposition}[Positive-time smoothing and regular restart]
\label{prop:positive-time-regular-restart}
Fix \(K_0<\infty\).  There are constants
\[
        \beta_{\mathrm{reg}}>0,\qquad C_{\mathrm{reg}}<\infty,
        \qquad \Lambda_{\mathrm{reg}}<\infty,
\]
depending only on \(K_0,N,\sigma\), with the following property.

\medskip
\noindent\textbf{Hypotheses.}
Let \(r(\tau)\), \(1\leq\tau\leq100\), be an invariant Ricci flow satisfying
\begin{equation}\label{eq:regular-restart-curvature}
        |\sec_{r(\tau)}|\leq K_0/\tau.
\end{equation}
Let \(q(\tau)\) be a Ricci--DeTurck flow relative to \(r(\tau)\) and assume
\begin{equation}\label{eq:regular-restart-low-order-closeness}
        \|q-r\|_{X_r^N[1,100]}\leq\beta_{\mathrm{reg}}.
\end{equation}

\medskip
\noindent\textbf{Rescale and average.}
Put
\[
        q^+(1)=100^{-1}q(100),\qquad
        \rho^+=100^{-1}r(100),\qquad
        r^+(1)=\overline{q^+(1)}.
\]

\medskip
\noindent\textbf{Conclusion.}
Then \(r^+(1)\) is Riemannian,
\[
        \BG_N(\rho^+;\Lambda_{\mathrm{reg}}),\qquad
        \BG_N(r^+(1);\Lambda_{\mathrm{reg}}),
\]
and
\begin{equation}\label{eq:regular-restart-high-order-estimate}
\begin{split}
 &\|q^+(1)-\rho^+\|_{C^{N+5,\sigma}(M,\rho^+)}
 +\|r^+(1)-\rho^+\|_{C^{N+5,\sigma}(M,\rho^+)}\\
 &\qquad\leq C_{\mathrm{reg}}
 \|q-r\|_{C^0(M\times[1,100],r)}.
\end{split}
\end{equation}
\end{proposition}

The proof of Proposition~\ref{prop:positive-time-regular-restart} is given in
Subsection~\ref{subsec:appendix-continuation-restart}.

\subsection{Stage preparation and stability}

The next lemma prepares one analytic stage from the analytic part of an
admissible stage start.  It averages the initial metric, evolves the average by
Ricci flow, and places the rescaled original flow in Ricci--DeTurck gauge
relative to that invariant flow.  Its additional high-order input is reproduced
at the rescaled endpoint by
Proposition~\ref{prop:positive-time-regular-restart}; no short-fiber hypothesis
is used in this analytic lemma.

\begin{lemma}[Stage preparation, stability, and regular restart]
\label{lem:stage-stability}
There are constants
\[
        K_{\mathrm{st}}<\infty,\qquad \Lambda_*<\infty,
        \qquad \beta_{\mathrm{reg}}>0,\qquad C_{\mathrm{st}}<\infty,
\]
depending only on \(K,N,\sigma\), with the following property.  For every
\(0<\beta_*\leq\beta_{\mathrm{reg}}\), there is a constant
\(\eta_{\mathrm{st}}>0\), depending only on \(K,N,\sigma,\beta_*\), such that the following holds.
Let \(p(\tau)\) be any parabolically rescaled \(\delta_0\)-large block, for
some \(\delta_0\geq c_0\), under the standing fixed-threshold application.
Fix the torus structure used during this
stage and put
\[
        q(1)=p(1),\qquad r(1)=\overline{q(1)}.
\]
Put
\[
        b:=\|q(1)-r(1)\|_{C^{N,\sigma}(M,r(1))}.
\]
Assume that \(r(1)\) is Riemannian,
\(\BG_N(r(1);\Lambda_*)\) holds, and \(b<\eta_{\mathrm{st}}\).  These are
exactly the analytic start conditions in
Definition~\ref{def:admissible-stage-start}, with
\(\Lambda=\Lambda_*\) and \(\eta=\eta_{\mathrm{st}}\).
Then the invariant Ricci flow \(r(\tau)\) and the Ricci--DeTurck flow
\(q(\tau)\) relative to it exist on \([1,100]\), and
\begin{align}
        |\sec_{r(\tau)}|&\leq K_{\mathrm{st}}/\tau,
                &&1\leq\tau\leq100,                               \\
        \|q-r\|_{X_r^N[1,100]}
                &<\beta_*,                                         \label{eq:stage-hypotheses}\\
        \|q-r\|_{X_r^N[1,100]}
                &\leq C_{\mathrm{st}}b.                                 \label{eq:full-stability-linear}
\end{align}
If
\[
        q^+(1)=100^{-1}q(100),\qquad
        \rho^+=100^{-1}r(100),\qquad
        r^+(1)=\overline{q^+(1)},
\]
then \(r^+(1)\) is Riemannian, \(\BG_N(r^+(1);\Lambda_*)\) holds,
\(|\sec_{r^+(1)}|\leq K_{\mathrm{st}}\), and
\begin{equation}\label{eq:stage-high-order-restart}
\begin{split}
 &\|q^+(1)-\rho^+\|_{C^{N+5,\sigma}(M,\rho^+)}
 +\|r^+(1)-\rho^+\|_{C^{N+5,\sigma}(M,\rho^+)}\\
 &\qquad\leq C_{\mathrm{st}}b.
\end{split}
\end{equation}
In particular,
\begin{equation}\label{eq:averaged-endpoint-close}
        \|r^+(1)-\rho^+\|_{C^{N,\sigma}(M,\rho^+)}
        \leq C_{\mathrm{st}}b.
\end{equation}
\end{lemma}

Before the proof, it is useful to display the two distinct gauges used in the
construction:
\begin{equation}\label{eq:two-deturck-gauges}
\begin{array}{ccccc}
 r(1)&\xrightarrow[\text{background }p]{\mathrm{Ricci\text{--}DeTurck}}
 &r_{\mathrm D}&\xrightarrow{\text{DeTurck diffeomorphisms}}&r,\\[3pt]
 p(1)&\xrightarrow[\text{background }r]{\mathrm{Ricci\text{--}DeTurck}}
 &q&\xrightarrow{\text{DeTurck diffeomorphisms}}&p.
\end{array}
\end{equation}
The first line of \eqref{eq:two-deturck-gauges} constructs the comparison
Ricci flow \(r\), proves its curvature control, and verifies that it remains
invariant.  Once that background is available, the second line gauges the
original rescaled flow relative to \(r\) and produces the stage metric \(q\).
The auxiliary metric \(r_{\mathrm D}\) is not part of the inductive state.

\begin{proof}
We first choose constants that do not depend on the requested stage closeness
\(\beta_*\).  The geometry of every rescaled block \(p\) is bounded uniformly
through order \(N\) by the argument in Step~1 below.  Let
\(\eta_1^{(0)},C_1^{(0)}\) be the uniform constants in
Proposition~\ref{prop:uniform-lifted-rdt-stability} for the reference flow \(p\),
relabeled from \(\eta_{\mathrm{cd}}\) and \(C_{\mathrm{cd}}\).  The standard
comparison formulas for nearby metrics and their Levi-Civita connections give
uniform constants \(b_{\mathrm{eq}}>0\) and \(C_{\mathrm{eq}}\geq1\) such that
\(b<b_{\mathrm{eq}}\) implies
\[
 \|r(1)-p(1)\|_{C^{N,\sigma}(M,p(1))}\leq C_{\mathrm{eq}}b.
\]
Set
\[
 \eta_1:=\min\{b_{\mathrm{eq}},\eta_1^{(0)}/C_{\mathrm{eq}}\},
 \qquad C_1:=C_1^{(0)}C_{\mathrm{eq}}.
\]
Restrict the curvature comparison to the neighborhood
\[
 \|w-p(\tau)\|_{C^2(M,p(\tau))}\leq\tfrac12.
\]
The \(C^0\)-part then gives
\(\tfrac12p(\tau)\leq w\leq\tfrac32p(\tau)\), so the inverse of \(w\)
is uniformly controlled relative to \(p(\tau)\).  The curvature comparison
formula and the uniformly bounded reference geometry give a constant
\(C_{\mathrm{curv}}^{(0)}\) such that, on this neighborhood,
\[
 \sup_M|\sec_w|\leq \sup_M|\sec_{p(\tau)}|
       +C_{\mathrm{curv}}^{(0)}\|w-p(\tau)\|_{C^2(M,p(\tau))}.
\]
Set \(C_{\mathrm{curv}}=C_{\mathrm{curv}}^{(0)}C_1\).  Thus
\(\|w-p(\tau)\|_{C^2(M,p(\tau))}\leq C_1b\leq\tfrac12\) implies
\[
 \sup_M|\sec_w|\leq \sup_M|\sec_{p(\tau)}|+C_{\mathrm{curv}}b.
\]
These constants depend only on the fixed reference bounds, not on the stage
threshold to be chosen below.  Put
\[
        K_0:=K+100C_{\mathrm{curv}},
        \qquad K_{\mathrm{st}}:=2K_0+1.
\]
Apply the positive-time smoothing result,
Proposition~\ref{prop:positive-time-regular-restart}, with curvature constant
\(K_{\mathrm{st}}\), and denote its constants by
\(\beta_{\mathrm{reg}},C_{\mathrm{reg}},\Lambda_{\mathrm{reg}}\).

We next fix the geometry constant needed at the first stage.  There is a
constant \(\Lambda_{\mathrm{III}}=\Lambda_{\mathrm{III}}(K,N,\sigma)\) such that every sufficiently
late rescaled slice
\(q_0=T^{-1}g(T)\), written in a Bamler torus-bundle gauge, satisfies
\(\BG_N(q_0;\Lambda_{\mathrm{III}})\).  Indeed, the
Type-III estimate on \([T/2,T]\) and Shi's estimates give the required
curvature-derivative bounds after rescaling.  The Bamler metric is invariant
and uniformly bilipschitz to \(q_0\), so
Lemma~\ref{lem:universal-cover-inj} gives the lifted injectivity-radius bound.
Apply Lemma~\ref{lem:high-order-averaged-slice} with
\(\Lambda_0=\Lambda_{\mathrm{III}}\), and call its output geometry constant
\(\Lambda_{\mathrm{init}}\).  Set
\[
        \Lambda_*:=\max\{\Lambda_{\mathrm{reg}},\Lambda_{\mathrm{init}}\}.
\]

Let \(s(\tau)\), \(1\leq\tau\leq100\), be any reference Ricci flow
whose initial metric satisfies \(\BG_N(s(1);\Lambda_*)\) and whose sectional
curvature is bounded by \(K_{\mathrm{st}}\).  Applying
Lemma~\ref{lem:slice-to-parabolic-geometry} with \(m=N\), \(\Lambda_0=\Lambda_*\),
\(K_0=K_{\mathrm{st}}\), and \(L=99\) gives a constant
\(\Lambda_{\mathrm{par}}=\Lambda_{\mathrm{par}}(K_{\mathrm{st}},\Lambda_*,N,\sigma)\) such that
\[
        \PBG_N(s;\Lambda_{\mathrm{par}};[1,100]).
\]
In particular, the initial curvature-derivative bounds, their propagation near
\(\tau=1\), the later Shi estimates, the lifted injectivity-radius bound, and
the resulting parabolic coordinate control are all uniform at every stage.
Let \(\eta_2,C_2\) be the corresponding uniform constants in
Proposition~\ref{prop:uniform-lifted-rdt-stability}.
Choose
\[
        C_{\mathrm{st}}\geq\max\{C_2,C_{\mathrm{reg}}C_2,1\}.
\]
After fixing \(0<\beta_*\leq\beta_{\mathrm{reg}}\), choose
\(\eta_{\mathrm{st}}>0\) no larger than \(\eta_1,\eta_2\), and sufficiently small
that
\begin{equation}\label{eq:stage-eta-choice}
 C_{\mathrm{st}}\eta_{\mathrm{st}}<\beta_*,
 \qquad C_1\eta_{\mathrm{st}}\leq\tfrac12,
 \qquad \eta_{\mathrm{st}}<1,
\end{equation}
and that the endpoint positivity and curvature perturbations used in Step~3
below are valid.  All constants in the statement have now been fixed.

\proofstep{Step 1. First gauge: construct and control the invariant flow.}

For every \(\tau\in[1,100]\), Lemma~\ref{lem:universal-cover-inj} applies to
\(g=p(\tau)\) and \(h=h_B(\tau)\), with \(B=2\) and
\(\Lambda=\max\{K,1\}/\tau\).  Hence
\[
        \inj_{\widetilde p(\tau)}
        \geq i_2\sqrt{\frac{\tau}{\max\{K,1\}}}
        \geq \frac{i_2}{\sqrt{\max\{K,1\}}}.
\]
The curvature bound on \([1/2,101]\) and Shi's estimates \cite{Shi} give the
curvature-derivative bounds through order \(N+3\) on \([1,100]\).  Thus
\(\PBG_N(p;\Lambda_p;[1,100])\) holds for a uniform constant \(\Lambda_p\).

Let \(r_{\mathrm{D}}\) be the Ricci--DeTurck solution with background \(p\) and
initial metric \(r(1)\).  By the definitions of \(\eta_1\) and \(C_1\), the
bound \(b<\eta_{\mathrm{st}}\leq\eta_1\) implies the initial-data hypothesis of
Proposition~\ref{prop:uniform-lifted-rdt-stability}, with the norm measured
relative to \(p(1)\), and gives
\begin{equation}\label{eq:first-stage-continuous-dependence}
        \|r_{\mathrm{D}}-p\|_{X_p^N[1,100]}\leq C_1b.
\end{equation}
The DeTurck diffeomorphisms convert \(r_{\mathrm{D}}\) into the ordinary Ricci
flow \(r\) starting from \(r(1)\).  By
\eqref{eq:stage-eta-choice}, \(C_1b<C_1\eta_{\mathrm{st}}\leq\tfrac12\).
Thus the \(C^2\)-part of \eqref{eq:first-stage-continuous-dependence}
places \(r_{\mathrm D}(\tau)\) in the curvature-comparison neighborhood above,
and the curvature perturbation formula gives
\[
\begin{aligned}
        \sup_M|\sec_{r(\tau)}|
        &=\sup_M|\sec_{r_{\mathrm{D}}(\tau)}|\\
        &\leq \frac{K}{\tau}+C_{\mathrm{curv}}b
        \leq \frac{K_0}{\tau},
        \qquad 1\leq\tau\leq100.
\end{aligned}
\]
Lift \(r\) to the cyclic cover \(M_{\mathrm{cyc}}\).  There the local fiber
translations form a global \(\TT^2\)-action, say \(R_a\), and
\(R_a^*r_{\mathrm{cyc}}(1)=r_{\mathrm{cyc}}(1)\).  For each \(a\in\TT^2\),
\(R_a^*r_{\mathrm{cyc}}(\tau)\) and \(r_{\mathrm{cyc}}(\tau)\) are complete
bounded-curvature Ricci flows with the same initial metric; completeness and
the curvature bound follow by lifting the smooth flow on the compact quotient.
The uniqueness theorem for complete noncompact bounded-curvature Ricci flows
\cite[Theorem~1.1]{ChenZhu} therefore gives
\(R_a^*r_{\mathrm{cyc}}(\tau)=r_{\mathrm{cyc}}(\tau)\) for
\(1\leq\tau\leq100\).  Hence \(r(\tau)\) remains invariant.

\proofstep{Step 2. Second gauge: compare the original block with the invariant flow.}

The analytic start conditions include \(\BG_N(r(1);\Lambda_*)\).  Together
with the curvature estimate just proved,
Lemma~\ref{lem:slice-to-parabolic-geometry} gives
\[
        \PBG_N(r;\Lambda_{\mathrm{par}};[1,100]).
\]
This is the high-order inductive input that was not supplied by the low-order
start error alone.

Apply Proposition~\ref{prop:uniform-lifted-rdt-stability} with reference flow
\(s=r\), and let \(q\) be the resulting Ricci--DeTurck flow with
\(q(1)=p(1)\).  Since \(b<\eta_{\mathrm{st}}\leq\eta_2\),
\begin{equation}\label{eq:second-stage-continuous-dependence}
        \|q-r\|_{X_r^N[1,100]}\leq C_2b\leq C_{\mathrm{st}}b.
\end{equation}
By \eqref{eq:stage-eta-choice}, this is smaller than \(\beta_*\), proving
\eqref{eq:stage-hypotheses} and \eqref{eq:full-stability-linear}.  The
ordinary Ricci flow obtained from \(q\) by the DeTurck diffeomorphisms starts
from \(p(1)\), so uniqueness identifies it with \(p\).

\proofstep{Step 3. Smooth, rescale, and average the endpoint.}

Since \(\|q-r\|_{X_r^N}\leq\beta_*\leq\beta_{\mathrm{reg}}\),
Proposition~\ref{prop:positive-time-regular-restart} applies to \(q\) and
\(r\).  Combining \eqref{eq:regular-restart-high-order-estimate} with
\eqref{eq:second-stage-continuous-dependence} gives
\eqref{eq:stage-high-order-restart}.  The same proposition shows that \(r^+(1)\) is Riemannian and satisfies
\(\BG_N(r^+(1);\Lambda_{\mathrm{reg}})\), hence also \(\BG_N(r^+(1);\Lambda_*)\).
Equation \eqref{eq:averaged-endpoint-close} is the lower-order part of
\eqref{eq:stage-high-order-restart}.

Finally,
\[
        |\sec_{\rho^+}|=100|\sec_{r(100)}|\leq K_0.
\]
The \(C^2\)-part of \eqref{eq:stage-high-order-restart} changes sectional
curvature by at most \(CC_{\mathrm{st}}b\).  The last smallness requirement in the
choice of \(\eta_{\mathrm{st}}\), together with
\(K_{\mathrm{st}}>2K_0\), gives
\(|\sec_{r^+(1)}|\leq K_{\mathrm{st}}\).  Thus the endpoint reproduces the
high-order initial-face hypothesis needed for the next stage.
\end{proof}

Thus one analytic stage can be constructed with constants that are uniform
under collapse: averaging produces an invariant reference flow, the original
block remains controlled in Ricci--DeTurck gauge, and the averaged rescaled
endpoint stays in the same stability neighborhood.

\section{One-block contraction and stagewise families}
\label{sec:deturck-iteration}

This is the central inductive section of the paper.  It combines the analytic
stage estimate with the fixed-fibration geometry to produce a block map that
reproduces its hypotheses while reducing the non-invariant error by a definite
factor.  The quotient-circle lower bound is not carried as an inductive
coordinate: the large-diameter hypothesis and fibration comparison reconstruct
it inside every block.

The first subsection, \emph{Admissible-start data}, records the unified
start conditions from Definition~\ref{def:admissible-stage-start} and the two
numerical state variables used in the block map.  The second,
\emph{Persistence of the retained fibration},
uses the pointwise fixed-fiber improvement in a first-hitting argument to keep
the retained fibers uniformly short and to recover a positive quotient-length
lower bound throughout the block.  The third, \emph{Contraction on one
large-diameter block}, combines that persistence with stage stability and
finite-time decay, first in a concise block contraction theorem and then in a
quantitative proposition.  The final subsection, \emph{Stagewise families and
switching times}, records the stagewise metric families and switching-time
notation used for iteration on a connected large-diameter interval in
Section~\ref{sec:quotient-length}.

\subsection{Admissible-start data}
For a stage start \((q_k(1),r_k(1))\), with
\(r_k(1)=\overline{q_k(1)}\), write
\begin{equation}\label{eq:stage-start-data}
 b_k=\|q_k(1)-r_k(1)\|_{C^{N,\sigma}(M,r_k(1))},
 \qquad
 d_k=d_{\TT^2}(r_k(1)).
\end{equation}
Thus the numerical start data consist only of the non-invariant error and the
retained-fiber diameter, while lifted bounded geometry is recorded separately.
By Definition~\ref{def:admissible-stage-start}, the pair is an admissible
stage start with constants
\(\Lambda_*,\eta_{\mathrm{blk}},\theta_{\mathrm{blk}}\) precisely when
\[
 b_k<\eta_{\mathrm{blk}},\qquad
 d_k<\theta_{\mathrm{blk}},\qquad
 \BG_N(r_k(1);\Lambda_*).
\]
The block theorem below proves the direct implication
\begin{equation}\label{eq:one-block-conceptual-form}
 \left.
 \begin{gathered}
  b_k<\eta_{\mathrm{blk}},\quad d_k<\theta_{\mathrm{blk}},\\
  \BG_N(r_k(1);\Lambda_*)
 \end{gathered}
 \right\}
 \quad\Longrightarrow\quad
 \left\{
 \begin{gathered}
  b_{k+1}\leq\tfrac12 b_k,\quad d_{k+1}<\theta_{\mathrm{blk}},\\
  \BG_N(r_{k+1}(1);\Lambda_*).
 \end{gathered}
 \right.
\end{equation}
No quotient-length condition is required at any stage start.

\subsection{Persistence of the retained fibration}

\begin{lemma}[Persistence of the retained fibration on a large block]
\label{lem:retained-fibration-persistence}
Fix \(\delta_0>0\), \(0<\lambda_{\mathrm{p}}<\delta_0/10\), and
\(\theta_{\max}>0\).  Let
\[
 \theta,\quad \beta_{\mathrm{fib}},\quad \eps_{\mathrm{fib}}
\]
be any fixed choice of constants for which
Lemma~\ref{lem:pointwise-fixed-fiber} holds with these parameters.

Let \(p(\tau)\), \(1\leq\tau\leq100\), be a parabolically rescaled
\(\delta_0\)-large block, and let \(q(\tau)=\Psi_\tau^*p(\tau)\) for a smooth
family of diffeomorphisms.  Let \(r(\tau)\) be a smooth family of invariant metrics for one fixed fibration
\(\pi:M\to S^1\).  Suppose
\begin{equation}
 d_{\TT^2}(r(1))<\theta
\end{equation}
and, for every \(\tau\in[1,100]\),
\begin{equation}\label{eq:persistence-stage-closeness}
 \|q(\tau)-r(\tau)\|_{C^1(M,r(\tau))}<\beta_{\mathrm{fib}}.
\end{equation}
Suppose also that at every such \(\tau\) there exist an invariant metric \(h_B(\tau)\) and a torus fibration
\(\pi_B(\tau):M\to S_B^1(\tau)\) such that, for every
\(\pi_B(\tau)\)-fiber \(F_B\),
\begin{equation}\label{eq:persistence-bamler-comparison}
 \|h_B(\tau)-q(\tau)\|_{C^1(M,q(\tau))}<\eps_{\mathrm{fib}},
 \qquad
 \diam_{q(\tau)}F_B\leq\eps_{\mathrm{fib}}.
\end{equation}
Then, throughout the block,
\begin{align}
 d_{\TT^2}(r(\tau))&<\theta,                                    \label{eq:persistence-fiber-bound}\\
 L(S^1,r_{S^1}(\tau))&\geq\lambda_{\mathrm{p}}.                     \label{eq:persistence-quotient-bound}
\end{align}
Moreover, every \(\pi_B(\tau)\) is fiber-equivalent to \(\pi\).
\end{lemma}

\begin{proof}
The function \(\tau\mapsto d_{\TT^2}(r(\tau))\) is continuous.  If it
reached \(\theta\), let \(\tau_*\) be the first such parameter.  Since
\(q(\tau)=\Psi_\tau^*p(\tau)\) and the block is \(\delta_0\)-large,
\[
 \diam(M,q(\tau_*))=\diam(M,p(\tau_*))
 \geq\delta_0\sqrt{\tau_*}\geq\delta_0.
\]
At \(\tau_*\), equations \eqref{eq:persistence-stage-closeness} and
\eqref{eq:persistence-bamler-comparison}, together with
\(d_{\TT^2}(r(\tau_*))=\theta\), verify all hypotheses of
Lemma~\ref{lem:pointwise-fixed-fiber}.  That lemma gives
\[
 d_{\TT^2}(r(\tau_*))<\theta/2,
\]
a contradiction.  This proves \eqref{eq:persistence-fiber-bound}.

Now fix any \(\tau\in[1,100]\).  The same diameter estimate,
\eqref{eq:persistence-fiber-bound}, and the two comparison assumptions again
allow Lemma~\ref{lem:pointwise-fixed-fiber} to be applied at time \(\tau\).
Its quotient conclusion gives \eqref{eq:persistence-quotient-bound}, and its
fibration conclusion says that \(\pi_B(\tau)\) is fiber-equivalent to
\(\pi\).
\end{proof}

\subsection{Contraction on one large-diameter block}

\begin{theorem}[Block contraction theorem]
\label{thm:one-block-contraction}
Under the standing fixed-threshold application, fix \(\delta_0\geq c_0\).
There are constants
\[
        T_{\mathrm{blk}}<\infty,\qquad
        \eta_{\mathrm{blk}}>0,\qquad
        \theta_{\mathrm{blk}}>0,\qquad
        \Lambda_*<\infty
\]
with the following property.  The constants
\(\eta_{\mathrm{blk}},\theta_{\mathrm{blk}},\Lambda_*\) depend only on
\(K,\delta_0,N,\sigma\).  The start-time threshold \(T_{\mathrm{blk}}\) may
additionally depend on the fixed Bamler tail data \(T_B(c_0)\) and
\(\eps_B^{c_0}\) for the given flow.

Let \(p(\tau)\) be a parabolically rescaled \(\delta_0\)-large block starting
at \(S\geq T_{\mathrm{blk}}\), fix a torus fibration \(\pi:M\to S^1\), and put
\[
        q(1)=p(1),\qquad r(1)=\overline{q(1)}.
\]
Suppose that \(r(1)\) is Riemannian and, with
\[
 b=\|q(1)-r(1)\|_{C^{N,\sigma}(M,r(1))},
 \qquad d=d_{\TT^2}(r(1)),
\]
one has
\[
        b<\eta_{\mathrm{blk}},\qquad
        d<\theta_{\mathrm{blk}},\qquad
        \BG_N(r(1);\Lambda_*).
\]
Equivalently, \((q(1),r(1))\) is an admissible stage start in the sense of
Definition~\ref{def:admissible-stage-start}, with constants
\(\Lambda_*,\eta_{\mathrm{blk}},\theta_{\mathrm{blk}}\).
Then the invariant Ricci flow \(r\) and the Ricci--DeTurck flow \(q\) relative
to it exist on \([1,100]\).  At the rescaled endpoint, put
\[
        q^+(1)=100^{-1}q(100),\qquad
        \rho^+=100^{-1}r(100),\qquad
        r^+(1)=\overline{q^+(1)},
\]
and set
\[
        b^+=\|q^+(1)-r^+(1)\|_{C^{N,\sigma}(M,r^+(1))}.
\]
Then \(r^+(1)\) is Riemannian and
\begin{equation}\label{eq:one-block-theorem-conclusion}
        b^+\leq\tfrac12 b,\qquad
        d_{\TT^2}(r^+(1))<\theta_{\mathrm{blk}},\qquad
        \BG_N(r^+(1);\Lambda_*).
\end{equation}
Thus the same admissible-start conditions hold at the restart.  The ordinary
Ricci flow corresponding to \(q\) is \(p\).
\end{theorem}

\begin{proposition}[Quantitative estimates for one block]
\label{prop:one-block-quantitative}
The constants in Theorem~\ref{thm:one-block-contraction} can be chosen together
with \(C_{\mathrm{st}}<\infty\), depending only on
\(K,\delta_0,N,\sigma\), so that, under the hypotheses and in the notation of
that theorem,
\begin{align}
 \|q-r\|_{X_r^N[1,100]}
        &\leq C_{\mathrm{st}}b,                                      \label{eq:block-stage-closeness}\\
 d_{\TT^2}(r(\tau))&<\theta_{\mathrm{blk}},                          \label{eq:block-fiber-bound}\\
 L(S^1,r_{S^1}(\tau))&\geq\delta_0/100,
        \qquad 1\leq\tau\leq100.                              \label{eq:block-quotient-lower}
\end{align}
At the restart, \(r^+(1)\) is Riemannian,
\(\BG_N(r^+(1);\Lambda_*)\) holds, and
\begin{align}
        b^+&\leq\tfrac12 b,                                      \label{eq:block-contraction}\\
 \|r^+(1)-\rho^+\|_{C^{N,\sigma}(M,\rho^+)}
        &\leq C_{\mathrm{st}}b,                                      \label{eq:block-endpoint-jump}\\
        d_{\TT^2}(r^+(1))&<\theta_{\mathrm{blk}}.                    \label{eq:block-next-fibers}
\end{align}
In particular, these estimates imply the qualitative conclusion
\eqref{eq:one-block-theorem-conclusion}.
\end{proposition}

Figure~\ref{fig:one-block-contraction} summarizes the logical order of the
one-block construction used in the proof.

\begin{figure}[t]
\centering
\includegraphics[width=0.78\textwidth]{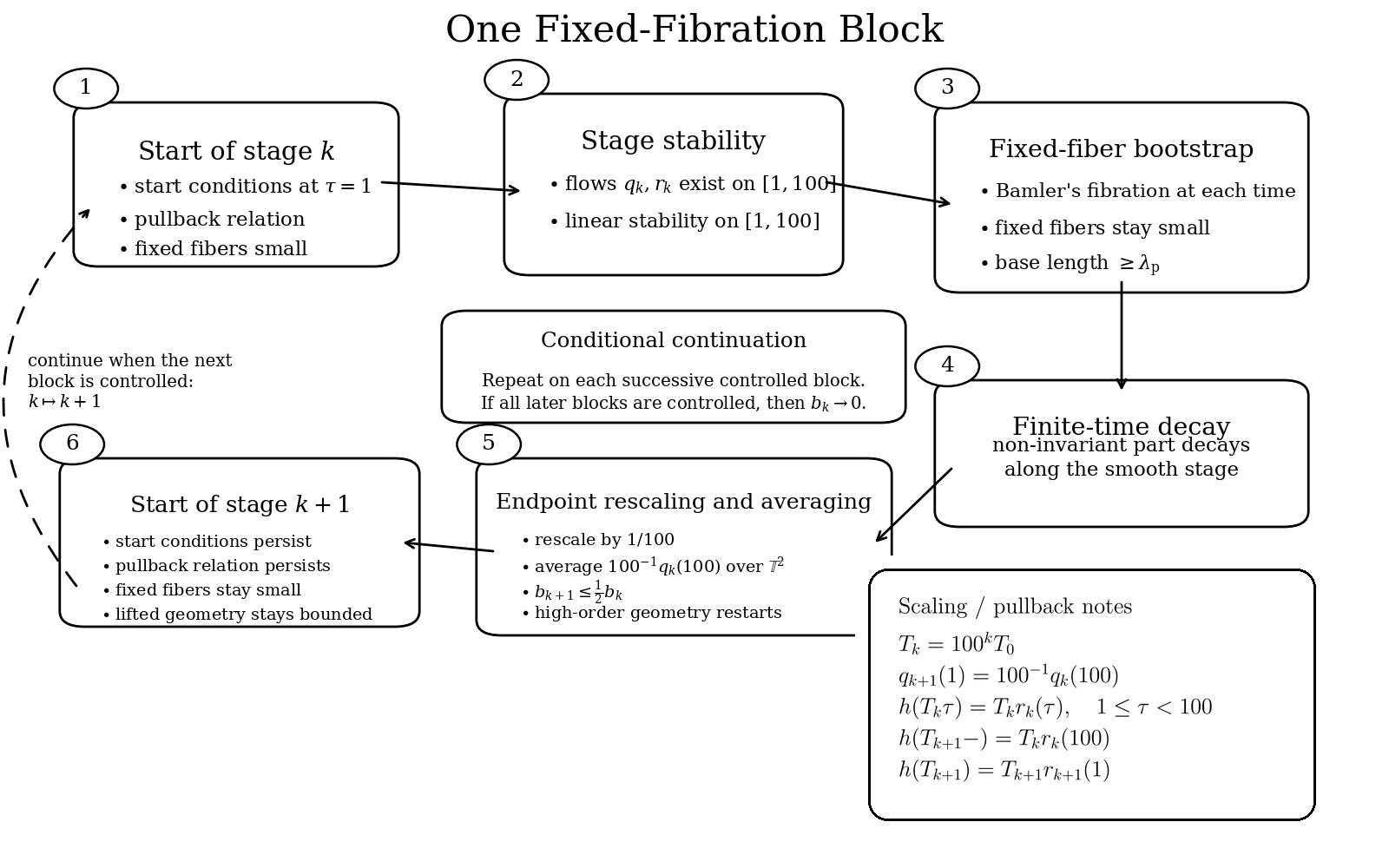}
\caption{One fixed-fibration block.  Stability and fibration comparison make
the finite-time decay estimate applicable; rescaling and averaging then
produce a contracted restart with the same hypotheses.}
\label{fig:one-block-contraction}
\end{figure}

\begin{proof}[Proof of Theorem~\ref{thm:one-block-contraction} and
Proposition~\ref{prop:one-block-quantitative}]
The proof has two steps.  We first choose the constants in their dependency
order.  We then apply stage stability, the retained-fibration lemma,
finite-time decay, and the restart estimate in that order.

\proofstep{Step 1. Choose the constants.}

The constants are selected in the order
\[
\begin{gathered}
 (K,\delta_0,N,\sigma)
 \longrightarrow
 (K_{\mathrm{st}},\Lambda_*,\beta_{\mathrm{reg}},C_{\mathrm{sc}})
 \longrightarrow \alpha_{\mathrm{blk}}
 \longrightarrow
 (\beta_{\mathrm{orb}}^{\mathrm{dec}},\beta_{\mathrm{pert}}^{\mathrm{dec}})\\[-1pt]
 \longrightarrow
 (\theta_{\mathrm{blk}},\beta_{\mathrm{fib}},\eps_{\mathrm{fib}})
 \longrightarrow \beta_*
 \longrightarrow
 (\eta_{\mathrm{blk}},T_{\mathrm{blk}}).
\end{gathered}
\]
The stage-stability constant \(C_{\mathrm{st}}\), which depends only on
\(K,N,\sigma\), is fixed together with
\(K_{\mathrm{st}},\Lambda_*\), and \(\beta_{\mathrm{reg}}\); it is omitted
from the chain because it is not itself a later smallness target.  The principal
dependency diagrams appear in Appendix~\ref{app:dependencies}.

Take \(K_{\mathrm{st}}\), \(\Lambda_*\), \(\beta_{\mathrm{reg}}\), and
\(C_{\mathrm{st}}\) from Lemma~\ref{lem:stage-stability}.  Choose
\(C_{\mathrm{sc}}\) so that, whenever \(q\) is sufficiently close to an
invariant metric \(r\),
\begin{equation}\label{eq:fixed-scaling-comparison}
 \|100^{-1}u\|_{C^{N,\sigma}(M,\overline{100^{-1}q})}
 \leq C_{\mathrm{sc}}\|u\|_{C^{N,\sigma}(M,r)},
\end{equation}
for every symmetric two-tensor \(u\).  This is the scaling rule for tensor
H\"older norms followed by equivalence of nearby reference metrics.  Choose
\[
        0<\alpha_{\mathrm{blk}}<(2C_{\mathrm{sc}})^{-1},
\]
and let
\[
 \beta_{\mathrm{orb}}^{\mathrm{dec}},\qquad
 \beta_{\mathrm{pert}}^{\mathrm{dec}}
\]
be the two thresholds supplied by
Theorem~\ref{thm:noninvariant-decay} for
\((\alpha_{\mathrm{blk}},K_{\mathrm{st}},\delta_0/100,N,\sigma)\).

Apply Lemma~\ref{lem:pointwise-fixed-fiber} with diameter lower bound
\(\delta_0\) and
\[
        \lambda_{\mathrm{p}}=\delta_0/100,
        \qquad
        \theta_{\max}=\beta_{\mathrm{orb}}^{\mathrm{dec}}/4.
\]
Denote its outputs by
\(\theta_{\mathrm{blk}},\beta_{\mathrm{fib}}\), and
\(\eps_{\mathrm{fib}}\), and choose
\[
        0<\beta_*\leq
        \min\{\beta_{\mathrm{pert}}^{\mathrm{dec}},
                    \beta_{\mathrm{fib}},\beta_{\mathrm{reg}}\}.
\]
Let \(\eta_{\mathrm{st}}\) be the stage threshold corresponding to this
\(\beta_*\) in Lemma~\ref{lem:stage-stability}.  Choose
\(\eta_{\mathrm{blk}}\leq\eta_{\mathrm{st}}\) below the finitely many
smallness thresholds needed for the endpoint norm comparison and restart:
\eqref{eq:fixed-scaling-comparison} must apply when the endpoint perturbation
has \(C^{N,\sigma}\)-size at most
\(C_{\mathrm{st}}\eta_{\mathrm{blk}}\), and a perturbation of this size must
carry fibers of diameter \(<\theta_{\mathrm{blk}}/10\) to fibers of diameter
\(<\theta_{\mathrm{blk}}\).

Finally, since \(\eps_B^{c_0}(t)\to0\) and Bamler's comparison converges in every
fixed order used below, choose \(T_{\mathrm{blk}}\geq2T_B(c_0)\) so large that
\[
        10\eps_B^{c_0}(T_{\mathrm{blk}})<\eps_{\mathrm{fib}},
\]
and so that on every block starting at \(S\geq T_{\mathrm{blk}}\), after any
simultaneous stage pullback, the Bamler metric comparison is within
\(\eps_{\mathrm{fib}}\) in \(C^1\) and lies in all the already fixed H\"older
and bilipschitz neighborhoods.

\proofstep{Step 2. Verify the stage hypotheses, retain the fibration, and contract.}

The parabolically rescaled block hypotheses supply the Type-III curvature bound
and an invariant small-fiber comparison at every slice.  Hence
Lemma~\ref{lem:stage-stability} applies and yields \(q\), \(r\),
\eqref{eq:block-stage-closeness}, and
\begin{equation}\label{eq:block-stage-beta-star}
        \|q(\tau)-r(\tau)\|_{C^1(M,r(\tau))}<\beta_*,
        \qquad 1\leq\tau\leq100.
\end{equation}

The DeTurck diffeomorphisms make \(q(\tau)\) a time-dependent pullback of
\(p(\tau)\).  Pull back Bamler's comparison metric and fibration by the same
diffeomorphisms, and denote the pulled-back objects again by
\(h_B(\tau)\), \(\pi_B(\tau)\), and \(F_B(\tau)\).  By the choice of
\(T_{\mathrm{blk}}\), for every \(\tau\in[1,100]\) and every
\(\pi_B(\tau)\)-fiber \(F_B(\tau)\),
\begin{equation}\label{eq:block-bamler-smallness}
        \|h_B(\tau)-q(\tau)\|_{C^1(M,q(\tau))}<\eps_{\mathrm{fib}},
        \qquad
        \diam_{q(\tau)}(F_B(\tau))
        \leq\eps_{\mathrm{fib}}.
\end{equation}
Indeed, since \(\eps_B^{c_0}\) is nonincreasing and \(1\leq\tau\leq100\),
Bamler's parabolically rescaled fiber estimate gives
\[
        \diam_{q(\tau)}(F_B(\tau))
        \leq\eps_B^{c_0}(S\tau)\sqrt\tau
        \leq10\eps_B^{c_0}(T_{\mathrm{blk}})<\eps_{\mathrm{fib}}.
\]

Apply Lemma~\ref{lem:retained-fibration-persistence} with
\[
 \lambda_{\mathrm{p}}=\delta_0/100,
 \qquad
 \theta_{\max}=\beta_{\mathrm{orb}}^{\mathrm{dec}}/4.
\]
Its output \(\theta\) is \(\theta_{\mathrm{blk}}\) by the choice in Step~1.
Its initial-fiber hypothesis is the assumed inequality
\(d_{\TT^2}(r(1))<\theta_{\mathrm{blk}}\); its stage closeness follows from \eqref{eq:block-stage-beta-star} and
\(\beta_*\leq\beta_{\mathrm{fib}}\); and its comparison hypotheses are
\eqref{eq:block-bamler-smallness}.  It therefore gives
\eqref{eq:block-fiber-bound} and \eqref{eq:block-quotient-lower} throughout
the block.  The curvature bound and the quotient-length estimate
\eqref{eq:block-quotient-lower}, together with
\eqref{eq:block-stage-closeness} and
\(\beta_*\leq\beta_{\mathrm{pert}}^{\mathrm{dec}}\), give the perturbative
hypothesis of Theorem~\ref{thm:noninvariant-decay}.  Moreover,
\eqref{eq:block-fiber-bound} and
\(\theta_{\mathrm{blk}}<\beta_{\mathrm{orb}}^{\mathrm{dec}}/4\) give its
short-orbit hypothesis.  Hence
\begin{equation}\label{eq:block-noninvariant-decay}
 \|q^\perp\|_{X_r^N[10,100]}
        \leq \alpha_{\mathrm{blk}} b.
\end{equation}

Averaging commutes with constant scaling, so
\[
        q^+(1)-r^+(1)=100^{-1}q(100)^\perp.
\]
Using \eqref{eq:fixed-scaling-comparison} and
\eqref{eq:block-noninvariant-decay} at \(\tau=100\), we obtain
\[
        b^+\leq C_{\mathrm{sc}}\alpha_{\mathrm{blk}} b\leq\tfrac12b,
\]
which is \eqref{eq:block-contraction}.  The endpoint clause of
Lemma~\ref{lem:stage-stability} gives \eqref{eq:block-endpoint-jump} and
\(\BG_N(r^+(1);\Lambda_*)\).  By
\eqref{eq:block-fiber-bound}, the fibers of \(\rho^+\) have diameter
\(<\theta_{\mathrm{blk}}/10\); the choice of \(\eta_{\mathrm{blk}}\) and
\eqref{eq:block-endpoint-jump} give \eqref{eq:block-next-fibers}.  Since \(b^+\leq b/2<\eta_{\mathrm{blk}}\),
\eqref{eq:block-contraction}, \eqref{eq:block-next-fibers}, and
\(\BG_N(r^+(1);\Lambda_*)\) reproduce the three admissible-start conditions
at the restart.  The last assertion follows from the DeTurck correspondence
and uniqueness for Ricci flow.
\end{proof}

Thus every large block reconstructs the quotient-length lower bound needed for
finite-time decay, while only the error, the retained-fiber diameter, and the
lifted regularity are passed to the next block.

\subsection{Stagewise families and switching times}
The block contraction theorem is now assembled over the geometric sequence of
switching times.  This subsection constructs the global gauge, the continuous physical-scale
family \(g'\), and the piecewise invariant comparison family \(h\).

We now iterate the block contraction theorem.  Once \(T_0\) and the fixed
fibration are chosen, put
\[
        T_k=100^kT_0,\qquad k=0,1,2,\ldots.
\]
On stage \(k\), write
\[
        q_k(1)=T_k^{-1}g'(T_k),\qquad
        r_k(1)=\overline{q_k(1)},
\]
let \(q_k(\tau)\), \(r_k(\tau)\) be the corresponding flows on
\([1,100]\), and define \(b_k,d_k\) by
\eqref{eq:stage-start-data}.  Define
\[
        g'(t)=T_kq_k(t/T_k),\qquad T_k\leq t\leq T_{k+1},
\]
and use the right-continuous convention
\[
        h(t)=T_kr_k(t/T_k),\qquad T_k\leq t<T_{k+1}.
\]
At a switching time,
\begin{equation}
        h(T_{k+1}-)=T_kr_k(100),
        \qquad
        h(T_{k+1})=T_{k+1}r_{k+1}(1).
\end{equation}
Thus \(g'\) is continuous and piecewise smooth, while \(h\) is an invariant
Ricci flow on each stage and may jump when the endpoint is averaged.

\section{Quotient-circle energy, length, and componentwise iteration}
\label{sec:quotient-length}
This section converts the blockwise invariant comparison into the geometric
diameter estimate.  For a locally \(\TT^2\)-invariant torus-bundle flow, the
relevant base geometry is governed by a one-dimensional vertical energy and
the length of the quotient circle; smooth evolution, fixed rescaling, and the
small restart jumps therefore reduce to a scalar recursion.  The resulting
bounds are inserted into one iteration theorem for an arbitrary connected
large-diameter time interval.

The first subsection, \emph{One-dimensional estimates on a smooth stage},
recalls the one-dimensional estimates on a smooth invariant stage.  The
second, \emph{Transition jump estimates}, controls the scalar changes at a
restart, and the third, \emph{The scalar energy-length recursion}, solves the
resulting energy--length recursion.  The final subsection, \emph{Iteration on
a connected large-diameter interval}, combines these estimates with the block
contraction theorem, including the initial full block and a possible terminal
partial block, and obtains quotient-length and total-diameter control on the
connected interval.  On a persistent tail, the concluding corollary upgrades
the scalar bounds to convergence of the normalized quotient length and then to
the full metric-circle limit.

\subsection{One-dimensional estimates on a smooth stage}
\label{subsec:one-dimensional-estimates}
We recall the one-dimensional estimates for a smooth invariant Ricci flow.  In an invariant local trivialization, write the metric in
Kaluza--Klein form
\[
\begin{aligned}
 g(t)={}&h_{yy}(y,t)\,dy^2\\
 &+G_{ij}(y,t)(dx^i+A^i(y,t)dy)(dx^j+A^j(y,t)dy),\\
 G(y+1,t)={}&H^T G(y,t)H,
\end{aligned}
\]
where \(H\in\SL(2,\ZZ)\) by the convention of
Subsection~\ref{subsec:torus-bundle-conventions}.  Here
\(h_{yy}(y,t)\,dy^2\) is the quotient metric on the base circle and
\(G_{ij}(y,t)\) is the metric on the torus fibers; the connection term
\(A^i dy\) does not enter the scalar quantities below.  Since the base is
one-dimensional, the local evolution formulas for the quotient coefficient and
vertical block may be computed in horizontal gauges with no mixed term.  Changes
of such gauges only reparametrize the base and apply fixed linear changes to the
vertical coordinates, so the scalar quantities \(E\) and \(L\) are unaffected.
Put
\[
        \mathcal E=h^{yy}\Tr\left((G^{-1}G_y)^2\right),
        \qquad
        E(t)=\mathcal E_{\max}(t),
\]
and let
\[
        L(t)=L(S^1,h_{yy}(t))=
        \int_{S^1}\sqrt{h_{yy}(y,t)}\,dy,
\]
be the length of the quotient circle.  The inequalities in Proposition~\ref{input:invariant-energy-length} are
written in \cite{LottSesum} in the modified
one-dimensional gauge obtained from the ordinary Ricci flow by a horizontal Lie
derivative.  Since \(E(t)\) and \(L(t)\) are invariant under reparametrizing the base
circle, the same estimates hold for the ordinary invariant Ricci flow.  Thus,
whenever the denominators are nonzero,
\begin{equation}\label{eq:LS-energy-decay}
        \frac1{E(t_1)}-\frac1{E(t_0)}\geq \frac12(t_1-t_0),
        \qquad t_1\geq t_0,
\end{equation}
and
\begin{equation}\label{eq:LS-length-differential}
        \frac{d}{dt}\log L(t)\leq \frac14 E(t).
\end{equation}
If \(E(t_0)=0\), then the maximum-principle inequality underlying
\eqref{eq:LS-energy-decay} gives \(E(t)=0\) for every \(t\geq t_0\).  In
this case, the reciprocal estimate is not invoked, and
\eqref{eq:LS-length-differential} has zero right-hand side.

\subsection{Transition jump estimates}

\begin{lemma}[Transition jump estimate]\label{lem:transition-jump}
Fix \(0\leq E_0<\infty\).  There are constants
\[
        \delta_{\mathrm{J}}>0,
        \qquad C_{\mathrm{J}}<\infty,
\]
depending only on \(E_0\), with the following property.  Let
\(\rho\) and \(\tilde\rho\) be smooth invariant
metrics for any fixed torus structure.  Suppose that
\[
        E(\rho)\leq E_0,
        \qquad
        \delta:=\|\tilde\rho-\rho\|_{C^1(M,\rho)}<\delta_{\mathrm{J}}.
\]
Then
\begin{equation}\label{eq:transition-jump-estimate}
        E(\tilde\rho)\leq E(\rho)+C_{\mathrm{J}}\delta,
        \qquad
        \left|\log\frac{L(\tilde\rho)}{L(\rho)}\right|
        \leq C_{\mathrm{J}}\delta.
\end{equation}
\end{lemma}

\begin{proof}
\proofstep{Step 1. Choose invariant coordinates and control the coefficients.}

Choose \(\rho\)-arclength \(s\) on a base arc and fiber coordinates for which
the connection one-form of \(\rho\) vanishes.  In this horizontal gauge for
\(\rho\),
\[
        \rho=ds^2+G_{ij}(s)\,dx^i dx^j.
\]
Any invariant metric \(\gamma\), written in the same coordinates, has the form
\[
        \gamma=a_\gamma(y)\,dy^2+
        G_{\gamma,ij}(y)(dx^i+A^i_\gamma(y)dy)
                              (dx^j+A^j_\gamma(y)dy),
\]
where \(a_\gamma(y)dy^2\) is the quotient metric and \(G_\gamma\) is the
metric on the torus fibers.  Therefore
\[
        E(\gamma)=\max_{S^1}a_\gamma^{-1}
        \Tr\left((G_\gamma^{-1}\partial_yG_\gamma)^2\right),
        \qquad
        L(\gamma)=\int_{S^1}\sqrt{a_\gamma}\,dy.
\]
The connection form \(A_\gamma\) does not enter either expression.

Write
\[
        a=a_\rho,\quad G=G_\rho,\quad
        \tilde a=a_{\tilde\rho},\quad \tilde G=G_{\tilde\rho}.
\]
The energy bound implies
\begin{equation}\label{eq:reference-vertical-derivative}
        \bigl|G^{-1/2}G_sG^{-1/2}\bigr|\leq C(E_0).
\end{equation}
The vertical block is the restriction of the metric to the fiber directions.
For \(x\in\pi^{-1}(y)\), let \(\mathcal V_x=\ker d\pi_x\) be the vertical
subspace.  The quotient coefficient can be characterized by
\[
 a_\gamma(y)=\min_{V\in\mathcal V_x}
 \gamma_x(\partial_y+V,\partial_y+V).
\]
Local \(\TT^2\)-invariance makes the minimum independent of the choice of
\(x\) in the fiber.  It also shows directly that \(a_\gamma\) depends smoothly
on the metric coefficients and on the inverse of the vertical block.  At a
fixed basepoint, make a constant linear change of vertical frame so that
\(G=I\).  In this
\(\rho\)-orthonormal adapted frame, the relative \(C^1(M,\rho)\)-bound and
\eqref{eq:reference-vertical-derivative} give
\begin{equation}\label{eq:local-jump-coefficient-control}
        |\widehat a-1|+|\tilde G-I|
        +|\partial_s\tilde G-\partial_sG|
        \leq C(E_0)\delta,
\end{equation}
where \(\widehat a=\tilde a/a\) is the quotient coefficient of
\(\tilde\rho\) in the reference arclength coordinate.  For the derivative
term, the first covariant derivative of \(\tilde\rho-\rho\) controls the
ordinary \(s\)-derivative of the vertical block; the conversion terms involve
only reference connection coefficients, bounded by
\eqref{eq:reference-vertical-derivative}.  Because the frame was normalized
before applying matrix estimates, these constants do not depend on the
condition number of the original \(G\), the invariant atlas, the monodromy, or
the fibration class.

\proofstep{Step 2. Compare quotient lengths.}

After decreasing \(\delta_{\mathrm J}\), the first term in
\eqref{eq:local-jump-coefficient-control} gives
\[
        \left|\log\sqrt{\widehat a}\right|
        \leq C(E_0)\delta.
\]
Since \(dL(\rho)=ds\) and
\(dL(\tilde\rho)=\sqrt{\widehat a}\,ds\), integration around the quotient
circle yields
\[
 e^{-C(E_0)\delta}L(\rho)
 \leq L(\tilde\rho)
 \leq e^{C(E_0)\delta}L(\rho),
\]
which is the length estimate in \eqref{eq:transition-jump-estimate}.

\proofstep{Step 3. Compare energies.}

At the chosen basepoint set
\[
        B_\rho=\partial_sG,
        \qquad
        \widetilde B=\widehat a^{-1/2}\tilde G^{-1/2}
                     (\partial_s\tilde G)\tilde G^{-1/2}.
\]
These are the normalized vertical derivatives in the reference-orthonormal
frame.  The traces of their squares are the pointwise energy densities of \(\rho\) and
\(\tilde\rho\).
For \(\delta\) small, \(\tilde G\) stays in a fixed neighborhood of the
identity, where \(S\mapsto S^{-1}\) and \(S\mapsto S^{-1/2}\) have uniformly
bounded derivatives.  Thus \eqref{eq:local-jump-coefficient-control} and
\(|B_\rho|\leq C(E_0)\) imply
\[
        |\widetilde B-B_\rho|\leq C(E_0)\delta,
        \qquad |\widetilde B|\leq C(E_0).
\]
Consequently,
\[
\begin{aligned}
        |\Tr(\widetilde B^2)-\Tr(B_\rho^2)|
        &=|\Tr((\widetilde B-B_\rho)\widetilde B)
          +\Tr(B_\rho(\widetilde B-B_\rho))| \\
        &\leq C(E_0)\delta.
\end{aligned}
\]
The basepoint was arbitrary.  Taking the maximum over \(S^1\) and enlarging
\(C_{\mathrm{J}}\) to cover both estimates proves
\eqref{eq:transition-jump-estimate}.
\end{proof}

\subsection{The scalar energy-length recursion}

\begin{lemma}[Scalar quotient recursion and convergence]
\label{lem:scalar-quotient-recursion}
Fix \(C_0\geq0\), and let \(b_k\geq0\), \(k\geq0\), be a summable
sequence.  Suppose that the sequences \(E_k\geq0\) and \(\ell_k>0\),
\(k\geq0\), satisfy
\begin{equation}\label{eq:abstract-scalar-recursion}
        E_{k+1}\leq \frac{100E_k}{1+99E_k/2}+C_0b_k,
        \qquad
        \ell_{k+1}\leq(1+C_0b_k)
        \left(\frac{1+99E_k/2}{100}\right)^{1/2}\ell_k.
\end{equation}
Then \(\sup_kE_k<\infty\) and \(\sup_k\ell_k<\infty\), with a bound depending
only on \(E_0\), \(\ell_0\), \(C_0\), and \(\sum_k b_k\).  If, in addition,
\begin{equation}\label{eq:abstract-length-lower-bound}
        \inf_k\ell_k>0,
\end{equation}
then there is a number \(\ell_\infty>0\) such that
\begin{equation}\label{eq:abstract-scalar-convergence}
        \ell_k\longrightarrow\ell_\infty,
        \qquad E_k\longrightarrow2.
\end{equation}
\end{lemma}

\begin{proof}
For a real number \(x\), write \(x_+=\max\{x,0\}\).  Set
\(y_k=(E_k-2)_+\).  Since
\[
        \frac{100x}{1+99x/2}-2=\frac{x-2}{1+99x/2},
\]
and \((x+y-2)_+\leq(x-2)_++y\) for \(y\geq0\), the energy recurrence gives
\[
        y_{k+1}\leq \frac1{100}y_k+C_0b_k.
\]
Iterating this inequality and using \(\sum_k b_k<\infty\) shows that
\begin{equation}\label{eq:energy-excess-summable}
        \sum_k y_k<\infty.
\end{equation}
In particular, \((E_k)\) is uniformly bounded.

On that bounded range, put
\[
        Q(x)=\left(\frac{1+99x/2}{100}\right)^{1/2}.
\]
Then \(Q(x)\leq1\) for \(x\leq2\), while
\(\log Q(x)\leq C(x-2)_+\) for \(x\geq2\).  Taking logarithms in the length
recurrence and iterating gives
\[
        \log\frac{\ell_k}{\ell_0}
        \leq C\sum_{j<k}b_j+C\sum_{j<k}y_j.
\]
Both sums are uniformly bounded, so \((\ell_k)\) is uniformly bounded as well.

Assume now \eqref{eq:abstract-length-lower-bound}.  Define
\[
 a_k=\log(1+C_0b_k)+\log Q(E_k).
\]
The positive parts are summable:
\begin{equation}\label{eq:positive-log-increments-summable}
 \sum_k(a_k)_+
 \leq C_0\sum_kb_k+C\sum_k(E_k-2)_+<\infty.
\end{equation}
Writing \(u_k=\log\ell_k\), the second inequality in
\eqref{eq:abstract-scalar-recursion} gives
\(u_{k+1}\leq u_k+a_k\).  Hence
\[
 z_k=u_k+\sum_{j=k}^\infty(a_j)_+
\]
is nonincreasing.  It is bounded below by
\(\log(\inf_j\ell_j)\), so it converges.  The tail in
\eqref{eq:positive-log-increments-summable} tends to zero, and therefore
\(u_k\), and hence \(\ell_k\), converges to a positive limit
\(\ell_\infty\).

It remains to identify the limiting energy.  Since
\(\ell_{k+1}/\ell_k\to1\) and \(b_k\to0\), the second inequality in
\eqref{eq:abstract-scalar-recursion} rules out
\(E_k\leq2-\varepsilon\) along any subsequence: on such a subsequence its
right-hand multiplicative factor is eventually bounded above by a number
strictly smaller than one.  Thus \(\liminf_kE_k\geq2\).  On the other hand,
\eqref{eq:energy-excess-summable} gives
\((E_k-2)_+\to0\), and hence \(\limsup_kE_k\leq2\).  This proves
\eqref{eq:abstract-scalar-convergence}.
\end{proof}

\subsection{Iteration on a connected large-diameter interval}

\begin{theorem}[Iteration on a connected large-diameter interval]
\label{thm:large-component-iteration}
Under the standing fixed-threshold application with cutoff \(c_0\), fix
\(\delta_0\geq c_0\), an integer \(N\geq2\), and \(\sigma\in(0,1)\).  There are constants \(\eta_0>0\) and
\(T_{\mathrm{p}}<\infty\) with the
following properties.  The constant \(\eta_0\) depends only on
\(K,\delta_0,N,\sigma\); the threshold \(T_{\mathrm{p}}\) may additionally
depend on the fixed Bamler tail data \(T_B(c_0)\) and
\(\eps_B^{c_0}\).  Neither depends
on the number \(D_1\) below.

Let \(I\) be a connected time interval on which
\begin{equation}\label{eq:component-large-diameter}
        \diam(M,g(t))\geq\delta_0\sqrt t.
\end{equation}
Suppose that \(S\in I\), \(S\geq T_{\mathrm{p}}\), and
\begin{equation}\label{eq:component-start-diameter}
        \diam(M,g(S))\leq D_1\sqrt S
\end{equation}
for some \(0<D_1<\infty\).  At time \(S\), choose the torus fibration from
case~\ref{item:bamler-large-slice} of
Proposition~\ref{input:bamler-fixed-threshold}, put
\[
        q_0(1)=S^{-1}g(S),\qquad r_0(1)=\overline{q_0(1)},
\]
where the average is taken in that fibration, and suppose that \(r_0(1)\) is
Riemannian and
\begin{equation}\label{eq:component-start-error}
 b_0=\|q_0(1)-r_0(1)\|_{C^{N,\sigma}(M,r_0(1))}<\eta_0.
\end{equation}

Set \(T_k=100^kS\).  Call the \(k\)th block \emph{controlled} if
\([T_k,T_{k+1}]\subset I\).  On every controlled block, the same retained
torus fibration supports stagewise families \(q_k(\tau)\), \(r_k(\tau)\),
\(1\leq\tau\leq100\), and the estimates
\begin{align}
 b_{k+1}&\leq\tfrac12b_k,                                      \label{eq:iteration-contraction}\\
 \|q_k-r_k\|_{X_{r_k}^N[1,100]}
        &\leq C_{\mathrm{st}}b_k,                              \label{eq:iteration-stage-stability}\\
 \|r_{k+1}(1)-100^{-1}r_k(100)\|_{C^{N,\sigma}}
        &\leq C_{\mathrm{st}}b_k.                              \label{eq:component-restart-jump}
\end{align}
Here and below the norm in the last line is taken relative to
\(100^{-1}r_k(100)\).  On the union of the controlled blocks these stages
assemble into a continuous piecewise smooth pullback
\(g'(t)=\Psi(t)^*g(t)\) and a right-continuous piecewise invariant family
\(h(t)\), as in Section~\ref{sec:deturck-iteration}.  There is a constant
\(C<\infty\), depending only on \(K,\delta_0,N,\sigma\) and \(D_1\), such that
\begin{equation}\label{eq:component-length-bound}
        L(S^1,h_{yy}(t))\leq C\sqrt t
\end{equation}
on every controlled block, for either one-sided value at a switching time,
and
\begin{equation}\label{eq:component-diameter-bound}
        \diam(M,g(t))\leq C\sqrt t
\end{equation}
for every \(t\in I\cap[S,\infty)\).  Thus the same theorem covers both an
unbounded persistent tail and a finite component with a terminal partial
block.
\end{theorem}

Figure~\ref{fig:theorem4-recursion} summarizes the energy--length recursion
used in the proof below.

\begin{figure}[H]
\centering
\includegraphics[width=0.98\textwidth]{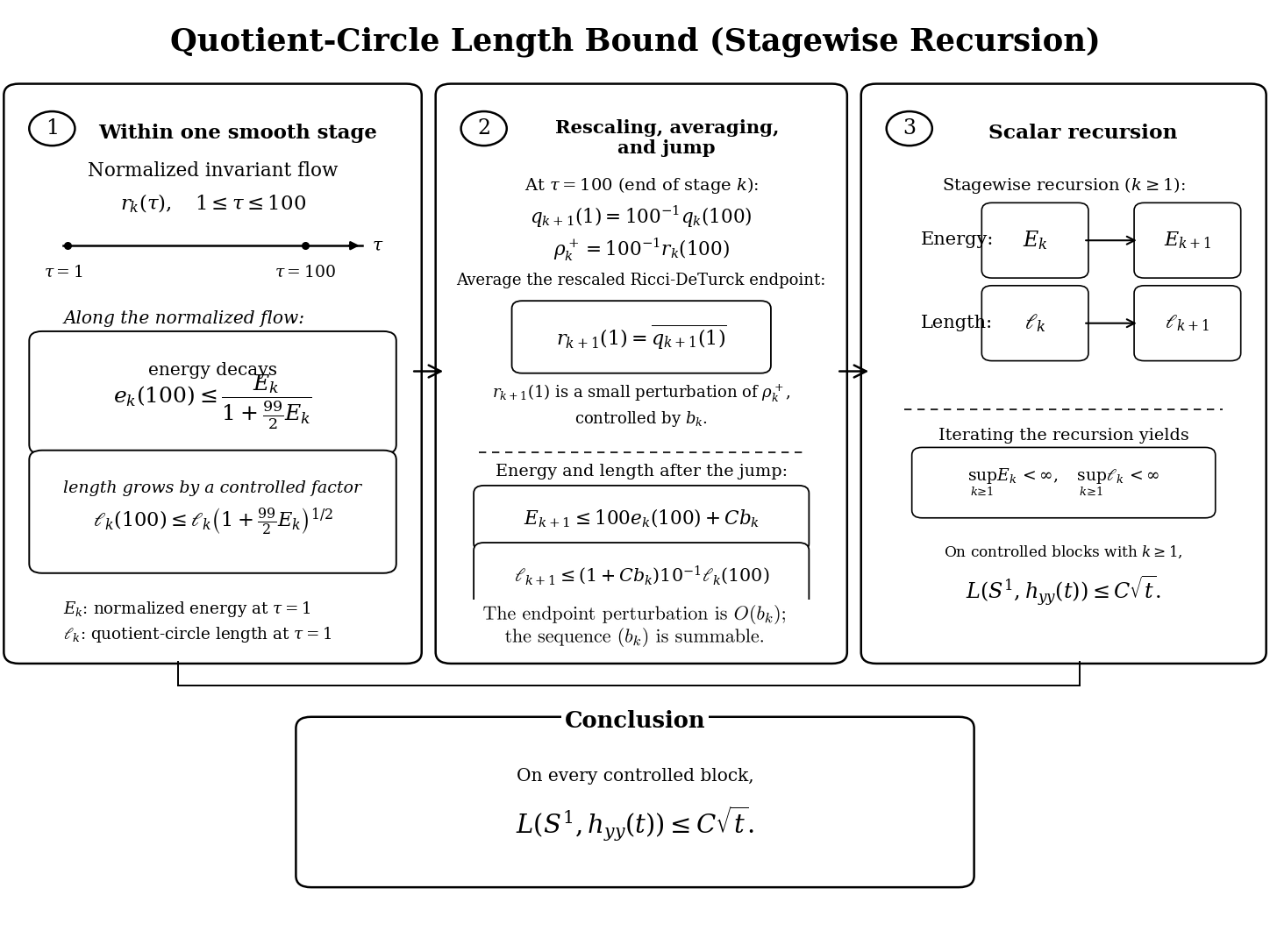}
\caption{Energy--length recursion.  The smooth-stage inequalities, the fixed
rescaling laws, and the summable jump errors reduce the geometric iteration to
a scalar contraction about the energy level \(2\).  In general this gives the
uniform length bound shown in the figure; on a persistent tail, the positive
lower length bound upgrades the same recursion to convergence.}
\label{fig:theorem4-recursion}
\end{figure}

\begin{proof}
\proofstep{Step 1. Choose the constants and prepare the initial slice.}

Apply Theorem~\ref{thm:one-block-contraction} and
Proposition~\ref{prop:one-block-quantitative} for the present
\(\delta_0,N,\sigma\).  Denote their constants by
\[
 T_{\mathrm{blk}},\quad \eta_{\mathrm{blk}},\quad
 \theta_{\mathrm{blk}},\quad C_{\mathrm{st}},\quad \Lambda_*.
\]
Let \(\delta_{\mathrm{J}}\) and \(C_{\mathrm{J}}\) be the constants in
Lemma~\ref{lem:transition-jump} for reference energy bound \(3\).  Choose
\begin{equation}\label{eq:component-eta-choice}
 0<\eta_0\leq\eta_{\mathrm{blk}},\qquad
 C_{\mathrm{st}}\eta_0<\delta_{\mathrm{J}},\qquad
 C_{\mathrm{J}}C_{\mathrm{st}}\eta_0\leq1.
\end{equation}
Choose \(T_{\mathrm{p}}\geq T_{\mathrm{blk}}\) so large that whenever
\(t\geq T_{\mathrm{p}}\) and
\(\diam(M,g(t))\geq\delta_0\sqrt t\), the Type-III estimate on
\([t/2,t]\), Shi's estimates, and
Lemma~\ref{lem:universal-cover-inj} give the uniform lifted geometry used in
Lemma~\ref{lem:high-order-averaged-slice}.  Require also that Bamler's
comparison at time \(t\) be sufficiently small in \(C^{N+5,\sigma}\) for
that lemma; \(C^{N+6}\)-smallness suffices.  If \(r_t\) is the average of
\(t^{-1}g(t)\) in Bamler's fibration, then
\begin{equation}\label{eq:component-prepared-start}
        \BG_N(r_t;\Lambda_*),\qquad
        d_{\TT^2}(r_t)<\theta_{\mathrm{blk}}.
\end{equation}
These choices have the dependence asserted in the theorem and are independent
of \(D_1\).  At \(t=S\), equations
\eqref{eq:component-start-error} and
\eqref{eq:component-prepared-start} show that
\((q_0(1),r_0(1))\) is an admissible stage start in the sense of
Definition~\ref{def:admissible-stage-start}, with constants
\(\Lambda_*,\eta_{\mathrm{blk}},\theta_{\mathrm{blk}}\).

The Type-III estimate also gives a constant \(C_{\mathrm{III}}\), depending
only on \(K\), such that
\begin{equation}\label{eq:component-typeIII-distortion}
 C_{\mathrm{III}}^{-1}g(t_1)\leq g(t_2)\leq C_{\mathrm{III}}g(t_1)
 \qquad
 (T_{\mathrm{p}}\leq t_1\leq t_2\leq100t_1).
\end{equation}
Indeed, \(|\partial_tg|_g=2|\Ric|_g\leq C/t\), and integration gives the
comparison.

\proofstep{Step 2. Iterate all controlled blocks and assemble the gauge.}

If \([S,100S]\not\subset I\), connectedness gives
\(\sup I\leq100S\).  Equations
\eqref{eq:component-start-diameter} and
\eqref{eq:component-typeIII-distortion} then imply
\[
 \diam(M,g(t))\leq CD_1\sqrt S\leq CD_1\sqrt t,
 \qquad t\in I\cap[S,\infty),
\]
so there is nothing more to prove.  Assume henceforth that the first block is
controlled.

Set \(\Phi_0=\Id\).  Inductively suppose that the \(k\)th block is controlled and
that, with
\[
 p_k(\tau)=\Phi_k^*\bigl(T_k^{-1}g(T_k\tau)\bigr),
 \qquad q_k(1)=p_k(1),\qquad r_k(1)=\overline{q_k(1)},
\]
the pair \(\bigl(q_k(1),r_k(1)\bigr)\) is an admissible stage start in the
sense of Definition~\ref{def:admissible-stage-start}.  Since
\eqref{eq:component-large-diameter} holds on
the whole block, Theorem~\ref{thm:one-block-contraction} and
Proposition~\ref{prop:one-block-quantitative} construct the stage and give
\eqref{eq:iteration-contraction}--\eqref{eq:component-restart-jump}, together
with
\[
 d_{\TT^2}(r_{k+1}(1))<\theta_{\mathrm{blk}},
 \qquad \BG_N(r_{k+1}(1);\Lambda_*).
\]
Thus the next start is admissible whenever the next block is controlled.

Let \(\psi_k(\tau)\), \(\psi_k(1)=\Id\), be the stage DeTurck
diffeomorphisms.  Then
\[
        q_k(\tau)=\psi_k(\tau)^*p_k(\tau).
\]
At the endpoint,
\[
 q_{k+1}(1)=100^{-1}q_k(100)
   =(\Phi_k\circ\psi_k(100))^*
       \bigl(T_{k+1}^{-1}g(T_{k+1})\bigr).
\]
Hence the induction continues with
\(\Phi_{k+1}=\Phi_k\circ\psi_k(100)\).  On every controlled block define
\[
 g'(T_k\tau)=T_kq_k(\tau),\qquad
 h(T_k\tau)=T_kr_k(\tau)\quad(1\leq\tau<100),
\]
using the right-continuous restart convention for \(h\), and set
\[
        \Psi(T_k\tau)=\Phi_k\circ\psi_k(\tau).
\]
The endpoint identities make \(g'\) and \(\Psi\) continuous across controlled
switching times, and
\begin{equation}\label{eq:physical-gauge-equivalence}
        g'(t)=\Psi(t)^*g(t).
\end{equation}

\proofstep{Step 3. Control the first full block directly.}

From \eqref{eq:component-start-diameter} and
\eqref{eq:component-typeIII-distortion},
\begin{equation}\label{eq:component-first-block-diameter}
 \diam(M,g(S\tau))\leq CD_1\sqrt S
        \leq CD_1\sqrt{S\tau},
 \qquad 1\leq\tau\leq100.
\end{equation}
The stage estimate \eqref{eq:iteration-stage-stability} makes \(r_0(\tau)\)
uniformly bilipschitz to the pullback \(q_0(\tau)\).  Since the quotient map
for an invariant metric is distance nonincreasing and a metric circle has
length twice its diameter,
\begin{equation}\label{eq:component-first-block-length}
 L(S^1,h_{yy}(S\tau))
 =\sqrt S\,L(S^1,(r_0)_{yy}(\tau))
 \leq CD_1\sqrt S
 \leq CD_1\sqrt{S\tau}.
\end{equation}
Thus both desired geometric bounds hold on the first controlled block without
any hypothesis on its initial vertical energy or quotient length.

\proofstep{Step 4. Create uniform scalar data at the first restart.}

Let
\[
 e_0(\tau)=\mathcal E_{\max}(r_0(\tau)),\qquad
 \rho_0^+=100^{-1}r_0(100).
\]
The energy decay on the first smooth stage gives
\[
 E(\rho_0^+)=100e_0(100)
 \leq\frac{100e_0(1)}{1+99e_0(1)/2}<\frac{200}{99}<3.
\]
By \eqref{eq:component-restart-jump} and
\eqref{eq:component-eta-choice},
\[
 \|r_1(1)-\rho_0^+\|_{C^1(M,\rho_0^+)}
 \leq C_{\mathrm{st}}b_0<\delta_{\mathrm{J}}.
\]
The transition-jump lemma therefore gives
\begin{equation}\label{eq:component-first-restart-energy}
        E_1:=E(r_1(1))\leq4.
\end{equation}
The stage and restart estimates also make \(r_1(1)\) uniformly bilipschitz to
\(100^{-1}q_0(100)\).  By
\eqref{eq:component-first-block-diameter}, the latter has diameter at most
\(CD_1\).  Hence
\begin{equation}\label{eq:component-first-restart-length}
        \ell_1:=L(S^1,(r_1)_{yy}(1))\leq CD_1.
\end{equation}

\proofstep{Step 5. Propagate the scalar energy and length bounds.}

For every controlled stage \(k\geq1\), define
\[
 e_k(\tau)=\mathcal E_{\max}(r_k(\tau)),\qquad E_k=e_k(1),
\]
and
\[
 \ell_k(\tau)=L(S^1,(r_k)_{yy}(\tau)),\qquad \ell_k=\ell_k(1).
\]
Constant scaling gives
\begin{equation}\label{eq:physical-normalized-length}
 L(S^1,h_{yy}(T_k\tau))=\sqrt{T_k}\,\ell_k(\tau),
 \qquad 1\leq\tau<100,
\end{equation}
with the same formula for the left-hand value at \(\tau=100\).
Integrating \eqref{eq:LS-energy-decay} and
\eqref{eq:LS-length-differential} gives
\begin{equation}\label{eq:stage-length-bound}
 e_k(\tau)\leq
 \frac{E_k}{1+\frac12(\tau-1)E_k},\qquad
 \ell_k(\tau)\leq\ell_k
 \left(1+\frac12(\tau-1)E_k\right)^{1/2}.
\end{equation}
In particular,
\begin{equation}\label{eq:end-stage-energy-length}
 e_k(100)\leq\frac{E_k}{1+\frac{99}{2}E_k},\qquad
 \ell_k(100)\leq\ell_k\left(1+\frac{99}{2}E_k\right)^{1/2}.
\end{equation}

Put \(\rho_k^+=100^{-1}r_k(100)\).  By
\eqref{eq:iteration-contraction},
\begin{equation}\label{eq:jump-smallness}
 b_k\leq2^{-k}b_0,
 \qquad C_{\mathrm{st}}b_k<\delta_{\mathrm{J}}.
\end{equation}
The restart estimate gives
\begin{equation}\label{eq:jump-metric-control}
 \|r_{k+1}(1)-\rho_k^+\|_{C^1(M,\rho_k^+)}
 \leq C_{\mathrm{st}}b_k.
\end{equation}
Constant scaling and \eqref{eq:end-stage-energy-length} imply
\[
 E(\rho_k^+)=100e_k(100)<3,
 \qquad L(\rho_k^+)=10^{-1}\ell_k(100).
\]
Applying Lemma~\ref{lem:transition-jump} gives
\begin{equation}\label{eq:jump-length-energy}
 E_{k+1}\leq100e_k(100)+C_{\mathrm{J}}C_{\mathrm{st}}b_k,
 \qquad
 \left|\log\frac{\ell_{k+1}}{10^{-1}\ell_k(100)}\right|
 \leq C_{\mathrm{J}}C_{\mathrm{st}}b_k.
\end{equation}
Since \(b_k\leq b_0\) and \(b_0\) was chosen small, exponentiating the
second estimate and enlarging one constant to \(C_2\) gives
\begin{equation}\label{eq:E-ell-recursion}
 E_{k+1}\leq\frac{100E_k}{1+99E_k/2}+C_2b_k,
 \qquad
 \ell_{k+1}\leq(1+C_2b_k)
 \left(\frac{1+99E_k/2}{100}\right)^{1/2}\ell_k.
\end{equation}
The sequence \((b_k)\) is summable.  Applying
Lemma~\ref{lem:scalar-quotient-recursion} to the shifted sequences beginning
at \(k=1\), with initial data
\eqref{eq:component-first-restart-energy} and
\eqref{eq:component-first-restart-length}, gives
\begin{equation}\label{eq:scalar-uniform-bounds}
        \sup_{k\geq1}E_k<\infty,
        \qquad \sup_{k\geq1}\ell_k<\infty.
\end{equation}
The bounds depend on \(D_1\) only through
\eqref{eq:component-first-restart-length}.

\proofstep{Step 6. Return to physical time on all controlled blocks.}

For \(k\geq1\), equations \eqref{eq:stage-length-bound} and
\eqref{eq:scalar-uniform-bounds} give
\(\ell_k(\tau)\leq C\sqrt\tau\).  Together with
\eqref{eq:physical-normalized-length}, this proves
\eqref{eq:component-length-bound} on every such block; the first block was
handled in \eqref{eq:component-first-block-length}.  The same argument applies
to either one-sided value at a switching time.

Proposition~\ref{prop:one-block-quantitative} gives
\(d_{\TT^2}(r_k(\tau))<\theta_{\mathrm{blk}}\) on every controlled stage, so
the physical retained fibers of \(h(t)\) have diameter \(O(\sqrt t)\).  The
stage estimate \eqref{eq:iteration-stage-stability} makes \(g'(t)\) uniformly
bilipschitz to \(h(t)\).  Applying
Lemma~\ref{lem:base-length-controls-diameter} to the retained fibration,
using the quotient-length and fiber bounds, gives
\[
        \diam(M,g'(t))\leq C\sqrt t
\]
on every controlled block.  Equation
\eqref{eq:physical-gauge-equivalence} and invariance of diameter under
pullback give the same estimate for \(g(t)\).

\proofstep{Step 7. Control a terminal partial block.}

Let \(k_*\) be the first uncontrolled index, if one exists.  The case
\(k_*=0\) was handled at the beginning of Step~2.  If \(k_*>0\), then
\(T_{k_*}\in I\), the preceding full block controls the metric at
\(T_{k_*}\), and connectedness gives
\(\sup I\leq T_{k_*+1}=100T_{k_*}\).  Applying
\eqref{eq:component-typeIII-distortion} from \(T_{k_*}\) extends the same
\(C\sqrt t\) diameter bound to
\(I\cap[T_{k_*},\infty)\).  If no uncontrolled index exists, the controlled
blocks cover \(I\cap[S,\infty)\).  This proves
\eqref{eq:component-diameter-bound} and completes the proof.
\end{proof}

\begin{corollary}[Diameter and circle convergence on a persistent large-diameter tail]
\label{cor:persistent-large-tail-diameter}
Let \(M\) be \(\TT^3\), a compact \(\Nil\) torus bundle with nontrivial
unipotent monodromy, or a compact \(\Sol\) torus bundle with monodromy
\(H\in\SL(2,\ZZ)\) satisfying \(\Tr H>2\), and let \(g(t)\), \(t\geq0\),
be an immortal Ricci flow on \(M\).
Fix a particular application of
Proposition~\ref{input:bamler-fixed-threshold}, with cutoff \(c_0\) and tail
data \(T_B(c_0),\eps_B^{c_0}\), and suppose that
\[
        \diam(M,g(t))\geq c_0\sqrt t
\]
for every sufficiently large \(t\).  Then there are constants
\(C<\infty\) and \(\ell_\infty>0\) such that
\[
        \diam(M,g(t))\leq C\sqrt t
\]
for all sufficiently large \(t\), and
\begin{equation}\label{eq:persistent-tail-circle-limit}
        (M,t^{-1}g(t))\longrightarrow C_{\ell_\infty}
\end{equation}
in the Gromov--Hausdorff sense, where \(C_{\ell_\infty}\) is the metric
circle of length \(\ell_\infty\).
\end{corollary}

\begin{proof}
Fix an integer \(N\geq2\) and \(\sigma\in(0,1)\).  Choose \(S\) after the
beginning of the persistent large-diameter tail and large enough for the
averaged Bamler slice at \(S\) to satisfy the start condition
\eqref{eq:component-start-error} with \(\delta_0=c_0\); the high-order
comparison also makes the average Riemannian.  Apply
Theorem~\ref{thm:large-component-iteration} after taking this fixed-threshold
application as the standing one, with
\[
        I=[S,\infty),
        \qquad D_1=S^{-1/2}\diam(M,g(S)).
\]
Every block is controlled, and \eqref{eq:component-diameter-bound} gives the
diameter conclusion.

We use the stage notation from the proof of
Theorem~\ref{thm:large-component-iteration}.  Thus \(T_k=100^kS\), the
non-invariant errors satisfy \(b_k\leq2^{-k}b_0\), and, for \(k\geq1\),
\[
 E_k=\mathcal E_{\max}(r_k(1)),
 \qquad
 \ell_k(\tau)=L(S^1,(r_k)_{yy}(\tau)),
 \qquad \ell_k=\ell_k(1).
\]
The quantitative one-block estimate
\eqref{eq:block-quotient-lower}, applied on every stage, gives
\begin{equation}\label{eq:persistent-tail-length-lower}
        \ell_k(\tau)\geq c_0/100,
        \qquad k\geq1,\quad 1\leq\tau\leq100.
\end{equation}
The sequences satisfy \eqref{eq:E-ell-recursion}.  Hence
Lemma~\ref{lem:scalar-quotient-recursion}, together with
\eqref{eq:persistent-tail-length-lower}, gives
\begin{equation}\label{eq:persistent-tail-start-convergence}
        \ell_k\longrightarrow\ell_\infty>0,
        \qquad E_k\longrightarrow2.
\end{equation}
The two-sided restart estimate \eqref{eq:jump-length-energy} and \(b_k\to0\)
then imply
\begin{equation}\label{eq:persistent-tail-end-length}
        10^{-1}\ell_k(100)\longrightarrow\ell_\infty.
\end{equation}

We next propagate the convergence through each smooth stage.  Put
\[
 A_k(\tau)=1+\tfrac12(\tau-1)E_k,
 \qquad
 B_k=1+\tfrac{99}{2}E_k.
\]
The length estimate in \eqref{eq:stage-length-bound} gives
\[
        \ell_k(\tau)\leq\ell_k A_k(\tau)^{1/2}.
\]
Applying the same invariant energy and length inequalities with initial time
\(\tau\) gives
\[
 \ell_k(100)
 \leq \ell_k(\tau)
       \left(1+\tfrac12(100-\tau)e_k(\tau)\right)^{1/2}
 \leq \ell_k(\tau)\left(\frac{B_k}{A_k(\tau)}\right)^{1/2},
\]
where the last inequality uses the energy bound in
\eqref{eq:stage-length-bound}.  Consequently,
\[
 \ell_k(100)\left(\frac{A_k(\tau)}{B_k}\right)^{1/2}
 \leq \ell_k(\tau)
 \leq \ell_k A_k(\tau)^{1/2}.
\]
Equations \eqref{eq:persistent-tail-start-convergence} and
\eqref{eq:persistent-tail-end-length} squeeze the two sides uniformly for
\(1\leq\tau\leq100\), and give
\begin{equation}\label{eq:persistent-tail-uniform-length}
 \sup_{1\leq\tau\leq100}
 \left|\frac{\ell_k(\tau)}{\sqrt\tau}-\ell_\infty\right|
 \longrightarrow0.
\end{equation}

It remains to show that the retained fibers vanish in the same normalization.
Let \(t=T_k\tau\), and pull Bamler's comparison at time \(t\) into the
assembled stage gauge by setting
\[
 \widehat g^B_{k,\tau}
 =T_k^{-1}\Psi(t)^*g_t^B,
 \qquad
 \widehat\pi^B_{k,\tau}=\pi_t^B\circ\Psi(t).
\]
The \(C^1\)-part of Bamler's comparison, the stage estimate
\eqref{eq:iteration-stage-stability}, and the monotonicity of
\(\eps_B^{c_0}\) show, uniformly for \(1\leq\tau\leq100\), that
\begin{equation}\label{eq:persistent-tail-bamler-to-retained}
 \|\widehat g^B_{k,\tau}-r_k(\tau)\|_{C^1(M,r_k(\tau))}
 \longrightarrow0.
\end{equation}
The bilipschitz comparison and Bamler's fiber estimate also give
\begin{equation}\label{eq:persistent-tail-bamler-fibers}
 \sup_{1\leq\tau\leq100}
 \sup_{F\text{ a }\widehat\pi^B_{k,\tau}\text{-fiber}}
 \diam_{\widehat g^B_{k,\tau}}F
 \longrightarrow0.
\end{equation}
Indeed, after scaling by \(T_k^{-1}\), the right side of the original fiber
estimate is at most a fixed multiple of
\(\eps_B^{c_0}(T_k)\sqrt\tau\).
The persistence conclusion in
Lemma~\ref{lem:retained-fibration-persistence} says that every
\(\widehat\pi^B_{k,\tau}\) is fiber-equivalent to the retained fibration.
Fixing an arbitrary \(\theta>0\), apply
Lemma~\ref{lem:isotopic-fibers-force-fixed} with
\(\lambda=c_0/100\), using
\eqref{eq:persistent-tail-length-lower}--\eqref{eq:persistent-tail-bamler-fibers}.
For all sufficiently large \(k\), it gives
\(d_{\TT^2}(r_k(\tau))<\theta\) simultaneously for all
\(\tau\in[1,100]\).  Thus
\begin{equation}\label{eq:persistent-tail-retained-fibers-vanish}
        \sup_{1\leq\tau\leq100}d_{\TT^2}(r_k(\tau))\longrightarrow0.
\end{equation}

Finally, the identity
\[
        t^{-1}\Psi(t)^*g(t)=\tau^{-1}q_k(\tau)
\]
shows that \((M,t^{-1}g(t))\) is isometric to
\((M,\tau^{-1}q_k(\tau))\).  The \(C^0\)-part of
\eqref{eq:iteration-stage-stability}, together with the uniform normalized
diameter bound, implies
\[
 d_{\mathrm{GH}}\bigl((M,\tau^{-1}q_k(\tau)),
                       (M,\tau^{-1}r_k(\tau))\bigr)\longrightarrow0
\]
uniformly in \(\tau\).  Lemma~\ref{lem:elementary-circle-collapse} gives
\[
 d_{\mathrm{GH}}\left((M,\tau^{-1}r_k(\tau)),
        C_{\ell_k(\tau)/\sqrt\tau}\right)
 \leq\frac{d_{\TT^2}(r_k(\tau))}{\sqrt\tau}.
\]
Together with \eqref{eq:persistent-tail-uniform-length} and
\eqref{eq:persistent-tail-retained-fibers-vanish}, this proves
\eqref{eq:persistent-tail-circle-limit} uniformly on every late block, and
hence for the full limit as \(t\to\infty\).
\end{proof}

\section{Descent through finite covers}
\label{sec:finite-cover-descent}

The main construction is performed on a fixed finite orientable torus-bundle
cover, while the theorem concerns the original manifold.  This section shows
how the diameter estimates, collapse statements, smooth
limits, and pointed universal-cover limits obtained on the working cover
transfer to the original manifold through the fixed finite quotient.

The section is organized as follows.  The first subsection, \emph{Circle
limits under finite quotients}, studies a continuous family of finite
quotients whose covering spaces converge to a metric circle: equivariant
compactness produces a limiting cyclic or dihedral action, the possible
quotients are respectively circles or intervals of explicitly related length,
and continuity rules out switching among different quotient types on the
late-time tail.  The second subsection, \emph{Diameter and smooth descent},
first compares the diameters of a manifold and a fixed finite Riemannian
cover, and then states the descent of normalized diameter bounds, point
collapse, circle limits, pointed universal-cover limits, and fixed-coordinate
smooth or exponential convergence.  These results are collected here so that
the Euclidean, Nil, and Sol case arguments can be carried out
entirely on one working cover and descended only once in the final proof.

\subsection{Circle limits under finite quotients}
The following compactness lemma is used when a circle limit on a fixed finite
cover is descended to the original manifold.

\begin{lemma}[Finite quotients of a circle collapse]
\label{lem:finite-quotient-circle-collapse}
Let \(\Gamma\) be a finite group acting on a compact space \(\widehat X\),
and let \(\widehat d_t\), \(t\geq T\), be a family of
\(\Gamma\)-invariant length metrics that is continuous in the uniform norm on
\(\widehat X\times\widehat X\).  Put
\(X_t=(\widehat X,\widehat d_t)/\Gamma\).  Suppose that
\[
        (\widehat X,\widehat d_t)\longrightarrow C_L
\]
in the Gromov--Hausdorff sense, where \(C_L\) is a metric circle of length
\(L>0\).  Then \(X_t\) has a full Gromov--Hausdorff limit.  The limit is
either a metric circle of length \(L/m\) or a closed interval of length
\(L/(2m)\), for some positive integer \(m\) dividing \(|\Gamma|\).  If every
limiting action of \(\Gamma\) on \(C_L\) is orientation preserving, then the
limit is a circle.
\end{lemma}

\begin{proof}
We first identify the possible subsequential limits and include the finite-group
compactness argument.  Let \(t_j\to\infty\), and choose
Gromov--Hausdorff \(\varepsilon_j\)-approximations
\[
 f_j:(\widehat X,\widehat d_{t_j})\longrightarrow C_L,
 \qquad \varepsilon_j\longrightarrow0.
\]
Fix a countable dense set \(A=\{a_1,a_2,\ldots\}\subset C_L\).  For every
\(n\) choose \(x_{j,n}\in\widehat X\) with
\(d_{C_L}(f_j(x_{j,n}),a_n)<2\varepsilon_j\).  Since \(C_L\) is compact and
\(\Gamma\times\mathbb N\) is countable, a diagonal subsequence may be chosen
so that
\[
        f_j(\gamma x_{j,n})
\]
converges for every \(\gamma\in\Gamma\) and every \(n\).  Define its limit to
be \(\rho(\gamma)a_n\).  The approximation property of \(f_j\) and the fact
that \(\gamma\) is an isometry of \((\widehat X,\widehat d_{t_j})\) give
\[
 d_{C_L}(\rho(\gamma)a_n,\rho(\gamma)a_m)
 =d_{C_L}(a_n,a_m).
\]
Thus \(\rho(\gamma)\) extends uniquely to a distance-preserving map
of \(C_L\).

We record the convergence property needed to pass the group law and the
quotient metric to the limit.  If \(z_j\in\widehat X\) and
\(f_j(z_j)\to z\in C_L\), then
\begin{equation}\label{eq:finite-group-approximate-equivariance}
        f_j(\gamma z_j)\longrightarrow\rho(\gamma)z.
\end{equation}
Indeed, approximate \(z\) by some \(a_n\).  The approximation property of
\(f_j\), the isometry of \(\gamma\), and the already established convergence
of \(f_j(\gamma x_{j,n})\) show that the two sides of
\eqref{eq:finite-group-approximate-equivariance} are as close as desired.
Applying this property twice and using the exact group relation on
\(\widehat X\) gives
\[
        \rho(\gamma\gamma')=
        \rho(\gamma)\rho(\gamma')
        \qquad(\gamma,\gamma'\in\Gamma).
\]
In particular, \(\rho(\gamma^{-1})\) is the inverse of
\(\rho(\gamma)\), so each extension is an isometry and
\[
        \rho:\Gamma\longrightarrow\operatorname{Isom}(C_L)
\]
is a homomorphism.

If \(f_j(x_j)\to x\) and \(f_j(y_j)\to y\), then finiteness of \(\Gamma\),
\eqref{eq:finite-group-approximate-equivariance}, and the approximation
property give
\begin{align*}
 d_{X_{t_j}}([x_j],[y_j])
 &=\min_{\gamma\in\Gamma}\widehat d_{t_j}(x_j,\gamma y_j)\\
 &\longrightarrow
   \min_{\gamma\in\Gamma}d_{C_L}(x,\rho(\gamma)y)
  =d_{C_L/\rho(\Gamma)}([x],[y]).
\end{align*}
The convergence is uniform over pairs: otherwise a violating sequence of
pairs would have a subsequence whose \(f_j\)-images converge, contradicting
the displayed limit.  Since the maps \(f_j\) are asymptotically onto, the
induced correspondence between the two orbit spaces is asymptotically onto as
well.  Hence
\begin{equation}\label{eq:finite-circle-quotient-subsequence}
        X_{t_j}\longrightarrow C_L/\rho(\Gamma).
\end{equation}

Every finite subgroup of \(\operatorname{Isom}(C_L)\cong O(2)\) is either a
rotation group of order \(m\) or a dihedral group of order \(2m\); here
the case \(m=1\) in the second alternative is the group generated by one
reflection.  In the rotation case, the quotient in
\eqref{eq:finite-circle-quotient-subsequence} is a circle of length \(L/m\).
In the dihedral case, it is a closed interval of length \(L/(2m)\).  In
either case, \(m\) divides \(|\Gamma|\).  Hence only
finitely many isometry classes can occur as subsequential limits; denote their
set by \(\mathcal S\).

It remains to rule out switching among these finitely many possibilities.  On
the fixed orbit set \(\widehat X/\Gamma\), the quotient distance is
\[
 d_t^\Gamma([x],[y])=\min_{\gamma\in\Gamma}
        \widehat d_t(x,\gamma y).
\]
Consequently,
\[
 \bigl\|d_t^\Gamma-d_s^\Gamma\bigr\|_{C^0}
 \leq \bigl\|\widehat d_t-\widehat d_s\bigr\|_{C^0},
\]
so \(t\mapsto X_t\) is continuous in the Gromov--Hausdorff metric.  The
subsequential compactness just proved implies
\[
        \operatorname{dist}_{\mathrm{GH}}(X_t,\mathcal S)\longrightarrow0;
\]
otherwise a sequence staying a fixed positive distance from \(\mathcal S\)
would have a subsequence satisfying
\eqref{eq:finite-circle-quotient-subsequence}.  If \(\mathcal S\) has only one
member, the full-limit assertion follows immediately.  Otherwise, put
\[
 \Delta=\min\bigl\{d_{\mathrm{GH}}(Y,Z):Y,Z\in\mathcal S,
                         \ Y\not\cong Z\bigr\}>0
\]
and choose \(0<r<\Delta/3\).  The Gromov--Hausdorff balls
\(B_{\mathrm{GH}}(Y,r)\), \(Y\in\mathcal S\), are pairwise disjoint.  For all
sufficiently large \(t\), the space \(X_t\) lies in their union.  Since the interval of sufficiently large times is connected and
\(t\mapsto X_t\) is continuous, its image lies
in one of these balls, say \(B_{\mathrm{GH}}(Y_*,r)\).  The choice
\(r<\Delta/3\), together with
\(\operatorname{dist}_{\mathrm{GH}}(X_t,\mathcal S)\to0\), then gives
\(d_{\mathrm{GH}}(X_t,Y_*)\to0\).  This proves the full-limit assertion.
Finally, if every limiting action is orientation preserving, its image is
cyclic, so the selected limit is a circle.
\end{proof}

\subsection{Diameter and smooth descent}
We first prove the elementary comparison used whenever the proof passes to a
fixed finite cover.

\begin{lemma}[Diameter under a finite Riemannian cover]
\label{lem:finite-cover-diameter}
Let
\(\pi_{\mathrm{fin}}:(M_{\mathrm{fin}},g_{\mathrm{fin}})\to(M,g)\) be a
connected, \(d\)-sheeted Riemannian covering of closed connected manifolds.
Then
\[
        \diam(M,g)\leq\diam(M_{\mathrm{fin}},g_{\mathrm{fin}})
        \leq 2d\,\diam(M,g).
\]
\end{lemma}

\begin{proof}
The first inequality follows because the quotient map is distance
nonincreasing.  For the second, put \(D=\diam(M,g)\), fix \(p\in M\), and
enumerate the fiber
\[
 \pi_{\mathrm{fin}}^{-1}(p)
 =\{p^{\mathrm{fin}}_1,\ldots,p^{\mathrm{fin}}_d\}.
\]
Every point \(x^{\mathrm{fin}}\in M_{\mathrm{fin}}\) lies within distance \(D\) of one of
the \(p^{\mathrm{fin}}_i\): project \(x^{\mathrm{fin}}\) to \(M\), choose a minimizing
path from its projection to \(p\), and lift this path starting at
\(x^{\mathrm{fin}}\).

For \(\varepsilon>0\), the path-connected open balls
\(B(p^{\mathrm{fin}}_i,D+\varepsilon)\) therefore cover the connected space
\(M_{\mathrm{fin}}\).  Their intersection graph is connected.  Hence any two
centers can be joined by a chain of at most \(d-1\) intersecting balls, and
successive centers in the chain are at distance less than
\(2(D+\varepsilon)\).  Two arbitrary points of \(M_{\mathrm{fin}}\) are each
within \(D\) of a center, so
\[
        \diam(M_{\mathrm{fin}},g_{\mathrm{fin}})
        \leq 2D+2(d-1)(D+\varepsilon).
\]
Letting \(\varepsilon\downarrow0\) proves the second inequality.
\end{proof}

\begin{proposition}[Finite-cover descent]
\label{prop:finite-cover-descent}
Let
\[
 \pi_{\mathrm{fin}}:M_{\mathrm{fin}}\longrightarrow M
\]
be a connected, \(d\)-sheeted finite regular cover of closed connected
manifolds, with deck group \(\Gamma\).  Let \(g(t)\), \(t\geq T\), be a
\(C^0\)-continuous family of Riemannian metrics on \(M\), let
\(g_{\mathrm{fin}}(t)=\pi_{\mathrm{fin}}^*g(t)\), and let \(a(t)>0\) be
continuous.  Put
\[
 \widehat g(t)=a(t)g(t),\qquad
 \widehat g_{\mathrm{fin}}(t)=a(t)g_{\mathrm{fin}}(t).
\]
Then the following conclusions hold.
\begin{enumerate}
\item\label{item:finite-cover-diameter-descent}
At every time,
\[
 \diam(M,\widehat g(t))
 \leq \diam(M_{\mathrm{fin}},\widehat g_{\mathrm{fin}}(t))
 \leq 2d\,\diam(M,\widehat g(t)).
\]
Consequently, any \(O(1)\), \(o(1)\), or two-sided positive diameter bound
for either normalized family transfers to the other, with constants depending
only on \(d\).

\item\label{item:finite-cover-point-descent}
If
\[
 (M_{\mathrm{fin}},\widehat g_{\mathrm{fin}}(t))\longrightarrow\pt
\]
in the Gromov--Hausdorff sense, then
\((M,\widehat g(t))\to\pt\).

\item\label{item:finite-cover-circle-descent}
Suppose that
\[
 (M_{\mathrm{fin}},\widehat g_{\mathrm{fin}}(t))\longrightarrow C_L
\]
for a metric circle of length \(L>0\).  Then
\((M,\widehat g(t))\) has a full Gromov--Hausdorff limit.  It is either a
metric circle of length \(L/m\) or a closed interval of length \(L/(2m)\),
for some positive integer \(m\) dividing \(|\Gamma|\).  The limit is a circle
if every limiting action of \(\Gamma\) on \(C_L\) is orientation preserving;
it is a closed interval if a limiting action contains an
orientation-reversing isometry.

\item\label{item:finite-cover-universal-cover}
The manifolds \(M_{\mathrm{fin}}\) and \(M\) have the same universal cover,
and the lifts of \(g_{\mathrm{fin}}(t)\) and \(g(t)\) to it agree.  Hence every
pointed Cheeger--Hamilton limit of one lifted family is the corresponding
limit of the other.

\item\label{item:finite-cover-smooth-descent}
Suppose, for an integer \(m\geq0\), that
\(\widehat g_{\mathrm{fin}}(t)\) converges in a fixed-coordinate
\(C^m\)-norm to a Riemannian metric \(g_{\mathrm{fin},\infty}\).  Then the
limit is \(\Gamma\)-invariant, descends to a metric \(g_\infty\) on \(M\),
and \(\widehat g(t)\to g_\infty\) in \(C^m\).  A uniform exponential rate
on the cover descends with the same exponent.
\end{enumerate}
\end{proposition}

\begin{proof}
Part~\ref{item:finite-cover-diameter-descent} is
Lemma~\ref{lem:finite-cover-diameter}, applied after the common constant
scaling by \(a(t)\).  Convergence of a compact metric space to a point is
equivalent to convergence of its diameter to zero, so
part~\ref{item:finite-cover-point-descent} follows from the first inequality.

For part~\ref{item:finite-cover-circle-descent}, let
\(\widehat d_t\) be the distance function of
\(\widehat g_{\mathrm{fin}}(t)\).  The deck group acts isometrically for every
\(t\).  Continuity of the metric family on the compact manifold implies
continuity of \(t\mapsto\widehat d_t\) in the uniform norm: nearby metrics
are uniformly bilipschitz, and their distance functions are therefore
uniformly close.  Lemma~\ref{lem:finite-quotient-circle-collapse} applies.
If a limiting action contains an orientation-reversing isometry, its finite
image in \(\operatorname{Isom}(C_L)\) is dihedral and its quotient is a closed
interval; the full-limit assertion in that lemma rules out a different
quotient type along another sequence.

Part~\ref{item:finite-cover-universal-cover} is immediate from covering-space
theory and the identity between the two lifted metrics.  Finally, every deck
transformation preserves \(\widehat g_{\mathrm{fin}}(t)\).  Passing to the
fixed-coordinate limit shows that it preserves
\(g_{\mathrm{fin},\infty}\), which therefore descends.  Pullback by a fixed
finite covering identifies the local \(C^m\)-norms upstairs and downstairs,
proving part~\ref{item:finite-cover-smooth-descent} and its statement about
rates.
\end{proof}

\section{The Euclidean and Sol cases}
\label{sec:euclidean-sol}
This section completes the long-time analysis in the Euclidean and Sol cases.
On the working cover, the Euclidean case has identity monodromy and is
eventually forced into the uniform flat-stability regime, while the Sol case
has hyperbolic monodromy with positive eigenvalues, which rules out the
almost-flat alternative and
leaves a persistent large-diameter tail to which the block and quotient-length
estimates apply.  The resulting cover-level
conclusions are exponential convergence to a flat metric in the Euclidean
case, and \(O(\sqrt t)\) diameter together with convergence of the compact
normalized flow to a metric circle in the Sol case.

The section is organized as follows.  After stating the two case conclusions,
the first subsection, \emph{The Euclidean case}, proves that almost-flat slices
occur arbitrarily far out, uses diameter normalization and near-flat
compactness to enter a uniform flat-stability neighborhood on a finite torus
cover, and obtains exponential convergence there.  This
descends directly to the working cover.  The second
subsection, \emph{The Sol case}, excludes the almost-flat alternative
by the non-virtual-nilpotence of the bundle group, applies the persistent-tail
corollary to obtain both the matching \(O(\sqrt t)\) upper bound and the full
metric-circle limit, and invokes the Type-III blowdown theorem only for the
universal-cover Sol limit.
The case conclusions are established on the working cover; the original
manifold is reintroduced only in the proof of the main theorem.

\paragraph{\normalfont\itshape Working cover.}
Throughout this section and Section~\ref{sec:nil-bootstrap}, we use the single
\hyperref[par:working-cover-convention]{working-cover convention}:
\(M\) denotes that fixed cover and \(g(t)\) its lifted flow.  The original
manifold is reintroduced only in Section~\ref{sec:main-theorem-proof}, where
the conclusions are descended at once using
Proposition~\ref{prop:finite-cover-descent}.  The only later finite cover is
the stability cover supplied by Proposition~\ref{prop:near-flat-stability-entry}
in the Euclidean proof.  It is automatically regular because the working
manifold there is a three-torus with abelian fundamental group, and it is used
only to enter the flat-stability theorem and descend the resulting flat limit,
not to alter the working monodromy.
We use throughout \(|\Rm_{g(t)}|\leq C/t\), which follows from
\eqref{eq:bamler-typeIII} in dimension three.

\begin{theorem}[Euclidean and Sol long-time limits]
\label{thm:euclidean-sol-long-time-limits}
Let \(g(t)\), \(t\geq0\), be an immortal Ricci flow on the orientable total
space \(M\) of a \(\TT^2\)-bundle over \(S^1\), under the
\hyperref[par:working-cover-convention]{working-cover convention}.  Let
\(H\in\SL(2,\ZZ)\) be its monodromy, as
in Subsection~\ref{subsec:torus-bundle-conventions}.  Assume either that
\(H=I\), so that \(M\cong\TT^3\) is in the Euclidean case, or that
\(\Tr H>2\), so that \(M\) is in the Sol case with positive-eigenvalue
monodromy.  Then the following hold.
\begin{enumerate}
\item If \(M\) has Euclidean type (equivalently, \(H=I\)), then for every
\(m\geq0\), \(g(t)\) converges
exponentially fast in the \(C^m\)-norm of any fixed smooth background metric,
as \(t\to\infty\), to a flat metric on \(M\).
\item If \(M\) has Sol type (equivalently, \(\Tr H>2\) under the
working-cover convention), then there
are constants \(0<c_-<C<\infty\) such that
\[
        c_-\sqrt t\leq \diam(M,g(t))\leq C\sqrt t,
\]
for all sufficiently large \(t\).  In addition, the Gromov--Hausdorff limit of
\((M,t^{-1}g(t))\) is a circle.  For every basepoint
\(\widetilde m\in\widetilde M\), the lifted blowdown flows
\[
        (\widetilde M,\widetilde m,s^{-1}\widetilde g(s\tau)),
        \qquad \tau>0,
\]
converge, as \(s\to\infty\), in the pointed Cheeger--Hamilton sense to a
pointed flow isometric to the Sol expanding soliton solution
\[
        (\RR^3,0,e^{-2z}dx^2+e^{2z}dy^2+4\tau\,dz^2),
        \qquad \tau>0.
\]
\end{enumerate}
\end{theorem}

\subsection{The Euclidean case}\label{subsec:euclidean-case}

The following proposition is specific to the Euclidean argument.  It packages
the near-flat compactness step and the passage from compactness modulo
diffeomorphisms to the fixed-coordinate stability neighborhood used below.

\begin{proposition}[Near-flat entry into a uniform stability neighborhood]
\label{prop:near-flat-stability-entry}
Fix \(\rho\in(0,1)\).  There exist an integer
\(m_{\mathrm{flat}}\geq2\), a number
\(\varepsilon_{\mathrm{flat}}>0\), a fixed model torus
\(\TT^3_*\), a compact set \(\mathcal F\) of flat metrics on
\(\TT^3_*\), and a uniform \(h^{2+\rho}\)-stability neighborhood
\(\mathcal U_{\mathcal F}\) of \(\mathcal F\), as in
Proposition~\ref{input:flat-stability}, with the following property.

Let \(q\) be a smooth metric on \(\TT^3\), put
\(D=\diam(\TT^3,q)>0\), and suppose that
\begin{equation}\label{eq:near-flat-entry-hypotheses}
 D^{j+2}\sup_{\TT^3}|\nabla^j\Rm_q|_q
 <\varepsilon_{\mathrm{flat}},
 \qquad 0\leq j\leq m_{\mathrm{flat}}.
\end{equation}
Then there is a connected finite covering map
\[
        p_{\mathrm{st}}:\TT^3_*\longrightarrow\TT^3
\]
such that
\begin{equation}\label{eq:near-flat-entry-conclusion}
        p_{\mathrm{st}}^*(D^{-2}q)\in\mathcal U_{\mathcal F}.
\end{equation}
In particular, the Ricci flow starting from the metric in
\eqref{eq:near-flat-entry-conclusion} converges exponentially in fixed
coordinates to a flat metric.
\end{proposition}

\begin{proof}
Claim~2.16 of \cite{LottSesum} gives the following diameter-normalized
compactness statement.  If \(q_i\) are metrics on \(\TT^3\) with diameter one
and \(\|\Rm_{q_i}\|_{C^0}\to0\), then, after passing to connected finite
covers and applying diffeomorphisms of the covering tori, the lifted metrics
are \(C^1\)-close to a compact subset of the moduli space of flat
three-tori.  The paragraph immediately following that claim observes that,
when scale-normalized curvature-derivative bounds are also available, the same
argument gives closeness in any prescribed finite \(C^\alpha\)-topology.

Choose an integer \(\alpha>2+\rho\), and choose
\(m_{\mathrm{flat}}\) large enough for this \(C^\alpha\)-compactness
argument.  To choose representatives, consider the finite-dimensional
space of marked flat metrics on a fixed model torus \(\TT^3_*\).  The
mapping-class group \(\operatorname{GL}(3,\ZZ)\) acts properly discontinuously
on this space, and the quotient map to flat moduli is open.  For each point of
the compact subset of moduli obtained above, choose a lift and a relatively
compact neighborhood of that lift.  Finitely many of the quotient images cover
the compact subset, so the union of the closures of the corresponding
neighborhoods is a compact set \(\mathcal F\) of marked flat metrics that maps
onto it.  Thus \(\mathcal F\) is a compact set of flat metrics on
\(\TT^3_*\) representing every limiting moduli class.  Proposition~\ref{input:flat-stability} gives a uniform open
stability neighborhood \(\mathcal U_{\mathcal F}\) in the fixed
\(h^{2+\rho}\)-topology.

If no \(\varepsilon_{\mathrm{flat}}\) had the stated property, there would be
a sequence of diameter-one metrics for which the left side of
\eqref{eq:near-flat-entry-hypotheses} tends to zero for every
\(0\leq j\leq m_{\mathrm{flat}}\), but for which no finite-cover pullback,
after any identification of the covering torus with \(\TT^3_*\), belongs to
\(\mathcal U_{\mathcal F}\).  The preceding \(C^\alpha\)-compactness gives
finite covers and diffeomorphisms whose pulled-back metrics converge in
\(C^\alpha\), hence in \(h^{2+\rho}\), to \(\mathcal F\).  For large indices
they lie in \(\mathcal U_{\mathcal F}\), a contradiction.

For a general \(q\), rescaling by \(D^{-2}\) makes the diameter one and turns
\eqref{eq:near-flat-entry-hypotheses} into the preceding scale-one
hypotheses.  The compactness argument initially produces a covering map
\(p_0:\TT^3_*\to\TT^3\) and a diffeomorphism
\(\chi:\TT^3_*\to\TT^3_*\) such that
\(\chi^*p_0^*(D^{-2}q)\in\mathcal U_{\mathcal F}\).  Replacing \(p_0\) by
\(p_0\circ\chi\) absorbs the diffeomorphism into the covering map and gives
\eqref{eq:near-flat-entry-conclusion}.  The final assertion is
Proposition~\ref{input:flat-stability}.
\end{proof}

The constants in Proposition~\ref{prop:near-flat-stability-entry} are
independent of the cutoff in any application of
Proposition~\ref{input:bamler-fixed-threshold}.  In the proof that follows, the
analytic tolerance \(\varepsilon_{\mathrm{flat}}\) determines the separate
normalized-diameter parameter \(\delta_{\mathrm{flat}}\); neither the
auxiliary cutoff \(\zeta\) nor the fixed global cutoff
\(c_{\mathrm{af}}\) is changed.

\begin{proof}[Proof of the Euclidean case]
We first prove the scale-invariant statement
\begin{equation}\label{eq:euclidean-diameter-liminf-zero}
        \liminf_{t\to\infty}t^{-1/2}\diam(M,g(t))=0.
\end{equation}
We now make a fresh application of
Proposition~\ref{input:bamler-fixed-threshold}, independent of the
fixed global application with \(c_0=c_{\mathrm{af}}\).  This new application
does not alter \(c_{\mathrm{af}}\), \(T_B(c_{\mathrm{af}})\), or
\(\eps_B^{c_{\mathrm{af}}}\).  Fix an auxiliary cutoff
\(\zeta\in(0,1)\) with \(K\zeta^2<\eps_{\mathrm{af}}\), apply the block
machinery with \(c_0=\zeta\), and denote the corresponding tail data by
\(T_B(\zeta)\) and \(\eps_B^{\zeta}\).  If the \(\zeta\)-small alternative
did not occur at arbitrarily
large times, then on some tail
\[
        \diam(M,g(t))\geq\zeta\sqrt t.
\]
Corollary~\ref{cor:persistent-large-tail-diameter}, applied with
\(c_0=\zeta\), \(T_B(\zeta)\), and \(\eps_B^{\zeta}\), would give
\(\diam(M,g(t))=O(\sqrt t)\).  The Euclidean case of
Proposition~\ref{input:lott-blowdown} would then imply
\[
        (M,t^{-1}g(t))\longrightarrow\pt,
\]
and hence \(t^{-1/2}\diam(M,g(t))\to0\), contradicting the displayed lower
bound.  Thus the \(\zeta\)-small alternative occurs arbitrarily late for every
such \(\zeta\), which proves \eqref{eq:euclidean-diameter-liminf-zero}.

Fix \(\rho\in(0,1)\), and let
\[
 m_{\mathrm{flat}},\qquad \varepsilon_{\mathrm{flat}},\qquad
 \TT^3_*,\qquad \mathcal F,\qquad \mathcal U_{\mathcal F}
\]
be the data in Proposition~\ref{prop:near-flat-stability-entry}.  By the
\hyperref[par:working-cover-convention]{working-cover convention},
\(M\cong\TT^3\) and \(H=I\) already, so no preliminary cover is needed
to reduce the monodromy to the identity.

The Type-III estimate on \([t_0/2,t_0]\) and Shi's derivative estimates give
constants \(C_j(K)\), independent of \(t_0\), such that
\begin{equation}\label{eq:euclidean-small-branch-derivatives}
 \sup_M|\nabla^j\Rm|_{g(t_0)}
 \leq C_j(K)t_0^{-1-j/2},
 \qquad 0\leq j\leq m_{\mathrm{flat}},
\end{equation}
whenever \(t_0\) is sufficiently large.  Choose
\(\delta_{\mathrm{flat}}>0\) so small that
\begin{equation}\label{eq:euclidean-flat-entry-tolerance}
 C_j(K)\delta_{\mathrm{flat}}^{j+2}
 <\varepsilon_{\mathrm{flat}},
 \qquad 0\leq j\leq m_{\mathrm{flat}}.
\end{equation}
By \eqref{eq:euclidean-diameter-liminf-zero}, choose a sufficiently large
\(t_0\) for which, with \(D=\diam(M,g(t_0))\),
\begin{equation}\label{eq:euclidean-flat-entry-slice}
        \frac{D}{\sqrt{t_0}}<\delta_{\mathrm{flat}}.
\end{equation}
This choice is where the analytic tolerance
\(\varepsilon_{\mathrm{flat}}\) enters; it is independent of both the
auxiliary geometric cutoff \(\zeta\) and the fixed global cutoff
\(c_{\mathrm{af}}\).

Combining \eqref{eq:euclidean-small-branch-derivatives},
\eqref{eq:euclidean-flat-entry-tolerance}, and
\eqref{eq:euclidean-flat-entry-slice} gives
\[
 D^{j+2}\sup_M|\nabla^j\Rm|_{g(t_0)}
 <\varepsilon_{\mathrm{flat}},
 \qquad 0\leq j\leq m_{\mathrm{flat}}.
\]
Proposition~\ref{prop:near-flat-stability-entry}, applied directly to
\(q=g(t_0)\), therefore gives a connected finite covering map
\[
        p_{\mathrm{st}}:\TT^3_*\longrightarrow M
\]
such that
\begin{equation}\label{eq:euclidean-fixed-coordinate-entry}
        p_{\mathrm{st}}^*(D^{-2}g(t_0))
        \in\mathcal U_{\mathcal F}.
\end{equation}
The diffeomorphism arising in compactness modulo diffeomorphisms has already
been absorbed into \(p_{\mathrm{st}}\) in that proposition.  Set
\[
 M_{\mathrm{st}}=\TT^3_*,\qquad
 \pi_{\mathrm{st}}=p_{\mathrm{st}},\qquad
 g_{\mathrm{st}}(t)=\pi_{\mathrm{st}}^*g(t).
\]

Define the restarted normalized flow
\begin{equation}\label{eq:euclidean-restarted-normalized-flow}
        \widehat g_{\mathrm{st}}(\tau)
        =D^{-2}g_{\mathrm{st}}(t_0+D^2\tau),
        \qquad \tau\geq0.
\end{equation}
By uniqueness, this is the actual Ricci flow starting from the metric in
\eqref{eq:euclidean-fixed-coordinate-entry}.  By
Proposition~\ref{input:flat-stability},
\(\widehat g_{\mathrm{st}}(\tau)\) converges exponentially in fixed coordinates
to a flat metric \(\widehat g_{\mathrm{st},\infty}\).  The fixed-background
input there is \cite[Theorem~3.7]{GIK}, which is stated for the actual Ricci
flow.  In its proof, Ricci--DeTurck convergence is established first, and
\cite[Proposition~3.6]{GIK} transfers it to the ungauged flow, without a
time-dependent pullback in the conclusion.  Scaling back gives exponential
convergence of \(g_{\mathrm{st}}(t)\) to
\(D^2\widehat g_{\mathrm{st},\infty}\) in every fixed \(C^m\)-norm.

Since the fundamental group of $M$ is abelian, the connected finite cover
\(\pi_{\mathrm{st}}:M_{\mathrm{st}}\to M\) is already regular.
Part~\ref{item:finite-cover-smooth-descent} of
Proposition~\ref{prop:finite-cover-descent} therefore applies directly: the
flat limit \(D^2\widehat g_{\mathrm{st},\infty}\) descends to a flat metric on
\(M\), and \(g(t)\) converges to it exponentially in every fixed
\(C^m\)-norm.
\end{proof}

\subsection{The Sol case}

\begin{proof}[Proof of the Sol case]
Use the fixed global application of
Proposition~\ref{input:bamler-fixed-threshold} with
\(c_0=c_{\mathrm{af}}\), tail time \(T_B(c_{\mathrm{af}})\), and error
\(\eps_B^{c_{\mathrm{af}}}\).  Its almost-flat alternative cannot occur at any sufficiently
large time.  Indeed, if it occurred, then
\eqref{eq:bamler-small-is-almost-flat} and the almost-flat topological theorem,
Proposition~\ref{input:almost-flat-topology}, would imply that
\(\pi_1(M)\) is virtually nilpotent.  This is impossible in the Sol case.
Indeed, \(\pi_1(M)\) has a finite-index subgroup that is a cocompact lattice
in the Lie group \(\Sol\), which has exponential volume growth.  By the
volume-growth comparison for compact quotients \cite{Svarc} and invariance
of growth type under passage to finite-index subgroups, \(\pi_1(M)\)
therefore has exponential growth.  On the
other hand, \(\pi_1(M)=\ZZ^2\rtimes_H\ZZ\) is polycyclic, so if it were
virtually nilpotent then it would have polynomial growth by
\cite[Theorem~4.3(1)]{WolfGrowth}.  Hence it is not virtually nilpotent.

Consequently, the torus-bundle case occurs for every sufficiently large time,
and
\[
        \diam(M,g(t))\geq c_{\mathrm{af}}\sqrt t.
\]
Corollary~\ref{cor:persistent-large-tail-diameter}, applied with
\(c_0=c_{\mathrm{af}}\), gives the matching upper bound and a number
\(\ell_\infty>0\) such that
\[
        (M,t^{-1}g(t))\longrightarrow C_{\ell_\infty}.
\]
Since the working cover has Thurston type \(\Sol\),
Proposition~\ref{input:lott-blowdown} gives the pointed universal-cover
blowdown to the Sol expanding soliton.  Its weaker classification of
subsequential compact limits in the Sol case is not used to obtain the full
circle limit.
\end{proof}

\section{The Nil case}
\label{sec:nil-bootstrap}
The Nil case, with nontrivial unipotent monodromy, is more delicate because
the flow can move
repeatedly between large-diameter and almost-flat regimes.  The late-time axis
is therefore decomposed into connected components of the normalized-diameter
superlevel set.  Theorem~\ref{thm:large-component-iteration} was formulated to
handle exactly such a component: it starts from one sufficiently late Bamler
slice, retains that fibration on all complete blocks in the component, creates
uniform scalar data after the first restart, propagates the quotient-length
bound, and controls a terminal partial block by the global scale-invariant
curvature estimate.  Thus no separate Nil-specific iteration is needed.

\subsection{Large-diameter components and almost-flat gaps}
For a finite large-diameter component, choose a starting time \(S\) just to
the right of its left endpoint.  The short interval before \(S\) is controlled
directly by continuity of the normalized diameter.  The complete geometric
blocks and the terminal partial block are then controlled by
Theorem~\ref{thm:large-component-iteration}; the complementary gaps lie in the
almost-flat alternative.  The schematic decomposition is shown in
Figure~\ref{fig:nil-large-component}.  If the component has no first
uncontrolled block, the controlled blocks continue indefinitely.

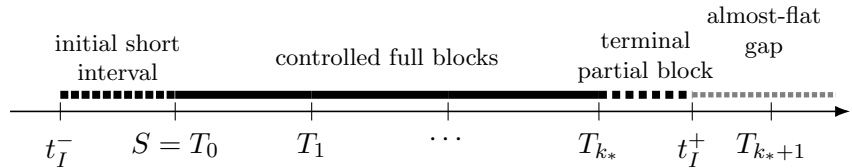
\begin{figure}[htbp]
\centering
\begin{tikzpicture}[x=0.95cm,y=1cm,font=\small]
  \draw[-{Latex[length=2mm]},semithick] (0,0)--(11.7,0);
  \foreach \x/\lab in {0.7/$t_I^-$,2.3/$S=T_0$,4.2/$T_1$,6.1/$\cdots$,8.2/$T_{k_*}$,9.5/$t_I^+$,10.6/$T_{k_*+1}$}{
    \draw (\x,-.12)--(\x,.12);
    \node[below=4pt] at (\x,0) {\lab};
  }
  \draw[line width=3pt,densely dotted] (0.7,.22)--(2.3,.22);
  \draw[line width=3pt] (2.3,.22)--(4.2,.22);
  \draw[line width=3pt] (4.2,.22)--(6.1,.22);
  \draw[line width=3pt] (6.1,.22)--(8.2,.22);
  \draw[line width=3pt,densely dashed] (8.2,.22)--(9.5,.22);
  \draw[line width=2pt,densely dotted,gray] (9.5,.22)--(11.5,.22);
  \node[font=\scriptsize,align=center] at (1.5,.72) {initial short\\ interval};
  \node[font=\scriptsize,align=center] at (5.25,.72) {controlled full blocks};
  \node[font=\scriptsize,align=center] at (8.85,.72) {terminal\\ partial block};
  \node[font=\scriptsize,align=center] at (10.5,1.05) {almost-flat\\ gap};
\end{tikzpicture}
\caption{Schematic decomposition of a finite large-diameter component in the
Nil case.  Theorem~\ref{thm:large-component-iteration} begins at \(S=T_0\) and
runs on every full block contained in the component; its diameter conclusion
also covers the terminal partial block.  If \(T_{k_*}\) is the first
uncontrolled start time, connectedness gives
\(t_I^+\leq T_{k_*+1}=100T_{k_*}\), so Type-III distortion controls the
terminal interval.}
\label{fig:nil-large-component}
\end{figure}

\subsection{The global diameter bound and Nil blowdown}

\begin{theorem}[Nil diameter bound and blowdown]
\label{thm:nil-diameter-blowdown}
Let \(g(t)\), \(t\geq0\), be an immortal Ricci flow on the orientable total
space \(M\) of a \(\TT^2\)-bundle over \(S^1\) with nontrivial unipotent
monodromy, under the
\hyperref[par:working-cover-convention]{working-cover convention}.  Then there
is a constant \(C<\infty\) such that
\[
        \diam(M,g(t))\leq C\sqrt t,
\]
for all sufficiently large \(t\).  In addition,
\[
        \lim_{t\to\infty}(M,t^{-1}g(t))=\pt
\]
in the Gromov--Hausdorff sense.  On every compact \(\tau\)-interval in
\((0,\infty)\), the lifted blowdown flows on the universal cover converge as
\(s\to\infty\) in the pointed Cheeger--Hamilton sense to a pointed flow
isometric to
\[
        \left(\RR^3,0,
        \frac{1}{3\tau^{1/3}}
        \left(dx+\frac12 y\,dz-\frac12 z\,dy\right)^2
        +\tau^{1/3}(dy^2+dz^2)
        \right).
\]
\end{theorem}

\begin{proof}
Fix \(N\geq2\) and \(\sigma\in(0,1)\), and set
\[
        \mathcal D(t)=t^{-1/2}\diam(M,g(t)).
\]
Let \(\eta_0\) and \(T_{\mathrm{p}}\) be the constants in
Theorem~\ref{thm:large-component-iteration}, applied with the
fixed global application \(c_0=c_{\mathrm{af}}\) and
\(\delta_0=c_{\mathrm{af}}\).  Choose
\(T_*\geq T_{\mathrm{p}}\) so large that, at every torus-bundle slice
\(t\geq T_*\), averaging the parabolically rescaled metric in Bamler's
fibration produces a Riemannian metric whose \(C^{N,\sigma}\) distance from the
rescaled metric is less than \(\eta_0\), as required by
\eqref{eq:component-start-error}.  Define
\[
        \mathcal L=\{t\geq T_*:\mathcal D(t)\geq c_{\mathrm{af}}\}.
\]
On \(\mathcal L^c\), the desired estimate holds with constant \(c_{\mathrm{af}}\), and the
slices are in the almost-flat alternative in
Proposition~\ref{input:bamler-fixed-threshold}.  At every time in
\(\mathcal L\), the torus-bundle comparison is available by the other case of
the same dichotomy.  It remains to control each connected component of
\(\mathcal L\).

Let \(I\) be a nontrivial component with finite left endpoint
\(t_I^->T_*\).  By continuity, \(\mathcal D(t_I^-)=c_{\mathrm{af}}\).  Choose \(S\in I\)
close enough to \(t_I^-\) that
\[
        \mathcal D(t)\leq2c_{\mathrm{af}}
        \qquad\text{for }t\in I\cap[t_I^-,S].
\]
At time \(S\), use the torus-bundle fibration and average in that same
fibration.  By the choice of \(T_*\), the resulting start metric satisfies the
smallness hypothesis of
Theorem~\ref{thm:large-component-iteration}.  Apply that theorem with the fixed global application
\(c_0=c_{\mathrm{af}}\), with \(\delta_0=c_{\mathrm{af}}\) and
\(D_1=2c_{\mathrm{af}}\).  Its constants are uniform over all such
components.  The short interval \([t_I^-,S]\cap I\) is already controlled by
\(\mathcal D\leq2c_{\mathrm{af}}\).  If \(I\) has a finite right endpoint \(t_I^+\), then
\(\mathcal D(t_I^+)=c_{\mathrm{af}}\); the component theorem already includes
that endpoint.  Singleton components also
satisfy \(\mathcal D=c_{\mathrm{af}}\).

There may be one component with left endpoint \(T_*\).  If it is bounded, it
does not affect the large-time estimate.  If it is unbounded, choose one
sufficiently large \(S\) in it and set \(D_1=\mathcal D(S)<\infty\).
Applying Theorem~\ref{thm:large-component-iteration} with the fixed global
application \(c_0=c_{\mathrm{af}}\) at that \(S\) controls the component from
\(S\) onward;
the constant may depend on this
single finite value of \(D_1\).  The bounded interval before \(S\) is
irrelevant.

Combining the estimates on \(\mathcal L\) and \(\mathcal L^c\) gives
\[
        \diam(M,g(t))=O(\sqrt t).
\]
Together with the Type-III curvature estimate, this is precisely the
hypothesis of Proposition~\ref{input:lott-blowdown}.  Its Nil case gives the
Gromov--Hausdorff collapse of \((M,t^{-1}g(t))\) to a point and the pointed
Cheeger--Hamilton blowdown to the Nil expanding soliton displayed in the
statement.
\end{proof}

\section{Proof of the main theorem}
\label{sec:main-theorem-proof}
The purpose of the
section is to return from the fixed working cover to the original manifold and
assemble the Euclidean, Nil, and Sol conclusions into the three
parts of the main theorem.

\begin{proof}[Proof of Theorem~\ref{thm:main}]
For this proof only, write \(M_0\) for the manifold denoted by \(M\) in the
statement.  Apply Proposition~\ref{prop:finite-torus-bundle-covers} once and
fix the finite regular orientable torus-bundle cover
\[
        \pi_{\mathrm{fin}}:M_{\mathrm{fin}}\longrightarrow M_0
\]
supplied by its stronger conclusion.  Let
\(g_{\mathrm{fin}}(t)=\pi_{\mathrm{fin}}^*g(t)\).  By construction,
\((M_{\mathrm{fin}},g_{\mathrm{fin}}(t))\) satisfies the
\hyperref[par:working-cover-convention]{working-cover convention}.  Apply
Proposition~\ref{input:bamler-fixed-threshold} directly to this lifted flow
with the fixed global choice \(c_0=c_{\mathrm{af}}\).
This cover is fixed for the remainder of the proof, and each case conclusion
is descended only after the corresponding cover-level theorem has been
applied.

\medskip\noindent\emph{The Euclidean case.}
The Euclidean conclusion of
Theorem~\ref{thm:euclidean-sol-long-time-limits} gives exponential convergence of
\(g_{\mathrm{fin}}(t)\) to a flat metric.  Part~\ref{item:finite-cover-smooth-descent}
of Proposition~\ref{prop:finite-cover-descent} descends the limit and
all fixed-order exponential estimates to \(M_0\).

The \(C^2\)-convergence estimate gives
\(|\Rm_{g_{\mathrm{fin}}(t)}|=O(e^{-\lambda t})\) for some \(\lambda>0\),
while the diameter remains bounded.  Proposition~\ref{input:lott-blowdown}
therefore gives point collapse of
\((M_{\mathrm{fin}},t^{-1}g_{\mathrm{fin}}(t))\) and the static Euclidean
blowdown on the universal cover.  Parts~\ref{item:finite-cover-point-descent}
and \ref{item:finite-cover-universal-cover} of
Proposition~\ref{prop:finite-cover-descent} give the corresponding
conclusions on \(M_0\).  This proves part~(1).

\medskip\noindent\emph{The Nil case.}
Applying
Theorem~\ref{thm:nil-diameter-blowdown} to \(g_{\mathrm{fin}}(t)\) gives point
collapse on the compact cover and the Nil expanding-soliton blowdown on the
common universal cover.  Parts~\ref{item:finite-cover-point-descent} and
\ref{item:finite-cover-universal-cover} of
Proposition~\ref{prop:finite-cover-descent} descend these conclusions.
This proves part~(2).

\medskip\noindent\emph{The Sol case.}
The Sol conclusion of
Theorem~\ref{thm:euclidean-sol-long-time-limits} gives a metric-circle limit on the
compact cover and the Sol blowdown on the common universal cover.
Parts~\ref{item:finite-cover-circle-descent} and
\ref{item:finite-cover-universal-cover} of
Proposition~\ref{prop:finite-cover-descent} show that the limit on \(M_0\) is
a metric circle or a compact interval of positive length; they also preserve
the universal-cover limit.  This proves part~(3).
\end{proof}

\appendix
\section{Uniform lifted parabolic estimates}
\label{app:parabolic-estimates}

This appendix provides the analytic foundation that makes the main iteration
uniform in the collapsing regime.  All collapse-uniform estimates are formulated on the
universal cover with constants controlled by finite-order lifted geometry,
rather than by the injectivity radius of the compact quotient.  It develops the
linear and quasilinear Schauder theory and establishes the
continuous-dependence, positive-time smoothing, and regular-restart results
invoked in Sections~\ref{sec:finite-time-decay} and~\ref{sec:stage-restart}.

\paragraph{\normalfont\itshape Standard inputs and deductions made here.}
The imported analytic ingredients are the harmonic-radius estimate of
Anderson--Cheeger \cite[Theorem~0.3 and Section~1]{AndersonCheeger}, the
harmonic-coordinate elliptic identity and regularity theorem of
DeTurck--Kazdan \cite[Lemma~4.1 and Theorem~4.5(b)]{DeTurckKazdan}, the scalar
parabolic estimates from Lieberman itemized in
Remark~\ref{rem:lieberman-scalar-versus-system}, and DeTurck's compact-manifold
local existence argument \cite[Theorem and proof, pp.~158--161]{DeTurck}.
The fixed time-slab atlas, its uniform overlap and cutoff bounds, the passage
to scalar-principal tensor systems, the lifted patching estimates, the
finite-order induction, and the collapse-independent restart bounds are the
deductions carried out in this appendix.

The appendix is organized as follows.  The opening unnumbered subsection,
\emph{Estimates used in the body}, displays the precise estimates
imported by the body and the chain from a bounded initial slice to uniform
parabolic geometry, linear control, stage stability, and regular restart.  The
subsection \emph{Parabolic H\"older spaces and lifted bounded geometry} fixes
the Euclidean and intrinsic parabolic norms, the metric-compatible time
derivative, and the uniformly local convention on the universal cover.  The
subsection \emph{From slice geometry to parabolic coordinate systems}
constructs a fixed-coordinate harmonic atlas with uniform overlap and cutoff
bounds, proves equivalence of intrinsic and chart norms, and propagates a
\(\BG_m\) initial-face bound to a full \(\PBG_m\) bound on a finite stage.  The
subsection \emph{Scalar-principal systems and coefficient regularity} derives
the difference equation for two Ricci--DeTurck solutions and gives the
finite derivative counts for its principal, first-order, and zeroth-order
coefficients.  The subsection \emph{Uniform Schauder estimates on lifted
cylinders} proves the required finite-order Cauchy--Dirichlet estimate,
including the initial and artificial boundary faces, and patches the local
estimates to global and interior collapse-independent estimates on the
universal cover.  The subsection \emph{Quasilinear derivative gain} obtains
one positive-time derivative at a time for Ricci--DeTurck systems by
differentiating the equation and applying the linear theory on nested
cylinders.  The final subsection, \emph{Continuation, smoothing, and regular
restart}, closes the uniform continuous-dependence bootstrap, uses a
fixed-quotient continuation criterion only after all collapse-independent
bounds have been established, gains the five endpoint derivatives needed by
the body, and verifies the rescaled averaged restart geometry.  The spaces
\(X_s^m\), \(Y_s^m\), \(\BG_m\), and \(\PBG_m\) were introduced in
Section~\ref{sec:finite-time-decay} and are used here with the precise
conventions given below.

\subsection*{Estimates used in the body}
The body uses the following estimates from the appendix.
All intrinsic norms are evaluated after lifting to the universal cover.
Whenever a coordinate estimate is used, \(\PBG\) provides fixed lifted charts
whose norms are uniformly equivalent to the intrinsic ones; the resulting
constants are therefore independent of collapse of the compact quotient.

\begin{center}
\small
\renewcommand{\arraystretch}{1.15}
\begin{tabular}{@{}>{\raggedright\arraybackslash}p{0.22\textwidth}>{\raggedright\arraybackslash}p{0.31\textwidth}>{\raggedright\arraybackslash}p{0.39\textwidth}@{}}
\textbf{Result} & \textbf{Hypotheses} & \textbf{Conclusion and use in the body}\\ \hline
Lemma~\ref{lem:slice-to-parabolic-geometry} & A stage-start bound \(\BG_m(s(a);\Lambda_0)\), a sectional-curvature bound, and a bounded time interval. & A full parabolic bound \(\PBG_m(s;\Lambda_1;[a,b])\); this provides the fixed lifted charts for every stage.\\
Proposition~\ref{prop:uniform-lifted-linear-schauder} & \(\PBG_m\), a finite tensor system with scalar principal part, and controlled coefficients. & Collapse-independent global and interior Schauder estimates; used for translation differences and endpoint smoothing.\\
Proposition~\ref{prop:uniform-lifted-rdt-stability} & \(\PBG_N\) and a sufficiently small \(C^{N,\sigma}\) initial perturbation. & Existence on the whole stage and a linear \(X_s^N\) stability estimate; used in stage preparation.\\
Proposition~\ref{prop:positive-time-regular-restart} & Stage closeness in \(X_r^N\) and curvature control for the locally \(\TT^2\)-invariant background. & A \(C^{N+5,\sigma}\)-small rescaled endpoint, an averaged restart, and uniform \(\BG_N\); used at every switch.\\
\end{tabular}
\end{center}

The estimates are related by
\[
\begin{gathered}
 \BG_m\text{ at the initial face}
 \xrightarrow{\text{Lemma~\ref{lem:slice-to-parabolic-geometry}}}
 \PBG_m\text{ on the stage},\\[3pt]
 \PBG_m+\text{scalar principal part}
 \xrightarrow{\text{Proposition~\ref{prop:uniform-lifted-linear-schauder}}}
 \text{uniform linear control},\\[3pt]
 \text{uniform linear control}+\text{quasilinear bootstrap}
 \Longrightarrow \text{stage stability and regular restart}.
\end{gathered}
\]
Lemma~\ref{lem:uniform-parabolic-atlas} constructs the fixed-coordinate atlas,
and Lemma~\ref{lem:local-rdt-one-derivative-gain} provides the finite number of
positive-time derivative gains.  The remainder of the appendix proves each
implication displayed above.

\subsection{Parabolic H\"older spaces and lifted bounded geometry}
\label{subsec:lifted-parabolic-estimates}

We first fix the analytic conventions used throughout the paper.  Spatial
H\"older spaces on a time slice are denoted by \(C^{m,\sigma}\).  On a
Euclidean parabolic cylinder \(Q\), put
\[
 d_{\mathrm{p}}((x,t),(y,s))=\max\{|x-y|,|t-s|^{1/2}\}.
\]
For a scalar function, or for a component of a tensor in a fixed frame, set
\[
\begin{gathered}
 {}[v]^{\mathrm{p}}_{\sigma;Q}
 =\sup_{P\ne P'\in Q}\frac{|v(P)-v(P')|}{d_{\mathrm{p}}(P,P')^\sigma},\\[3pt]
 \langle v\rangle_{\theta,t;Q}
 =\sup_{\substack{(x,t),(x,s)\in Q\\t\ne s}}
   \frac{|v(x,t)-v(x,s)|}{|t-s|^{\theta/2}}.
\end{gathered}
\]
For an integer \(m\geq0\) and \(\sigma\in(0,1)\), define
\begin{equation}\label{eq:parabolic-holder-norm-definition}
\begin{split}
 \|v\|_{C^{m+\sigma,(m+\sigma)/2}(Q)}
 ={}&\sum_{|\nu|+2j\leq m}
       \|\partial_x^\nu\partial_t^jv\|_{C^0(Q)}\\
 &+\sum_{|\nu|+2j=m}
       [\partial_x^\nu\partial_t^jv]^{\mathrm{p}}_{\sigma;Q}\\
 &+\sum_{|\nu|+2j=m-1}
       \langle\partial_x^\nu\partial_t^jv\rangle_{1+\sigma,t;Q},
\end{split}
\end{equation}
where the last sum is empty when \(m=0\).  For a finite tensor system, the
norm is the sum of the component norms in the chosen fixed frame.  This is
Lieberman's \(H_{m+\sigma}\)-norm from
\cite[Chapter~IV, Section~1]{Lieberman}, written in the
\(C^{m+\sigma,(m+\sigma)/2}\) notation used here.  Thus one spatial
derivative has parabolic weight one and one time derivative has weight two.

\paragraph{\normalfont\itshape Norm convention.}\label{par:lifted-norm-convention}
The spaces \(X_s^m\) and \(Y_s^m\) are defined intrinsically, before any
bounded-geometry hypothesis is imposed.  Lift \(s(t)\) and a tensor \(U\) to
the universal cover.  For a time-dependent vector field \(V\), define the
metric-compatible covariant time derivative
\begin{equation}\label{eq:metric-compatible-time-derivative}
        \mathcal D_t^sV
        =\partial_tV+\frac12s(t)^{-1}(\partial_ts(t))V,
\end{equation}
and extend \(\mathcal D_t^s\) to arbitrary tensor fields by duality and the
Leibniz rule.  Its time-parallel transport
\(\mathsf T^s_{u\to t}(x)\) is an isometry from the tensor norm induced by
\(s(u)\) at \(x\) to that induced by \(s(t)\).

Set
\begin{equation}\label{eq:intrinsic-holder-radius}
 r_s=\frac18\min\left\{1,
       \inf_{(x,t)\in\widetilde M\times[a,b]}
       \inj_{\widetilde s(t)}(x)\right\}.
\end{equation}
For the fixed smooth family on the compact time interval, the infimum is
positive.  This canonical choice also has a uniform positive lower bound
whenever a \(\PBG\) bound is imposed.  If \(V\) is a time-dependent tensor, let
\begin{align}
 [V]^x_{\alpha;s}
 &=\sup_t\sup_{0<d_{s(t)}(x,y)<r_s}
   \frac{\left|V(x,t)-\mathsf P^{s(t)}_{y\to x}V(y,t)\right|_{s(t)}}
        {d_{s(t)}(x,y)^\alpha},                                    \\
 \langle V\rangle^t_{\alpha;s}
 &=\sup_x\sup_{u\ne t}
   \frac{\left|V(x,t)-\mathsf T^s_{u\to t}(x)V(x,u)\right|_{s(t)}}
        {|t-u|^{\alpha/2}},                                        \\
 [V]^p_{\alpha;s}&=[V]^x_{\alpha;s}+\langle V\rangle^t_{\alpha;s}.
\end{align}
Here \(\mathsf P^{s(t)}_{y\to x}\) denotes parallel transport along the
unique short \(s(t)\)-geodesic.  Replacing the numerical factor \(1/8\) in
\eqref{eq:intrinsic-holder-radius} by any fixed smaller positive factor gives
an equivalent norm.

For an integer \(m\geq0\), define
\begin{equation}
\begin{split}
 \|U\|_{\mathcal C_s^{m+\sigma,(m+\sigma)/2}[a,b]}
 ={}&\sum_{k+2j\leq m}
   \left\|(\nabla^{s(t)})^k(\mathcal D_t^s)^jU\right\|_{C^0_s}\\
 &+\sum_{k+2j=m}
   \left[(\nabla^{s(t)})^k(\mathcal D_t^s)^jU\right]^p_{\sigma;s}\\
 &+\sum_{k+2j=m-1}
   \left\langle(\nabla^{s(t)})^k(\mathcal D_t^s)^jU
   \right\rangle^t_{1+\sigma;s},
\end{split}
\end{equation}
where the last sum is empty when \(m=0\), and \(C^0_s\) denotes the
space-time supremum measured with \(s(t)\).  Taking the supremum on the
universal cover defines the deck-invariant norm.  The spaces in
\eqref{eq:XY-spaces} are
\[
 \|U\|_{X_s^m[a,b]}
 =\|U\|_{\mathcal C_s^{m+\sigma,(m+\sigma)/2}[a,b]},
 \qquad
 \|U\|_{Y_s^m[a,b]}
 =\|U\|_{\mathcal C_s^{m-2+\sigma,(m-2+\sigma)/2}[a,b]}.
\]
Thus they are well defined for every smooth reference flow, even before a
uniform atlas is specified.

For a tensor on the universal cover that need not be deck-invariant, we retain
\[
 \|U\|_{C^{m+\sigma,(m+\sigma)/2}_{\mathrm{uloc}}}
\]
for the supremum of the Euclidean component norms
\eqref{eq:parabolic-holder-norm-definition} in a fixed lifted parabolic atlas.
This notation is used for translation differences only after the required
\(\PBG\) bound has been established.

\subsection{From slice geometry to parabolic coordinate systems}

\begin{lemma}[Uniform parabolic atlas from bounded geometry]
\label{lem:uniform-parabolic-atlas}
Fix \(n\), \(m\geq2\), \(\sigma\in(0,1)\), \(\Lambda\geq1\), and
\(L>0\).  Let \(s(t)\), \(a\leq t\leq b\), be a Ricci flow with
\(b-a\leq L\), and suppose that the inequalities in
\eqref{eq:bounded-lifted-geometry} hold with this \(\Lambda\).  Then there
are \(r_*>0\) and \(A_*<\infty\), depending only on
\(n,m,\sigma,\Lambda,L\), with the following properties.  For every
\((\widetilde x_0,t_0)\), there are harmonic
coordinates for \(\widetilde s(t_0)\) on
\(B_{\widetilde s(t_0)}(\widetilde x_0,4r_*)\).  Keeping the spatial
coordinates fixed, they remain valid on
\[
 B_{\widetilde s(t_0)}(\widetilde x_0,4r_*)
 \times\bigl([t_0-(4r_*)^2,t_0+(4r_*)^2]\cap[a,b]\bigr).
\]
Parabolic estimates may therefore be applied on backward, forward, or
endpoint-truncated subcylinders contained in this set.  In these coordinates
the coefficients of \(\widetilde s\), \(\widetilde s^{-1}\), and the
Christoffel symbols have the finite parabolic H\"older bounds through order
\(m\) needed below, all bounded by \(A_*\).  On overlaps, the transition maps
and their inverses have uniform finite-order parabolic H\"older bounds through
the same order.  Radius-\(r_*/4\) one-sided or truncated subcylinders can be
chosen to cover space-time, while the radius-\(2r_*\) cylinders have uniformly
bounded overlap and admit spatial cutoffs with uniform derivative bounds.
\end{lemma}

\begin{proof}
Choose \(p>2n\) so large that
\(1-2n/p>\sigma\).  The local \(L^{1,p}\) harmonic-radius estimate of
Anderson--Cheeger
\cite[Theorem~0.3 and Section~1, pp.~268--274]{AndersonCheeger}, applied on
the universal cover after rescaling on the injectivity-radius scale, gives a
radius \(r_h>0\) on which harmonic coordinates have uniform ellipticity and
uniform \(W^{1,p}\)-bounds.  Qualitative existence and regularity of the
harmonic coordinate functions are given by
\cite[Lemma~1.2 and Corollary~1.4]{DeTurckKazdan}; the radius and uniform
constants used here come from the Anderson--Cheeger estimate and the
quantitative bootstrap below.  In harmonic coordinates the metric satisfies
the uniformly elliptic system
\[
 s^{ab}\partial_a\partial_b s_{ij}
 =-2\Ric_{ij}(s)+\mathcal Q_{ij}(s^{-1},\partial s).
\]
The first interior \(W^{2,p/2}\)-estimate gives uniform
\(C^{1,\alpha}\)-control for some \(\alpha\in(\sigma,1)\).  The bounds for
\(\nabla^j\Rm\), converted inductively to coordinate bounds using the already
controlled Christoffel symbols, and repeated interior elliptic estimates for
this system then bootstrap the center-time coefficients through the finite
order needed here.  This is the quantitative harmonic-coordinate regularity
argument of
\cite[Lemma~4.1 and Theorem~4.5(b), pp.~254--256]{DeTurckKazdan}.  All of
these estimates are on the universal cover and depend only on
\(n,m,\sigma,\Lambda\).

Since \(|\Ric(s)|_s\leq n|\Rm(s)|_s\) and
\(\partial_t s=-2\Ric(s)\), the curvature bound gives
\(|\partial_t s|_s\leq 2n\Lambda\).  Choose \(r_*<r_h/16\) and put
\(\delta_*=(4r_*)^2\), decreasing \(r_*\) so that
\(e^{2n\Lambda\delta_*}\leq4/3\).  It follows that
\[
 e^{-2n\Lambda|t-t_0|}s(t_0)\leq s(t)\leq
 e^{2n\Lambda|t-t_0|}s(t_0).
\]
Thus a harmonic chart chosen once at time \(t_0\) on the
\(s(t_0)\)-ball of radius \(8r_*\) is kept fixed as a spatial chart throughout
\([t_0-\delta_*,t_0+\delta_*]\cap[a,b]\); the metric comparison ensures that
all cylinders in the statement remain inside its original spatial domain.
No time-dependent coordinate change is made.

For completeness, put \(\ell_*=(r_*/4)^2\) and divide \([a,b]\) into
at most \(2+L/\ell_*\) closed time intervals \(J_\nu\), each of length at
most \(\ell_*\), with center time \(t_\nu\).  At each level choose a maximal
\(r_*/8\)-separated set in the universal cover and harmonic charts of radius
\(8r_*\) at its points.  The radius-\(r_*/8\) balls at the center time cover
space.  Keep every chart fixed on
\(I_\nu=[t_\nu-\delta_*,t_\nu+\delta_*]\cap[a,b]\).  Since
\(J_\nu\subset I_\nu\), the products of the covering balls with \(J_\nu\)
are covered by a fixed number of radius-\(r_*/4\) backward, forward, or
endpoint-truncated parabolic cylinders in these fixed charts.  They therefore
cover space-time; call them the selected subcylinders.

Curvature bounds, the injectivity-radius lower bound, volume comparison, and
the disjointness of the radius-\(r_*/16\) balls about the chosen centers give a
uniform spatial overlap bound for the radius-\(2r_*\) cylinders.  Only a
uniformly bounded number of the intervals \(I_\nu\) meet any given time, so the
space-time overlap is uniform as well.  The spatial families may be infinite
on the universal cover, but they are uniformly locally finite; none of these
constants involves the compact quotient.

It remains to establish the coefficient bounds in these fixed coordinates.  At
\(t_\nu\) they follow from the preceding harmonic-coordinate bootstrap.  On
\(I_\nu\), differentiate \(\partial_t s=-2\Ric(s)\) in the fixed coordinates.
More explicitly, the Ricci-flow evolution and commutation formulas used in
Shi's derivative estimates \cite{Shi} express, for \(j\geq1\),
\(\partial_x^\mu\partial_t^j s\) in terms of lower-weight coordinate
derivatives of \(s\) and \(\nabla^q\Rm(s)\) with
\(q\leq |\mu|+2j-2\).  The center-time harmonic-coordinate bounds, induction
on \(|\mu|+2j\), and Gr\"onwall for \(j=0\) therefore give
\[
 \sup |\partial_x^\mu\partial_t^j s_{kl}|\leq C
 \qquad\text{whenever }|\mu|+2j\leq m+2.
\]
The same bounds hold for \(s^{-1}\), and the formula for the Christoffel
symbols gives the required bounds for them.  The bounds through parabolic
weight \(m+2\) provide the extra spatial and time derivatives needed in
\eqref{eq:parabolic-holder-norm-definition}; Euclidean mean-value inequalities
then give its H\"older seminorms.

It remains to justify the corresponding bounds for transition maps.  The
time projection of each selected subcylinder lies in one of the intervals
\(J_\nu\).  If two selected subcylinders overlap, their associated intervals
are equal or adjacent, so their center times differ by at most
\(2\ell_*<\delta_*\); each of the two fixed charts is therefore valid at the
center time of the other.  Let \(x\) be a chart chosen at time \(t_\nu\) and
\(y\) one chosen at time \(t_\mu\).  On a uniform neighborhood of such an
overlap, each component of \(y\circ x^{-1}\) satisfies, in the
\(x\)-coordinates, the scalar equation
\[
        \Delta_{\widetilde s(t_\mu)}y^\alpha=0,
\]
because the \(y\)-coordinates are harmonic for \(\widetilde s(t_\mu)\).
The coefficients of this equation in the \(x\)-chart are uniformly elliptic
and have the finite-order bounds already proved.  The radius-\(8r_*\) chart
domains and the radius-\(2r_*\) enlargements of the selected subcylinders leave
a uniform interior margin, and the harmonic-coordinate normalization gives a
uniform \(C^0\)-bound for the coordinate functions.  Interior elliptic
Schauder estimates, iterated a finite number of times, therefore give uniform
spatial H\"older bounds for \(y\circ x^{-1}\) through the required order on the
overlap.
Interchanging \(x\) and \(y\) gives the same bounds for the inverse transition
map.  Since both coordinate systems are kept fixed in time, these transition
maps have zero time derivatives, and the spatial estimates are precisely the
required parabolic H\"older bounds.  Thus the overlap-map bounds are not an
additional compactness theorem: they follow from scalar interior elliptic
regularity on the uniform interior margins just constructed.

Finally, no quotient-dependent partition-of-unity estimate is used.  Choose
one fixed Euclidean cutoff on the model chart, rescale it by the already fixed
radius \(r_*\), and pull it back through each harmonic chart.  The coefficient
bounds above then give uniform spatial derivative bounds for these cutoffs.
This proves every assertion of the lemma, with constants independent of
collapse downstairs.
\end{proof}

\begin{lemma}[Equivalence of intrinsic and fixed-chart norms]
\label{lem:intrinsic-fixed-chart-equivalence}
Fix \(n\), \(m\geq2\), \(\sigma\in(0,1)\), \(\Lambda\geq1\), and \(L>0\).
Suppose that
\(\PBG_m(s;\Lambda;[a,b])\) holds and \(b-a\leq L\).  Let
\(\|\cdot\|_{X_{s,{\mathrm{fc}}}^m[a,b]}\) and
\(\|\cdot\|_{Y_{s,{\mathrm{fc}}}^m[a,b]}\) denote the suprema of the Euclidean
component norms in the fixed lifted atlas supplied by
Lemma~\ref{lem:uniform-parabolic-atlas}.  Then there is
\(C=C(n,m,\sigma,\Lambda,L)\) such that every deck-invariant tensor \(U\)
satisfies
\begin{align}
 C^{-1}\|U\|_{X_s^m[a,b]}
 &\leq \|U\|_{X_{s,{\mathrm{fc}}}^m[a,b]}
 \leq C\|U\|_{X_s^m[a,b]},                                      \label{eq:X-intrinsic-fixed-equivalence}\\
 C^{-1}\|U\|_{Y_s^m[a,b]}
 &\leq \|U\|_{Y_{s,{\mathrm{fc}}}^m[a,b]}
 \leq C\|U\|_{Y_s^m[a,b]}.                                      \label{eq:Y-intrinsic-fixed-equivalence}
\end{align}
The analogous comparison holds for every spatial
\(C^{k,\sigma}\)-norm with \(0\leq k\leq m\).  For tensors on the universal
cover that are not deck-invariant, the same local comparison holds with the
fixed-chart side interpreted as the uniformly local norm.
\end{lemma}

\begin{proof}
In the atlas of Lemma~\ref{lem:uniform-parabolic-atlas}, the metric and inverse
metric, the Christoffel symbols, the coefficients of the time connection
\(\mathcal D_t^s\), the transition maps, and the cutoffs have uniform
finite-order parabolic H\"older bounds.  Covariant space-time derivatives are
therefore linear combinations of coordinate derivatives of the same or lower
parabolic order, with uniformly controlled coefficients; the converse follows
by solving the same triangular relations for the coordinate derivatives.
Parallel transport along the short spatial geodesics and the time transport
of \eqref{eq:metric-compatible-time-derivative} differ from componentwise
identification by uniformly controlled lower-order terms.  The spatial and
temporal H\"older seminorms are consequently equivalent in both directions.
Taking suprema over the uniformly locally finite atlas proves
\eqref{eq:X-intrinsic-fixed-equivalence}--
\eqref{eq:Y-intrinsic-fixed-equivalence}; the proof is local and is therefore
unchanged for a tensor that is not deck-invariant.
\end{proof}

\begin{lemma}[Slice-to-parabolic bounds]
\label{lem:slice-to-parabolic-geometry}
Fix \(n,m\), \(\sigma\in(0,1)\), and \(\Lambda_0,K_0,L>0\), with
\(m\geq2\) and \(\Lambda_0\geq1\).  There is \(\Lambda_1<\infty\), depending only on
these data, with the
following property.  Let \(s(t)\), \(a\leq t\leq b\), be a Ricci flow on a
compact \(n\)-manifold, with \(b-a\leq L\), such that
\[
        \BG_m(s(a);\Lambda_0),
        \qquad
        \sup_{M\times[a,b]}|\sec_{s(t)}|\leq K_0.
\]
Then
\[
        \PBG_m(s;\Lambda_1;[a,b]).
\]
The constant is independent of the injectivity radius of the compact
quotient.
\end{lemma}

\begin{proof}
All curvature estimates may be lifted to the universal cover; since the
quantities being estimated are deck-invariant, the maximum-principle arguments
can be carried out on the compact quotient.  The curvature bound gives
\(|\Ric|\leq c_nK_0\), and hence
\begin{equation}\label{eq:slice-to-parabolic-metric-comparison}
 e^{-c_nK_0L}s(a)\leq s(t)\leq e^{c_nK_0L}s(a)
 \qquad (a\leq t\leq b).
\end{equation}

We first establish the curvature-derivative bounds at the lower time face.  For
\(j\geq1\), the standard differentiated curvature evolution equation has the
schematic maximum-principle form
\[
\begin{aligned}
 (\partial_t-\Delta)|\nabla^j\Rm|^2
 \leq{}&-|\nabla^{j+1}\Rm|^2
      +C_j|\Rm|\,|\nabla^j\Rm|^2\\
 &+C_j\!\sum_{\substack{p+q=j\\p,q<j}}
          |\nabla^p\Rm|\,|\nabla^q\Rm|\,|\nabla^j\Rm|.
\end{aligned}
\]
Starting with the bounds in \(\BG_m(s(a);\Lambda_0)\), induction on \(j\) and the
maximum principle give uniform estimates for
\(|\nabla^j\Rm|\), \(0\leq j\leq m+3\), on a fixed initial interval whose
length depends only on \(n,m,\Lambda_0,K_0\).  On the part of \([a,b]\) separated
from \(a\) by half that interval, Shi's derivative estimates \cite{Shi},
applied to the curvature bound on \([a,b]\), give the same finite collection
of estimates.
If \(b-a\) is shorter than the fixed initial interval, the first argument
already covers the whole flow.  Thus
\begin{equation}\label{eq:slice-to-parabolic-curvature-derivatives}
        \sup_{M\times[a,b]}|\nabla^j\Rm_{s(t)}|\leq A_1,
        \qquad 0\leq j\leq m+3,
\end{equation}
for a constant \(A_1\) depending only on the stated data.

It remains to make the lifted injectivity estimate uniform in time.  At time
\(a\), the injectivity and curvature bounds in \(\BG_m(s(a);\Lambda_0)\) give
numbers \(\rho,v_0>0\), depending only on \(n,\Lambda_0\), such that
\[
 \operatorname{vol}_{\widetilde s(a)}
 B_{\widetilde s(a)}(\widetilde x,r)\geq v_0r^n
 \qquad(0<r\leq\rho).
\]
For example, one may take a fixed fraction of \(\Lambda_0^{-1}\) for \(\rho\) and
use normal-coordinate volume comparison.  Put \(A=e^{c_nK_0L}\).  The metric
comparison \eqref{eq:slice-to-parabolic-metric-comparison} gives
\[
 B_{\widetilde s(a)}(\widetilde x,A^{-1/2}\rho)
 \subset B_{\widetilde s(t)}(\widetilde x,\rho),
 \qquad
 d\mu_{\widetilde s(t)}\geq A^{-n/2}d\mu_{\widetilde s(a)}.
\]
Thus the \(\widetilde s(t)\)-volume of the radius-\(\rho\) ball is at least
\(v_0A^{-n}\rho^n\), uniformly for \(t\in[a,b]\).  Together with the
sectional-curvature bound, the Cheeger--Gromov--Taylor estimate
\cite[Section~4]{CGT} gives
\begin{equation}\label{eq:slice-to-parabolic-injectivity}
        \inf_{\widetilde M\times[a,b]}
        \inj_{\widetilde s(t)}\geq i_1>0,
\end{equation}
where \(i_1\) again depends only on the fixed data.

Equations \eqref{eq:slice-to-parabolic-curvature-derivatives} and
\eqref{eq:slice-to-parabolic-injectivity} give the curvature and
injectivity-radius part of \(\PBG_m\).  Applying
Lemma~\ref{lem:uniform-parabolic-atlas} and increasing one constant to include
its radius, coefficient, and time-H\"older bounds gives
\(\PBG_m(s;\Lambda_1;[a,b])\).  This also makes explicit that the coordinate control
at the initial time comes from the propagated initial curvature derivatives,
whereas the later-time control comes from Shi's estimates.
\end{proof}

\subsection{Scalar-principal systems and coefficient regularity}
We use the Ricci--DeTurck equation \eqref{eq:rdt} and the parabolic spaces
in \eqref{eq:XY-spaces}, both introduced in
Section~\ref{sec:finite-time-decay}.

\begin{lemma}[Difference equation for two Ricci--DeTurck solutions]
\label{lem:rdt-difference-system}
Let \(s(t)\) be a Ricci-flow background and let \(v_0(t)\) and \(v_1(t)\)
be two positive definite solutions of the Ricci--DeTurck equation relative to
\(s(t)\) on the same time interval.  Put
\[
        U=v_1-v_0,
        \qquad v_\vartheta=v_0+\vartheta(v_1-v_0),
        \quad 0\leq\vartheta\leq1,
\]
and assume that every \(v_\vartheta\) is positive definite.  With \(\nabla\)
the connection of \(s(t)\), the difference satisfies a homogeneous linear
system
\begin{equation}\label{eq:rdt-difference-system}
        \partial_tU-P^{pq}\nabla_p\nabla_qU
        =B^p\nabla_pU+CU,
\end{equation}
where the second-order part is diagonal in the tensor components: in a fixed
frame, every component \(U^A\) is acted on by the same operator
\(P^{pq}\nabla_p\nabla_q\), and no second derivative of a different component
occurs.  More precisely, for
\[
        \mathcal Q_s(v)=-2\Ric(v)+\Lie_{W(v,s)}v,
\]
one has the structural formula
\begin{equation}\label{eq:rdt-structural-form}
\begin{split}
        \mathcal Q_s(v)
        ={}&v^{pq}\nabla_p\nabla_qv
        +v^{-1}*v^{-1}*\nabla v*\nabla v \\
        &\quad+\mathcal R(\Rm(s),v,v^{-1}),
\end{split}
\end{equation}
and the mean-value identity
\begin{equation}\label{eq:mean-value-linearization}
\begin{split}
        \partial_tU
        &=\mathcal Q_s(v_1)-\mathcal Q_s(v_0)\\
        &=\int_0^1D\mathcal Q_s|_{v_\vartheta}[U] \,d\vartheta.
\end{split}
\end{equation}
Differentiating \eqref{eq:rdt-structural-form} gives
\begin{equation}\label{eq:linearized-rdt-structure}
        D\mathcal Q_s|_{v_\vartheta}[Z]
        =v_\vartheta^{pq}\nabla_p\nabla_qZ
          +\mathcal B_\vartheta^p\nabla_pZ+\mathcal C_\vartheta Z,
\end{equation}
with
\begin{equation}\label{eq:linearized-coefficient-structure}
\begin{split}
        \mathcal B_\vartheta
        &=v_\vartheta^{-1}*v_\vartheta^{-1}*\nabla v_\vartheta,\\
        \mathcal C_\vartheta
        &=v_\vartheta^{-1}*v_\vartheta^{-1}*\nabla^2v_\vartheta
          +v_\vartheta^{-1}*v_\vartheta^{-1}*v_\vartheta^{-1}
             *\nabla v_\vartheta*\nabla v_\vartheta\\
        &\qquad+\mathcal R_1(\Rm(s),v_\vartheta,v_\vartheta^{-1}).
\end{split}
\end{equation}
Here \(*\) denotes universal tensor contractions, and \(\mathcal R\) and
\(\mathcal R_1\) are algebraic in the displayed arguments; in particular,
\(\mathcal R_1\) contains no derivatives of \(v_\vartheta\).  The
coefficients in \eqref{eq:rdt-difference-system} are
\[
 P^{pq}=\int_0^1v_\vartheta^{pq}\,d\vartheta,
 \qquad B^p=\int_0^1\mathcal B_\vartheta^p\,d\vartheta,
 \qquad C=\int_0^1\mathcal C_\vartheta\,d\vartheta.
\]
Consequently, uniform metric equivalence and finite-order bounds for
\(v_0,v_1\), and \(s\) give the corresponding ellipticity and H\"older bounds
for \(P,B,C\).
\end{lemma}

\begin{proof}
Formula \eqref{eq:rdt-structural-form} is the usual covariant DeTurck
calculation \cite{DeTurck}; all second derivatives of the unknown metric occur
in its scalar principal term.  The fundamental theorem of calculus gives
\eqref{eq:mean-value-linearization}.  Differentiating the principal term gives
\[
 D\bigl(v^{pq}\nabla_p\nabla_qv\bigr)[Z]
 =v^{pq}\nabla_p\nabla_qZ
  -(v^{-1}Zv^{-1})^{pq}\nabla_p\nabla_qv,
\]
while
\[
\begin{split}
 D\bigl(v^{-1}*v^{-1}*\nabla v*\nabla v\bigr)[Z]
  ={}&v^{-1}*v^{-1}*\nabla v*\nabla Z\\
     &+v^{-1}*v^{-1}*v^{-1}*Z*\nabla v*\nabla v.
\end{split}
\]
The curvature term varies only algebraically in \(Z\).  This proves
\eqref{eq:linearized-rdt-structure}--\eqref{eq:linearized-coefficient-structure};
integrating in \(\vartheta\) proves \eqref{eq:rdt-difference-system} and the
last assertion.
\end{proof}

\begin{remark}[What is imported from Lieberman]
\label{rem:lieberman-scalar-versus-system}
The scalar analytic inputs are the interpolation inequality
\cite[Proposition~4.2, p.~49]{Lieberman}, the interior Schauder estimate
\cite[Theorem~4.9, p.~59]{Lieberman}, and the second-order
Cauchy--Dirichlet estimate \cite[Theorem~4.28, p.~77]{Lieberman}.  The
initial--lateral corner compatibility condition for a cylindrical domain is
stated explicitly in \cite[p.~78]{Lieberman}.
The finite-order version for the scalar-principal systems used here is proved
in Lemma~\ref{lem:finite-order-cauchy-dirichlet}, following the spatial
differentiation procedure indicated in \cite[Exercise~4.5(b),
p.~84]{Lieberman}.  For a localized cylinder
\(Q^+=\Omega\times[a,a+\theta]\), the \emph{artificial spatial boundary}
means \(\partial\Omega\times[a,a+\theta]\); it is introduced by
localization and is not a boundary of \(M\).  Theorem~4.28 is used below only
after localization to such a fixed smooth spatial cylinder.

A localization point used throughout is worth recording here.  Each spatial
cutoff is chosen identically zero in a collar of the artificial spatial
boundary.  After subtraction of an extension of the initial trace, the
localized unknown has zero data on both the initial and lateral faces, while
the cutoff-generated forcing vanishes to all orders at their intersection.
Hence the standard Cauchy--Dirichlet corner compatibility conditions are
automatic; no additional compatibility condition is imposed on the original
Cauchy data away from this artificial corner.

The passage from the scalar estimates to the finite tensor systems used below
is carried out in this paper.  In a fixed frame, \emph{scalar principal part}
means that the same uniformly parabolic operator
\(\partial_t-a^{ij}\partial_i\partial_j\) acts componentwise, with the
coefficient matrix \(a^{ij}\) carrying no tensor-component indices.
\end{remark}

\subsection{Uniform Schauder estimates on lifted cylinders}

\begin{lemma}[Finite-order Cauchy--Dirichlet estimate for scalar-principal systems]
\label{lem:finite-order-cauchy-dirichlet}
Fix \(m\geq2\) and \(\sigma\in(0,1)\).  Let
\(Q_i^+=\Omega_i\times[a,a+\theta]\), where
\(\Omega_1\Subset\Omega_2\) are smooth Euclidean domains.  Suppose that a
smooth finite-dimensional vector-valued function \(Z\) satisfies on
\(Q_2^+\)
\[
 \partial_tZ-a^{ij}\partial_i\partial_jZ
 =\mathsf B^i\partial_iZ+\mathsf CZ+G,
\]
where \(a^{ij}=a^{ji}\) is uniformly parabolic and
\(\partial_t-a^{ij}\partial_i\partial_j\) acts componentwise with the same
coefficient matrix on every component; only \(\mathsf B^i\) and
\(\mathsf C\) may couple components.  The coefficients
\(a^{ij},\mathsf B^i,\mathsf C\) have uniformly bounded
\(C^{m-2+\sigma,(m-2+\sigma)/2}\)-norms.  Assume that \(Z=0\) on the initial
face, and that \(Z\) and \(G\) vanish in a fixed spatial collar of
\(\partial\Omega_2\times[a,a+\theta]\).  Then
\[
 \|Z\|_{C^{m+\sigma,(m+\sigma)/2}(Q_1^+)}
 \leq C\left(
   \|G\|_{C^{m-2+\sigma,(m-2+\sigma)/2}(Q_2^+)}
   +\|Z\|_{C^0(Q_2^+)}\right),
\]
where \(C\) depends only on the displayed coefficient and parabolicity bounds,
\(m,\sigma,\theta\), the rank, and the fixed domains \(\Omega_1\Subset
\Omega_2\).
\end{lemma}

The componentwise second-order estimate and interpolation inequality are
imported from Lieberman; the nested-domain absorption and the finite-order
spatial-differentiation induction for the coupled lower-order terms are the
arguments supplied here.

\begin{proof}
Write \(I=[a,a+\theta]\) and
\(L_0=\partial_t-a^{ij}\partial_i\partial_j\).  We first record the
localized second-order estimate that will be used at every induction step.
Suppose that a finite system \(V\), with zero initial trace, satisfies
\[
 L_0V=\mathsf A^i\partial_iV+\mathsf DV+H
\]
on \(\Omega''\times I\), where
\(\Omega'\Subset\Omega''\Subset\Omega_2\) and the displayed coefficients
have uniform \(C^{\sigma,\sigma/2}\)-bounds.  Multiply by a spatial cutoff
which is one on \(\Omega'\) and zero in a collar of
\(\partial\Omega''\).  The localized unknown has zero data on the initial
face and on the artificial spatial boundary.  The scalar estimate
\cite[Theorem~4.28, p.~77]{Lieberman}, applied componentwise, together with
\cite[Proposition~4.2, p.~49]{Lieberman}, gives, for every
\(\varepsilon>0\),
\begin{equation}\label{eq:localized-cauchy-dirichlet-step}
\begin{split}
 \|V\|_{C^{2+\sigma,1+\sigma/2}(\Omega'\times I)}
 \leq{}&C_\varepsilon\left(
       \|H\|_{C^{\sigma,\sigma/2}(\Omega''\times I)}
       +\|V\|_{C^0(\Omega''\times I)}\right)\\
 &+\varepsilon
   \|V\|_{C^{2+\sigma,1+\sigma/2}(\Omega''\times I)}.
\end{split}
\end{equation}
Here the terms involving derivatives of the cutoff, as well as the
first- and zeroth-order couplings, are estimated by interpolation;
\(C_\varepsilon\) also depends on the fixed separation of \(\Omega'\) from
\(\partial\Omega''\).

For clarity, the last term in
\eqref{eq:localized-cauchy-dirichlet-step} can be removed without assuming an
outer a priori estimate.  Choose a smoothly nested family \(\Omega(r)\),
\(r_0\leq r<r_*\), with
\(\Omega_1\Subset\Omega(r_0)\) and
\(\overline{\Omega(r_*)}\Subset\Omega_2\), and set
\[
 \Phi(r)=\|V\|_{C^{2+\sigma,1+\sigma/2}(\Omega(r)\times I)},
 \qquad
 A=\|H\|_{C^{\sigma,\sigma/2}(Q_2^+)}+\|V\|_{C^0(Q_2^+)}.
\]
Fix \(\varepsilon>0\).  Keeping track of cutoff derivatives in
\eqref{eq:localized-cauchy-dirichlet-step} gives, for \(r<R<r_*\),
\[
        \Phi(r)\leq C_\varepsilon(R-r)^{-p}A
                    +\varepsilon\Phi(R)
\]
for a fixed \(p\).  With
\(r_j=r_*-(r_*-r_0)\tau^j\), choose
\(\varepsilon\tau^{-p}<1/2\) and iterate.  Then
\[
 \Phi(r_0)
 \leq CA\sum_{j=0}^{J-1}(\varepsilon\tau^{-p})^j
       +\varepsilon^J\Phi(r_J).
\]
The last term tends to zero as \(J\to\infty\), since \(V\) is smooth on the
fixed compact cylinder \(\overline{\Omega(r_*)}\times I\).  Thus
\eqref{eq:localized-cauchy-dirichlet-step} yields the second-order estimate
with no outer top-order term.

We now perform the finite-order induction.  Choose smooth domains
\[
 \Omega_1\Subset D_{m-2}\Subset D_{m-3}\Subset\cdots
 \Subset D_0\Subset D_{-1}\Subset\Omega_2,
\]
with the evident interpretation when \(m=2\).  For
\(0\leq\ell\leq m-2\), let \(V_\ell\) be the finite vector consisting of all
\(\partial^\alpha Z\) with \(|\alpha|=\ell\).  Differentiating the equation
spatially and collecting equal-order derivatives gives
\begin{equation}\label{eq:differentiated-cauchy-dirichlet-system}
 L_0V_\ell
 =\mathsf B_\ell^i\partial_iV_\ell+\mathsf C_\ell V_\ell+H_\ell.
\end{equation}
The first-order coefficient \(\mathsf B_\ell\) is assembled from
\(\mathsf B\) and \(\partial a\).  The zeroth-order coefficient
\(\mathsf C_\ell\) is assembled from \(\mathsf C\), \(\partial\mathsf B\),
and, when \(\ell\geq2\), \(\partial^2a\).  All remaining commutators are in
\(H_\ell\).  Schematically,
\begin{equation}\label{eq:cauchy-dirichlet-commutators}
\begin{split}
 H_\ell={}&\partial^\ell G
 +\sum_{3\leq |\beta|\leq\ell}
       \partial^\beta a*\partial^{\ell-|\beta|+2}Z\\
 &+\sum_{2\leq |\beta|\leq\ell}
       \partial^\beta\mathsf B*\partial^{\ell-|\beta|+1}Z
 +\sum_{1\leq |\beta|\leq\ell}
       \partial^\beta\mathsf C*\partial^{\ell-|\beta|}Z,
\end{split}
\end{equation}
where each sum includes the usual multiindex splittings and is omitted when
its lower limit exceeds \(\ell\).  In particular, every derivative of \(Z\)
in \(H_\ell\) has order at most \(\ell-1\).  The assumed
\(C^{m-2+\sigma,(m-2+\sigma)/2}\)-bounds therefore give uniform
\(C^{\sigma,\sigma/2}\)-bounds for \(\mathsf B_\ell\) and
\(\mathsf C_\ell\).

For \(\ell=0\), the localized second-order estimate on
\(D_0\Subset D_{-1}\) gives
\[
 \|V_0\|_{C^{2+\sigma,1+\sigma/2}(D_0\times I)}
 \leq C\left(
   \|G\|_{C^{\sigma,\sigma/2}(Q_2^+)}+\|Z\|_{C^0(Q_2^+)}\right).
\]
Suppose inductively that, for every \(0\leq j\leq\ell-1\), we have proved
\[
 \|V_j\|_{C^{2+\sigma,1+\sigma/2}(D_j\times I)}
 \leq C\left(
   \|G\|_{C^{j+\sigma,(j+\sigma)/2}(Q_2^+)}
   +\|Z\|_{C^0(Q_2^+)}\right).
\]
Formula
\eqref{eq:cauchy-dirichlet-commutators}, the parabolic product estimates, and
the induction hypothesis then give
\[
 \|H_\ell\|_{C^{\sigma,\sigma/2}(D_{\ell-1}\times I)}
 +\|V_\ell\|_{C^0(D_{\ell-1}\times I)}
 \leq C\left(
   \|G\|_{C^{\ell+\sigma,(\ell+\sigma)/2}(Q_2^+)}
   +\|Z\|_{C^0(Q_2^+)}\right).
\]
Applying the localized second-order estimate to
\eqref{eq:differentiated-cauchy-dirichlet-system} on
\(D_\ell\Subset D_{\ell-1}\) advances the induction.  At every step the
localized unknown has zero initial trace and vanishes on the artificial
spatial boundary.  The full induction uses only the time-jet identities
through \(r_m=\lfloor m/2\rfloor\); after \(\ell\) spatial differentiations,
only those with \(\ell+2j\leq m\) occur.  Since the cutoff is zero in a spatial
collar, the localized unknown, forcing, and every jet in this finite recursion
vanish at the artificial corner.  Hence the required compatibility identities
are automatic.  The original hypotheses that \(Z\) and \(G\) vanish near
\(\partial\Omega_2\) give the same conclusion if the outer boundary itself is
used.

At \(\ell=m-2\) we have controlled all spatial derivatives of \(Z\) through
order \(m\), with the required parabolic H\"older seminorms.  The remaining
mixed derivatives are recovered recursively from
\[
 \partial_tZ=a^{ij}\partial_i\partial_jZ
              +\mathsf B^i\partial_iZ+\mathsf CZ+G.
\]
Indeed, if \(|\nu|+2j\leq m\), apply
\(\partial_x^\nu\partial_t^{j-1}\) to this identity.  Every coefficient
derivative that occurs has parabolic weight at most \(m-2\), while every
derivative of \(Z\) on the right either has fewer time derivatives or has
already been controlled by the spatial induction.  Induction on \(j\),
together with the parabolic product estimates, gives the sup norms, the
\(\sigma\)-seminorms at weight \(m\), and the time seminorms at weight
\(m-1\) appearing in
\eqref{eq:parabolic-holder-norm-definition}.  Restriction from
\(D_{m-2}\) to \(\Omega_1\) proves the stated estimate.  This is the
finite-system version of the spatial differentiation procedure indicated in
\cite[Exercise~4.5(b), p.~84]{Lieberman}.
\end{proof}

\begin{proposition}[Uniform lifted linear Schauder estimate]
\label{prop:uniform-lifted-linear-schauder}
Fix \(m\geq2\) and \(\sigma\in(0,1)\).  Let \(s(t)\), \(a\leq t\leq b\),
be a Ricci flow satisfying \(\PBG_m(s;\Lambda;[a,b])\) for some \(\Lambda\geq1\).
Let \(E\) be a fixed finite-rank tensor bundle and consider,
on the universal cover and in the charts of
Lemma~\ref{lem:uniform-parabolic-atlas}, a system
\begin{equation}\label{eq:general-scalar-principal-system}
 \mathcal P U
 =\partial_tU-a^{ij}\nabla_i\nabla_jU-B^i\nabla_iU-CU=F.
\end{equation}
Assume that the second-order part is diagonal in the bundle components and
uses the same coefficient matrix \(a^{ij}\) on every component.  For fixed
constants \(0<\lambda_-\leq\lambda_+<\infty\), suppose that
\[
 \lambda_- s^{ij}\xi_i\xi_j\leq a^{ij}\xi_i\xi_j
 \leq\lambda_+ s^{ij}\xi_i\xi_j,
\]
and that the local
\(C^{m-2+\sigma,(m-2+\sigma)/2}\)-norms of \(a,B,C\) are bounded by
\(A_0\).

Every smooth bounded solution on the universal cover satisfies, for
\(\theta>0\) with \(a+\theta\leq b\),
\begin{equation}\label{eq:uniform-interior-linear-schauder}
\begin{aligned}
 \|U\|_{C^{m+\sigma,(m+\sigma)/2}_{\mathrm{uloc}}
          (\widetilde M\times[a+\theta,b],s)}
 \leq C_\theta\bigl(&
       \|U\|_{C^0(\widetilde M\times[a,b],s)}\\
       &+\|F\|_{C^{m-2+\sigma,(m-2+\sigma)/2}_{\mathrm{uloc}}}\bigr).
\end{aligned}
\end{equation}
For the global estimate, assume in addition that \(\mathcal P\) is
deck-invariant, equivalently that the coefficient fields \(a,B,C\) descend to
\(M\).  If \(U\) and \(F\) are deck-invariant and descend to tensors \(u\)
and \(f\) on \(M\), with \(u(a)\in C^{m,\sigma}\) and
\(f\in Y_s^m[a,b]\), then
\begin{equation}\label{eq:uniform-global-linear-schauder}
 \|u\|_{X_s^m[a,b]}
 \leq C\left(
      \|u(a)\|_{C^{m,\sigma}(M,s(a))}
      +\|f\|_{Y_s^m[a,b]}\right).
\end{equation}
The constants depend only on
\(m,\sigma,\Lambda,\lambda_-,\lambda_+,A_0\), on \(\theta\) in the interior
estimate, and on an upper bound for \(b-a\).  The interior and initial-face
estimates are based, respectively, on Lieberman's interior Schauder theorem
\cite[Theorem~4.9, p.~59]{Lieberman} and his Cauchy--Dirichlet Schauder theorem
\cite[Theorem~4.28, p.~77]{Lieberman}.  All atlas-radius, overlap, and cutoff
constants, fixed-coordinate ellipticity constants, localized initial-face
boundary-chart constants, and coefficient H\"older constants are uniform over
the admissible family specified by \(\PBG_m\) and the displayed bounds.  In
particular, the constants are independent of collapse of the compact quotient.
\end{proposition}

\begin{proof}
Fix one lifted parabolic chart, choose concentric cylinders
\(Q_0\Subset Q_1\Subset Q_2\) whose radii are fixed fractions of the atlas
radius, and choose a spatial cutoff \(\chi\) supported in \(Q_2\) with
\(\chi=1\) on \(Q_1\).  In the fixed tensor frame,
\eqref{eq:general-scalar-principal-system} reads
\[
 \partial_tU^A-a^{ij}\partial_i\partial_j U^A
 =\mathsf B^{iA}{}_D\partial_iU^D+\mathsf C^A{}_DU^D+F^A,
\]
where \(A,D\in\{1,\ldots,\operatorname{rank}E\}\), repeated \(D\)-indices
are summed, and the terms from the reference connection have been included in
\(\mathsf B\) and \(\mathsf C\).  The base-order interior estimate follows by
the same localization, interpolation, and radius-iteration argument used to
remove the outer top-order term in the proof of
Lemma~\ref{lem:finite-order-cauchy-dirichlet}: apply the scalar interior
estimate \cite[Theorem~4.9, p.~59]{Lieberman} componentwise to \(\chi U^A\),
estimate the first- and zeroth-order couplings and the cutoff terms with
\cite[Proposition~4.2, p.~49]{Lieberman}, and iterate over concentric
cylinders.  This gives
\[
 \|U\|_{C^{2+\sigma,1+\sigma/2}(Q_0)}
 \leq C\left(
   \|F\|_{C^{\sigma,\sigma/2}(Q_2)}+\|U\|_{C^0(Q_2)}\right).
\]

For the higher-order estimate, repeat the spatial-differentiation
induction in the proof of
Lemma~\ref{lem:finite-order-cauchy-dirichlet}, with the interior estimate above
in place of the localized Cauchy--Dirichlet estimate.  On a fixed chain of
nested cylinders, collect the components \(\partial^\alpha U^A\),
\(|\alpha|=\ell\), into one finite system.  Differentiation preserves the
scalar principal part; the terms with one derivative on \(a^{ij}\) become
first-order couplings, and the terms linear in \(\partial^\alpha U\) become
zeroth-order couplings.  The remaining commutators contain derivatives of
\(F\) of order at most \(\ell\) and derivatives of \(U\) of order at most
\(\ell+1\).  The coefficient bounds, the induction hypothesis, and the same
interpolation and radius-iteration argument therefore advance the induction
through \(\ell=m-2\).  The last paragraph of the proof of
Lemma~\ref{lem:finite-order-cauchy-dirichlet}, applied to the present equation,
then recovers the mixed time derivatives of parabolic weight at most \(m\).
This proves the full \(C^{m+\sigma,(m+\sigma)/2}\)-bound on \(Q_0\).
The argument up to this point uses only fixed-coordinate uniform parabolicity,
the stated coefficient bounds, and the cylinder separations, and therefore
proves the local estimate
\eqref{eq:local-nested-cylinder-schauder}.  The \(\PBG_m\) hypothesis is used
next only to make these local data uniform over the lifted atlas.

Taking the supremum over the atlas of
Lemma~\ref{lem:uniform-parabolic-atlas} proves
\eqref{eq:uniform-interior-linear-schauder}.  This part of the argument does
not require \(U\) to descend to the compact quotient.

Suppose now that \(\mathcal P\), \(U\), and \(F\) are deck-invariant.  The
equation for
\(|u|_s^2\), the coefficient bounds, the scalar maximum principle on the
compact quotient, and Gr\"onwall's inequality give
\[
 \|u\|_{C^0(M\times[a,b],s)}
 \leq C\left(\|u(a)\|_{C^0(M,s(a))}+\|f\|_{C^0}\right).
\]
For cylinders meeting the initial time slice \(t=a\), write
\(Q_i^+=Q_i\cap\{t\geq a\}\), choose the spatial sections \(\Omega_i\) to be
concentric Euclidean balls, choose the same type of spatial cutoff \(\chi\),
and put \(L_0=\partial_t-a^{ij}\partial_i\partial_j\) in the fixed frame.
The artificial spatial boundary is the one defined in
Remark~\ref{rem:lieberman-scalar-versus-system}.  If
\(u_0^A(x)=U^A(x,a)\), define the time-independent extension
\[
        E^A(x,t)=\chi(x)u_0^A(x),
        \qquad Z^A=\chi U^A-E^A.
\]
Then \(Z=0\) on the initial face \(t=a\) and on the artificial spatial
boundary, and its equation is
\begin{equation}\label{eq:initial-face-localized-system}
\begin{split}
 L_0Z^A={}&\chi F^A
 +\chi\mathsf B^{iA}{}_D\partial_iU^D
 +\chi\mathsf C^A{}_DU^D
 +[L_0,\chi]U^A
 +a^{ij}\partial_i\partial_jE^A.
\end{split}
\end{equation}
Using \(\chi U=Z+E\), rewrite the two coupling terms as
\[
 \chi\mathsf B^i\partial_iU
 =\mathsf B^i\partial_iZ+\mathsf B^i\partial_iE
   -\mathsf B^i(\partial_i\chi)U,
 \qquad
 \chi\mathsf CU=\mathsf CZ+\mathsf CE.
\]
Thus the terms involving \(Z\) have the first- and zeroth-order form in
Lemma~\ref{lem:finite-order-cauchy-dirichlet}.  Denote by \(\mathcal G\) the
sum of the remaining terms, which form its forcing.
The remainder of the initial-face argument is the finite-order induction
from the proof of Lemma~\ref{lem:finite-order-cauchy-dirichlet}, with two
points requiring comment.  First, every term involving \(E\)
is bounded by \(\|u(a)\|_{C^{m,\sigma}}\).  Second, derivatives of the
cutoff terms in \(\mathcal G\) contain derivatives of \(U\) of order at most
\(\ell+1\) at the \(\ell\)-th spatial-differentiation step.  This is one
order below the
\(C^{\ell+2+\sigma,(\ell+2+\sigma)/2}\)-norm being estimated.
Lieberman's interpolation inequality
\cite[Proposition~4.2, p.~49]{Lieberman} therefore bounds these terms by an
arbitrarily small multiple of the corresponding norm on the next larger
cylinder, plus a constant times \(\|U\|_{C^0}\); the nested-radius iteration
used in the proof of Lemma~\ref{lem:finite-order-cauchy-dirichlet} absorbs the
former term.

The compatibility check is also the same as in that proof.  With
\(r_m=\lfloor m/2\rfloor\), the required identities are
\((\partial_t^jZ)(\cdot,a)|_{\partial\Omega_2}=0\) for
\(0\leq j\leq r_m\), where the higher time jets are generated recursively from
\eqref{eq:initial-face-localized-system}; after \(\ell\) spatial
differentiations, only those with \(\ell+2j\leq m\) are used.  Since \(\chi\),
and hence \(E\), \(Z\), and \(\mathcal G\), vanish in a full collar of the
artificial boundary, every jet in this recursion vanishes there.  Thus all
required corner identities are automatic.  Repeating the induction in the
proof of Lemma~\ref{lem:finite-order-cauchy-dirichlet}, using the componentwise
Cauchy--Dirichlet estimate \cite[Theorem~4.28, p.~77]{Lieberman} for its
localized second-order step and then recovering the mixed time derivatives
from the equation as in the lemma's final paragraph gives
\[
\begin{aligned}
 \|U\|_{C^{m+\sigma,(m+\sigma)/2}(Q_1^+)}
 \leq C\bigl(&\|F\|_{C^{m-2+\sigma,(m-2+\sigma)/2}(Q_2^+)}
       +\|u(a)\|_{C^{m,\sigma}}\\
       &+\|U\|_{C^0(Q_2^+)}\bigr).
\end{aligned}
\]
Taking the supremum of these initial-face and interior estimates and using
the preceding maximum-principle bound for the \(C^0\)-term proves
\eqref{eq:uniform-global-linear-schauder}.  No sum over
quotient charts occurs, which is why the constants are uniform under
collapse.
\end{proof}

\begin{remark}
\label{rem:local-nested-cylinder-schauder}
The proof also gives the following purely local estimate.  Let
\(Q_0\Subset Q_2\) be fixed nested coordinate cylinders in one chart, with a
fixed positive spatial separation and lower-time gap; their upper time faces
may agree.  Suppose that, in a fixed tensor frame on \(Q_2\), the system is
\[
 \partial_tU^A-a^{ij}\partial_i\partial_jU^A
 =\mathsf B^{iA}{}_D\partial_iU^D+\mathsf C^A{}_DU^D+F^A,
\]
where \(a^{ij}\) is common to all components and is uniformly parabolic with
fixed-coordinate constants \(\lambda_-\) and \(\lambda_+\), and the
\(C^{m-2+\sigma,(m-2+\sigma)/2}\)-norms of
\(a,\mathsf B,\mathsf C\) are at most \(A_1\).  Then
\begin{equation}\label{eq:local-nested-cylinder-schauder}
 \|U\|_{C^{m+\sigma,(m+\sigma)/2}(Q_0)}
 \leq C\left(
      \|U\|_{C^0(Q_2)}
      +\|F\|_{C^{m-2+\sigma,(m-2+\sigma)/2}(Q_2)}\right).
\end{equation}
Here \(C\) depends only on
\(m,\sigma,\lambda_-,\lambda_+,A_1\), the rank and dimension, and the fixed
cylinder-separation data.  This local estimate requires neither
\(\PBG_m\), a covering space, nor deck invariance.
\end{remark}

\subsection{Quasilinear derivative gain}

\begin{lemma}[Local one-derivative gain for Ricci--DeTurck systems]
\label{lem:local-rdt-one-derivative-gain}
Fix an integer \(k\geq2\) and \(\sigma\in(0,1)\).  Let \(Q\) and \(Q'\) be
concentric coordinate cylinders in a fixed spatial chart, with \(Q'\) having
smaller spatial radius and a positive time gap from the lower time face of
\(Q\); the two cylinders may share their upper time face.  Let \(r(t)\) be a
smooth reference metric on \(Q\), and let \(q(t)\) be a smooth positive
definite solution of the Ricci--DeTurck equation relative to \(r(t)\).  In a
fixed tensor frame this equation has the scalar-principal form
\begin{equation}\label{eq:abstract-local-quasilinear-rdt}
 \partial_tq^A-q^{ij}\partial_i\partial_jq^A
 =\mathcal F^A\bigl(x,t,q,q^{-1},\partial q;
                 r,r^{-1},\partial r,\partial^2r\bigr),
\end{equation}
where \(\mathcal F\) is a universal smooth expression and contains no second
spatial derivatives of \(q\).  Suppose, in the fixed coordinates, that
\[
\begin{gathered}
 \lambda I\leq(q_{ij})\leq\Lambda I,
 \qquad \|r^{-1}\|_{C^0(Q)}\leq A,\\
 \|q\|_{C^{k+\sigma,(k+\sigma)/2}(Q)}
 +\|r\|_{C^{k+1+\sigma,(k+1+\sigma)/2}(Q)}\leq A.
\end{gathered}
\]
Then
\begin{equation}\label{eq:local-rdt-one-derivative-gain}
 \|q\|_{C^{k+1+\sigma,(k+1+\sigma)/2}(Q')}
 \leq C,
\end{equation}
where \(C\) depends only on the displayed bounds, the rank and dimension, and
the fixed spatial and lower-time separations between \(Q'\) and \(Q\).
Consequently, the estimate is uniform in any bounded-geometry atlas with fixed
nested-cylinder ratios.
\end{lemma}

\begin{proof}
The bound on \(r^{-1}\), together with the assumed H\"older bound on \(r\),
controls \(r^{-1}\) in
\(C^{k+1+\sigma,(k+1+\sigma)/2}(Q)\).  Indeed, differentiate
\(rr^{-1}=I\) and apply the parabolic H\"older product estimates.  In
particular, the reference Christoffel symbols and their derivatives through
parabolic order \(k\) are uniformly bounded.  The ellipticity bound on \(q\)
similarly controls \(q^{-1}\) to the order supplied by the hypotheses.

Let \(|\alpha|=k-1\), and collect the tensors
\(V_\alpha^A=\partial^\alpha q^A\) into one finite system.  Spatially
differentiating \eqref{eq:abstract-local-quasilinear-rdt} and moving the term
in which all derivatives remain on \(q^A\) to the left gives
\begin{equation}\label{eq:differentiated-local-quasilinear-rdt}
 (\partial_t-q^{ij}\partial_i\partial_j)V_\alpha^A
 =\mathsf B_{\alpha\beta}^{iA}{}_D\partial_iV_\beta^D
  +\mathsf C_{\alpha\beta}^{A}{}_DV_\beta^D
  +G_\alpha^A.
\end{equation}
Here \(\beta\) ranges over the multiindices of length \(k-1\).  The
first-order terms include those in which one derivative falls on \(q^{ij}\)
and the highest derivatives arising from the dependence of \(\mathcal F\) on
\(\partial q\).  The terms linear in an undifferentiated \(V_\beta\) are
placed in \(\mathsf C\).  After all terms linear in \(V_\beta\) or \(\partial V_\beta\) have been
placed in \(\mathsf C\) or \(\mathsf B\), every factor in \(G_\alpha\)
contains a derivative of \(q\) of order at most \(k-1\), or a derivative of
\(r\) of order at most \(k+1\).  The hypotheses, the preceding control
of both inverse metrics, and the parabolic H\"older product estimates
therefore give
\[
 \|\mathsf B\|_{C^{\sigma,\sigma/2}(Q_1)}
 +\|\mathsf C\|_{C^{\sigma,\sigma/2}(Q_1)}
 +\|G\|_{C^{\sigma,\sigma/2}(Q_1)}\leq C
\]
on any fixed intermediate cylinder \(Q'\subset Q_1\subset Q\).

Equation \eqref{eq:differentiated-local-quasilinear-rdt} is a finite system
with the same scalar principal operator on every component.  Applying the
purely local estimate \eqref{eq:local-nested-cylinder-schauder} from
Remark~\ref{rem:local-nested-cylinder-schauder} gives
\(V\in C^{2+\sigma,1+\sigma/2}(Q')\), with a bound depending only on the
stated data; its \(C^0\)-term is already controlled by the assumed
\(C^{k+\sigma,(k+\sigma)/2}\)-norm of \(q\).  Thus all spatial derivatives of
\(q\) through order \(k+1\) are bounded with the required H\"older control.
The equation and its differentiated forms then recover the mixed time
derivatives of parabolic weight at most \(k+1\).  This proves
\eqref{eq:local-rdt-one-derivative-gain}.
\end{proof}

\subsection{Continuation, smoothing, and regular restart}
\label{subsec:appendix-continuation-restart}

\begin{lemma}[Fixed-quotient Ricci--DeTurck continuation]
\label{lem:fixed-quotient-rdt-continuation}
Let \(M\) be compact, let \(g_0\) be a fixed smooth metric, and let \(s(t)\)
be a fixed smooth background family on \([a,b]\).  Given
\(c>0\), \(C<\infty\), and \(\sigma\in(0,1)\), there is
\(\tau_M>0\) with the following property.  For every \(t_0\in[a,b]\) and
every smooth metric \(v_0\) satisfying
\[
        c g_0\leq v_0\leq Cg_0,
        \qquad \|v_0\|_{C^{2,\sigma}(M,g_0)}\leq C,
\]
the Ricci--DeTurck equation relative to \(s(t)\), with initial value
\(v(t_0)=v_0\), has a unique solution on
\([t_0,t_0+\tau_M]\cap[a,b]\).  The number \(\tau_M\) may depend on the
fixed quotient, atlas, background family, and displayed bounds; no
collapse-uniformity is asserted.
\end{lemma}

\begin{proof}
In a fixed finite atlas, the equation has scalar strongly parabolic principal
part
\[
        \partial_t v_{ij}=v^{pq}\partial_p\partial_qv_{ij}
        +F_{ij}(x,t,v,\partial v;s,\partial s,\partial^2s).
\]
The two-sided metric bound gives a common parabolicity constant, while the
\(C^{2,\sigma}\)-bound and the fixed smooth background place the coefficients
and nonlinearities in a bounded set with uniform H\"older and local Lipschitz
bounds.  The contraction argument in DeTurck's local existence proof
\cite[Theorem and proof, pp.~158--161]{DeTurck} therefore gives a lifespan
bounded below by one number \(\tau_M\) for every such restart slice.  The same
argument gives uniqueness on overlaps.
\end{proof}

\begin{proposition}[Uniform lifted Ricci--DeTurck continuous dependence]
\label{prop:uniform-lifted-rdt-stability}
Fix \(N\geq2\) and \(\sigma\in(0,1)\).  Let \(s(t)\), \(a\leq t\leq b\),
be a Ricci flow satisfying \(\PBG_N(s;\Lambda;[a,b])\) for some \(\Lambda\geq1\).
There are constants \(\eta_{\mathrm{cd}}>0\) and
\(C_{\mathrm{cd}}<\infty\), depending only on these data and an upper bound for
\(b-a\), with the following property.

\medskip
\noindent\textbf{Hypotheses.}
If \(v(a)\) is a smooth Riemannian metric and
\[
 \|v(a)-s(a)\|_{C^{N,\sigma}(M,s(a))}<\eta_{\mathrm{cd}},
\]

\medskip
\noindent\textbf{Conclusion.}
Then the Ricci--DeTurck flow \(v(t)\) relative to \(s(t)\) exists on
\([a,b]\) and
\begin{equation}
 \|v-s\|_{X_s^N[a,b]}
 \leq C_{\mathrm{cd}}\|v(a)-s(a)\|_{C^{N,\sigma}(M,s(a))}.
\end{equation}
The constants are uniform under collapse of the quotient.
\end{proposition}

\begin{proof}
\proofstep{Step 1. Derive the quadratic perturbation estimate.}

Put \(u=v-s\) and write
\(\mathcal L_s=\partial_t-\Delta_{L,s}\), where the Lichnerowicz Laplacian on
symmetric two-tensors is
\[
 (\Delta_{L,s}u)_{ij}=s^{pq}\nabla_p\nabla_q u_{ij}
 +2R_{ipqj}(s)u^{pq}-R_i{}^p(s)u_{pj}-R_j{}^p(s)u_{ip}.
\]
All contractions and covariant derivatives in this formula use \(s\).  For the
spatial
Ricci--DeTurck operator
\[
        \mathcal Q_s(w)=-2\Ric(w)+\Lie_{W(w,s)}w,
\]
Taylor's formula at \(s\) gives
\begin{equation}
 \mathcal L_su=\mathcal N_s(u),
 \qquad
 \mathcal N_s(u)=\int_0^1(1-\vartheta)
 D^2\mathcal Q_s|_{s+\vartheta u}[u,u] \,d\vartheta.
\end{equation}
Every term in \(\mathcal N_s(u)\) is at least quadratic and has schematic form
\[
        u*\nabla^2u+\nabla u*\nabla u+\Rm(s)*u*u.
\]
The parabolic H\"older spaces in \eqref{eq:XY-spaces} are Banach algebras.
Consequently, on every interval \([a,t]\) contained in the existence interval
and on which \(\|u\|_{X_s^N[a,t]}\) is sufficiently small,
\begin{equation}\label{eq:stage-remainder-bound}
        \|\mathcal N_s(u)\|_{Y_s^N[a,t]}
        \leq C\|u\|_{X_s^N[a,t]}^2.
\end{equation}
Proposition~\ref{prop:uniform-lifted-linear-schauder}, applied to
\(\mathcal L_s\) on the same interval, yields
\begin{equation}
 \|u\|_{X_s^N[a,t]}
 \leq C_{\mathrm{lin}}\|u(a)\|_{C^{N,\sigma}}
      +C\|u\|_{X_s^N[a,t]}^2.
\end{equation}
We also record the initial time derivatives needed to start the bootstrap.
Put \(\varepsilon_0=\|u(a)\|_{C^{N,\sigma}(M,s(a))}\).  For
\(\varepsilon_0\) below a fixed uniform threshold, the equation recursively
gives
\[
 \sum_{j=0}^{\lfloor N/2\rfloor}
 \left\|\left. (\mathcal D_t^s)^ju\right|_{t=a}
       \right\|_{C^{N-2j,\sigma}(M,s(a))}
 \leq C_{\mathrm{init}}\varepsilon_0.
\]
To see this, evaluate \(\mathcal L_su=\mathcal N_s(u)\) at \(t=a\), then
differentiate it in time and substitute the already determined lower time
derivatives.  At parabolic weight at most \(N\), this recursion uses spatial
derivatives of \(u(a)\) through order \(N\) only.  The background
coefficients are controlled by \(\PBG_N\), and every resulting term vanishes
when \(u(a)=0\).  The H\"older product estimates therefore give the displayed
linear bound on a sufficiently small fixed neighborhood.  Passing from
ordinary to metric-compatible time derivatives uses only the same controlled
background coefficients.  Enlarge \(C_{\mathrm{init}}\) to include the
fixed-chart norm-equivalence constants, and increase
\(C_{\mathrm{lin}}\) so that \(C_{\mathrm{lin}}\geq\max\{C_{\mathrm{init}},1\}\).

Choose \(\delta>0\) below this initial-data threshold and so that every tensor
\(u\) with \(\|u\|_{X_s^N[a,t]}\leq\delta\) makes \(s+u\) positive definite and keeps the
Ricci--DeTurck system uniformly parabolic, and also so that
\(C\delta\leq1/4\).  This \(\delta\) is the \emph{bootstrap radius}.
On the fixed compact quotient,
Lemma~\ref{lem:fixed-quotient-rdt-continuation} gives a unique short-time
solution.  The prescribed smooth time-dependent background enters only as a
known coefficient in that fixed-atlas quasilinear system.

Let \([a,t_*]\) be any interval on which
\(\|u\|_{X_s^N[a,t_*]}\leq\delta\).  If
\(\|u(a)\|_{C^{N,\sigma}}\leq\delta/(4C_{\mathrm{lin}})\), then the preceding
inequality and \(\|u\|^2\leq\delta\|u\|\) give
\[
 \|u\|_{X_s^N[a,t_*]}
 \leq \frac{\delta}{4}+\frac14\|u\|_{X_s^N[a,t_*]},
 \qquad\text{hence}\qquad
 \|u\|_{X_s^N[a,t_*]}\leq\frac{\delta}{3}.
\]
We call this strict improvement from \(\delta\) to \(\delta/3\) the
\emph{bootstrap estimate}.

\proofstep{Step 2. Close the bootstrap and continue to the final time.}

We now make the continuation argument quantitative.  Let
\([a,T_{\max})\), with \(T_{\max}\leq b\), be the maximal existence interval.
For the smooth local solution, the purely temporal H\"older seminorms in
\(X_s^N[a,t]\) tend to zero as \(t\downarrow a\): their time exponents are
\(\sigma/2\) and \((1+\sigma)/2\), both strictly less than one.  The sup norms
and spatial H\"older parts converge to those of the initial time derivatives.
Consequently, the initial-jet estimate in Step~1 gives
\[
 \limsup_{t\downarrow a}\|u\|_{X_s^N[a,t]}
 \leq C_{\mathrm{init}}\varepsilon_0\leq\delta/4.
\]
Thus the bootstrap condition holds on a short initial interval.  The length
of this interval may depend on higher derivatives of the fixed smooth
initial metric; no collapse-uniform lower bound for it is needed.  Let
\(T_*\) be the supremum of the times up to which the norm is at most
\(\delta\).  The estimate just proved gives the strict bound \(\delta/3\) on
\([a,T_*]\); continuity of the parabolic H\"older norm on smooth compact
subintervals then extends the bootstrap condition past \(T_*\).  Hence
\(T_*=T_{\max}\).
Consequently,
\begin{equation}\label{eq:continuous-dependence-bootstrap-all-times}
        \|u\|_{X_s^N[a,t]}\leq\delta/3,
        \qquad a<t<T_{\max}.
\end{equation}
In particular, every slice \(v(t)=s(t)+u(t)\) is uniformly positive definite,
and its lifted \(C^{N,\sigma}\)-norm is bounded by a constant depending only
on the data in the proposition.

The remaining continuation argument is made on the fixed compact quotient;
no collapse-uniform restart time is needed.  Choose a finite smooth coordinate
atlas on this quotient and a fixed smooth background metric \(g_0\).  The
intrinsic lifted estimate
\eqref{eq:continuous-dependence-bootstrap-all-times}, restricted to finitely
many lifted coordinate neighborhoods and compared with this fixed atlas,
gives constants \(c_M>0\) and \(C_M<\infty\) such that
\begin{equation}\label{eq:fixed-quotient-continuation-bound}
        c_M g_0\leq v(t)\leq C_M g_0,
        \qquad
        \|v(t)\|_{C^{N,\sigma}(M,g_0)}\leq C_M,
        \qquad a<t<T_{\max}.
\end{equation}
Here the constants may depend on the particular compact quotient and its
chosen atlas; this is harmless, since all collapse-independent estimates have
already been obtained in
\eqref{eq:continuous-dependence-bootstrap-all-times}.  The coefficients of the
known background \(s(t)\) have bounded fixed-atlas norms on \([a,b]\).
Therefore, the Ricci--DeTurck equation is a uniformly strongly parabolic
quasilinear system on this fixed compact manifold, with a uniform
\(C^{2,\sigma}\)-bound on its solution slices.

Lemma~\ref{lem:fixed-quotient-rdt-continuation}, applied with the bounds in
\eqref{eq:fixed-quotient-continuation-bound}, gives a number \(\tau_M>0\),
independent of the restart time \(t_0<T_{\max}\), for which the solution with
initial value \(v(t_0)\) exists on
\([t_0,t_0+\tau_M]\cap[a,b]\).  Only the \(C^{2,\sigma}\)-bound and uniform
parabolicity enter this fixed-quotient lifespan; the higher finite-order bounds
are then propagated by differentiating the system and applying the local
parabolic estimates already established above.

If \(T_{\max}<b\), choose \(t_0<T_{\max}\) so close to \(T_{\max}\) that
\(t_0+\tau_M>T_{\max}\).  The restarted solution extends past
\(T_{\max}\), and uniqueness for the strongly parabolic Ricci--DeTurck system
identifies it with the original solution on their overlap, contradicting
maximality.  Thus \(T_{\max}=b\).  Choosing
\(t_0\in(b-\tau_M,b)\) and restarting on \([t_0,b]\) gives the endpoint
\(t=b\).  Again uniqueness identifies the two solutions on their overlap, so
there is no gluing loss at the limiting time.

\proofstep{Step 3. Obtain the linear global stability estimate.}

Finally, on \([a,b]\), combine
\eqref{eq:continuous-dependence-bootstrap-all-times} with
\eqref{eq:stage-remainder-bound} and the linear estimate.  Since
\(C\delta\leq1/4\), the quadratic term can be absorbed, giving
\[
 \|u\|_{X_s^N[a,b]}
 \leq 2C_{\mathrm{lin}}\|u(a)\|_{C^{N,\sigma}}.
\]
Taking \(\eta_{\mathrm{cd}}=\delta/(4C_{\mathrm{lin}})\) and
\(C_{\mathrm{cd}}=2C_{\mathrm{lin}}\) proves the proposition.  The usual DeTurck ODE
then identifies this solution with a pullback of the corresponding Ricci flow.
\end{proof}

The preceding results give the continuous-dependence estimate needed in the body.  We finish by proving the positive-time smoothing and restart
statement used at each switching time.

\begin{proof}[Proof of Proposition~\ref{prop:positive-time-regular-restart}]
\proofstep{Step 1. Create high-order parabolic geometry for the background.}

We work on the universal cover.  From
\eqref{eq:regular-restart-curvature}, local \(\TT^2\)-invariance, and
Lemma~\ref{lem:universal-cover-inj}, the metrics \(\widetilde r(\tau)\) have
a uniform injectivity-radius lower bound on \([10,100]\).  Shi's estimates
give bounds for all curvature derivatives through the finite order needed
below on \([10,100]\).  Hence, after increasing a constant \(\Lambda_{\mathrm{sm}}\) depending only on the
fixed data,
\[
        \PBG_{N+5}(r;\Lambda_{\mathrm{sm}};[10,100])
\]
holds.  The norm equivalence in
Lemma~\ref{lem:intrinsic-fixed-chart-equivalence} therefore converts
\eqref{eq:regular-restart-low-order-closeness} into the uniform order-\(N\)
fixed-chart bound on \([10,100]\) used in the bootstrap below.

\proofstep{Step 2. Gain five derivatives for the Ricci--DeTurck metric.}

Choose \(\beta_{\mathrm{reg}}\) initially so small that
\eqref{eq:regular-restart-low-order-closeness} makes \(q\) uniformly
bilipschitz to \(r\) and keeps the Ricci--DeTurck equation uniformly
parabolic.  In the universal-cover charts for \(r\), the equation for
\(q\) has the quasilinear scalar-principal form
\eqref{eq:abstract-local-quasilinear-rdt}.  We now give the finite bootstrap
that will
be used below.

Choose once and for all six nested spatial cylinders in each fixed chart,
with radii decreasing by a uniform amount.  For \(j=0,\ldots,5\), put
\(I_j=[10+2j,100]\).  We claim inductively that
\begin{equation}
 \|q\|_{C^{N+j+\sigma,(N+j+\sigma)/2}_{\mathrm{uloc}}
          (\widetilde M\times I_j,r)}\leq C_j.
\end{equation}
For \(j=0\), this follows from
\eqref{eq:regular-restart-low-order-closeness} and the
\(\PBG_{N+5}\)-bound for \(r\).  Suppose it holds for some \(j<5\).  Apply
Lemma~\ref{lem:local-rdt-one-derivative-gain} with \(k=N+j\), using the
\(j\)-th spatial cylinder as the outer cylinder and the next one as the inner
cylinder.  The two units of temporal room between \(I_j\) and \(I_{j+1}\)
provide the required separation from the lower time face; backward cylinders
allow the estimate to include \(\tau=100\).

The hypotheses of the lemma are the induction hypothesis, uniform
ellipticity of \(q\), and the \(\PBG_{N+5}\)-bound for \(r\).  In particular,
Lemma~\ref{lem:uniform-parabolic-atlas} gives a uniform coordinate
\(C^0\)-bound for \(r^{-1}\) on every outer cylinder, so the reference
nondegeneracy assumption is satisfied at all five steps.  More explicitly,
differentiating the equation \(N+j-1\) times uses derivatives of \(q\) of
order at most \(N+j\) and derivatives of \(r\) of order at most
\(N+j+1\leq N+5\).  The lemma therefore gives the asserted
\(C^{N+j+1+\sigma,(N+j+1+\sigma)/2}\)-bound on \(I_{j+1}\).  This proves the
induction.  After the five steps,
\[
 \|q\|_{C^{N+5+\sigma,(N+5+\sigma)/2}_{\mathrm{uloc}}
          (\widetilde M\times[20,100],r)}\leq C_5.
\]
The argument is carried out entirely in the uniform universal-cover atlas, so
its constants do not involve the injectivity radius of the collapsed quotient.
For related local derivative estimates, see
\cite[Appendix~A.2, ``Local derivative estimates'']{Bamler-Kleiner}.
Lemma~\ref{lem:local-rdt-one-derivative-gain} provides the precise finite-order
estimate used here.

\proofstep{Step 3. Smooth the difference by the homogeneous linear system.}

Set \(u=q-r\).  Since \(q\) and \(r\) solve the Ricci--DeTurck equation
with the same background \(r\),
Lemma~\ref{lem:rdt-difference-system}, applied on the universal cover, gives a
homogeneous system
\[
        \partial_\tau u-P^{ij}\nabla_i\nabla_j u
        =B^i\nabla_i u+Cu
\]
with scalar principal part.  The preceding regularity of \(q\) and \(r\)
gives the coefficient bounds required by
Proposition~\ref{prop:uniform-lifted-linear-schauder} with \(m=N+5\).
Its interior estimate, applied with room at the lower time face, yields
\begin{equation}\label{eq:regular-restart-interior-smoothing}
 \|q-r\|_{X_r^{N+5}[30,100]}
 \leq C\|q-r\|_{C^0(M\times[1,100],r)}.
\end{equation}

\proofstep{Step 4. Rescale, average, and verify the restart geometry.}

The fixed factor-\(100\) rescaling changes the spatial H\"older norm by a factor depending
only on \(N\) and \(\sigma\).  Evaluating
\eqref{eq:regular-restart-interior-smoothing} at \(\tau=100\) gives the
first term on the left-hand side of
\eqref{eq:regular-restart-high-order-estimate}.  Since \(\rho^+\) is
invariant, Lemma~\ref{lem:averaging-invariant-holder} and
\[
        r^+(1)-\rho^+
        =\overline{q^+(1)-\rho^+}
\]
give the second term.

It remains to verify the uniform geometry of the restart.  Shi's estimates at
\(\tau=100\), followed by the fixed rescaling, bound
\(|\widetilde\nabla^j\Rm_{\widetilde\rho^+}|\) for
\(0\leq j\leq N+3\).  Lemma~\ref{lem:universal-cover-inj}, applied to the
invariant metric \(r(100)\) and then rescaled, gives a uniform lower bound for
\(\inj_{\widetilde\rho^+}\).  Thus \(\BG_N(\rho^+;\Lambda')\) holds for a uniform \(\Lambda'\).  By decreasing
\(\beta_{\mathrm{reg}}\) once more,
\eqref{eq:regular-restart-high-order-estimate} makes \(r^+(1)\) a sufficiently
small \(C^{N+5,\sigma}\)-perturbation of \(\rho^+\).  The formulas for the
curvature tensor and its covariant derivatives then give uniform bounds through
order \(N+3\), while positivity follows from the \(C^0\)-part.  Finally,
\(r^+(1)\) is invariant, so Lemma~\ref{lem:universal-cover-inj} gives its
uniform lifted injectivity-radius lower bound.  Enlarging one constant gives
\(\BG_N(\rho^+;\Lambda_{\mathrm{reg}})\) and
\(\BG_N(r^+(1);\Lambda_{\mathrm{reg}})\), proving the proposition.
\end{proof}

This completes the uniform analytic estimates used in the body of the paper.

\clearpage
\section{Source guide for the external inputs}
\label{app:source-guide}
\addtocontents{toc}{\protect\setcounter{tocdepth}{1}}
Section~\ref{sec:external-inputs} states the external results in the forms used
by the block argument.  This appendix records precisely what is quoted, what
is a harmless reformulation, and what uniformization is supplied in the
present paper.

\subsection*{The Type-III blowdown theorem}
Proposition~\ref{input:lott-blowdown} is the portion of
\cite[Theorem~1.2]{Lott10} used here, restricted to the three geometries under
consideration and written with the two hypotheses displayed in
\eqref{eq:lott-input-hypotheses}.  It asserts only compact blowdown information
and the lifted limits.  In the Euclidean and Nil cases, the compact conclusion
is full convergence to a point.  In the Sol case, the quoted compact conclusion
only classifies any Gromov--Hausdorff subsequential limit as a circle or
interval; no quantitative lower length bound, precompactness, or uniqueness of
the compact limit is imported.  The full metric-circle limit on
the working cover is instead proved in
Corollary~\ref{cor:persistent-large-tail-diameter}: the scalar recursion and the
positive quotient-length lower bound give convergence of \(\ell_k\) and
\(E_k\to2\), the two-sided restart estimate propagates the length limit through
each block, and the comparison with Bamler's fibers makes the retained fibers
vanish.  Proposition~\ref{prop:finite-cover-descent} then gives a full circle or
interval limit of positive length on the original manifold.  Exponential
convergence of the unrescaled flow in the Euclidean case is a later consequence
of near-flat compactness and Proposition~\ref{input:flat-stability}, not part of
this quoted input.

\subsection*{The almost-flat theorem}
Proposition~\ref{input:almost-flat-topology} is the three-dimensional
topological consequence of the Gromov--Ruh almost-flat theorem
\cite{GromovAlmostFlat,Ruh}.  No analytic conclusion is added: it is used only
to turn the scale-invariant small-curvature condition into virtual nilpotence
of the fundamental group.

\subsection*{Flat stability and gauge transfer}
For one flat background, \cite[Theorem~3.7]{GIK} gives exponential stability
for the Ricci flow.  In its proof, the corresponding Ricci--DeTurck
convergence is established first, and \cite[Proposition~3.6]{GIK} is then used
to transfer this convergence to the ungauged Ricci flow, including convergence
of the relevant diffeomorphisms.  The compact-family version in
Proposition~\ref{input:flat-stability} is not a
verbatim quotation: it follows by taking a finite subcover of a compact set of
actual flat representatives and then taking the minimum decay rate and maximum
constant among the finitely many fixed-background estimates.  The proposition
is stated for smooth initial metrics in the little-H\"older neighborhood.
Its all-orders conclusion is obtained by positive-time parabolic estimates in
DeTurck gauge and convergence of the associated diffeomorphisms; higher-order
constants may also depend on higher norms of the smooth initial metric, while
the decay rate is uniform on the compact family.  When near-flat compactness
is initially stated modulo diffeomorphisms, the relevant
diffeomorphism must first be applied to obtain actual representatives on the
fixed torus; Proposition~\ref{prop:near-flat-stability-entry} absorbs it into
the finite covering map.

\subsection*{Invariant energy and quotient length}
The length inequality in
Proposition~\ref{input:invariant-energy-length} is
\cite[Lemma~3.7]{LottSesum}, and the reciprocal energy inequality is the
integrated form of \cite[Lemma~3.3]{LottSesum}.  The sentence treating
\(E(t_0)=0\) only records the zero solution of the same maximum-principle
inequality, so no reciprocal expression is used at vanishing energy.

\subsection*{The global scale-invariant curvature bound}
Bamler's \cite[Corollary~1.2]{Bamler} gives the estimate in
Proposition~\ref{input:bamler-typeIII} for all sufficiently large times.
Increasing \(K\), if necessary, over the compact initial time interval gives
the stated bound for every \(t>0\).  Pulling the flow back to a finite cover
leaves the sectional-curvature estimate and its constant unchanged.

\subsection*{The fixed-threshold Bamler alternative}
The initial normalization required by \cite[Theorems~1.1 and~1.4]{Bamler}
is obtained by a constant parabolic rescaling.  Choose \(a>0\) sufficiently
large that \(\widehat g(0)=a g(0)\) satisfies the normalized initial
conditions, and set
\begin{equation}\label{eq:bamler-initial-rescaling}
 \widehat g(u)=a g(u/a),\qquad
 u=at>0,\qquad u^{-1}\widehat g(u)=t^{-1}g(t).
\end{equation}
Regard \(\widehat g\) as a flow with no surgeries, apply the cited theorem,
and scale back.  The normalized metric, its comparison norms, and the
normalized fiber diameters are unchanged; a tail time \(\widehat T\) becomes
\(\widehat T/a\), and an error function \(\widehat\varepsilon(u)\) becomes
\(\widehat\varepsilon(at)\).  This justifies using the theorem for an
arbitrary immortal flow on the working cover.

Proposition~\ref{input:bamler-fixed-threshold} is a fixed-threshold
reformulation, not a verbatim quotation.  Theorem~1.4(b)--(d) of
\cite{Bamler} supplies the topology-specific global alternatives and a single
time-dependent error function \(\varepsilon(t)\to0\).  The detailed
construction of the invariant comparison metrics, together with the
high-order closeness and fiber-diameter estimates, is given in
\cite[Section~4.4, especially Proposition~4.9]{BamlerD}, using
\cite{BamlerA,BamlerB,BamlerC}.  In the whole-manifold torus-bundle cases,
these fiberwise local actions constitute the twisted torus-bundle invariance
fixed in Subsection~\ref{subsec:torus-bundle-conventions}; no globally defined
untwisted \(\TT^2\)-action is being asserted when the monodromy is nontrivial.
In the Sol case, Proposition~\ref{prop:finite-torus-bundle-covers} and the
working-cover convention ensure \(H\in\SL(2,\ZZ)\) and \(\Tr H>2\).
Its positive eigenvalues give a real logarithm \(A=\log H\) with trace zero;
\(\RR^2\rtimes_{e^{tA}}\RR\) is isomorphic to the connected Lie group
\(\Sol\), and \(\ZZ^2\rtimes_H\ZZ\) is a cocompact lattice in it.
The working torus bundle is therefore a Solv manifold in the direct scope of
\cite[Theorem~1.4(d)]{Bamler}.  The base double cover is taken before the
core refinement, whose monodromy is a power of \(H_0^2\); no later cover or
descent of Bamler's comparison metric is invoked.  Apply the cited results
directly to the flow on this fixed working cover.  After fixing a cutoff \(c_0\in(0,1)\) with
\(Kc_0^2<\eps_{\mathrm{af}}\), choose the associated tail time
\(T_B(c_0)\) so that
\[
        \varepsilon(u)\leq c_0/2
        \qquad\text{for every }u\geq T_B(c_0).
\]
In the flat \(\TT^3\) and small-Nil branches, the normalized diameter is then
eventually below this fixed threshold; in the remaining branches, the cited
alternatives give the invariant torus-bundle comparison whenever the diameter
is at least \(c_0\sqrt t\).  Replacing the error by its nonincreasing envelope
\[
        \eps_B^{c_0}(t)=\sup_{u\geq t}\varepsilon(u)
\]
only weakens the metric and fiber estimates.  Moreover, since
\(\eps_B^{c_0}(t)\geq\varepsilon(t)\), every integer \(m\) with
\(0\leq m<\bigl(\eps_B^{c_0}(t)\bigr)^{-1}\) also satisfies
\(m<\varepsilon(t)^{-1}\), so the available derivative order still suffices.
This gives the diameter-defined, tail-uniform
formulation in Proposition~\ref{input:bamler-fixed-threshold}.  In every
application, the working cover is selected once according to the
\hyperref[par:working-cover-convention]{working-cover convention}; the
fixed-threshold alternative is invoked only on that cover and never on an
arbitrary torus bundle representing the Euclidean, Nil, or Sol geometry.
In particular, a negative-trace hyperbolic bundle is handled through its
initial working cover, not by applying Proposition~\ref{input:bamler-fixed-threshold}
directly to that bundle.  The fixed
global application
uses \(c_0=c_{\mathrm{af}}\), with tail data
\(T_B(c_{\mathrm{af}}),\eps_B^{c_{\mathrm{af}}}\); the Euclidean proof makes
one separate application with \(c_0=\zeta\), \(T_B(\zeta)\), and
\(\eps_B^{\zeta}\), without altering the fixed global application.

\addtocontents{toc}{\protect\setcounter{tocdepth}{2}}
\clearpage
\section{Dependency trees}
\label{app:dependencies}
This appendix is a navigational aid rather than an additional part of the
proof.  Tree~I traces the lifted geometric and analytic inputs through
finite-time decay, fibration comparison, and stage preparation to the
quantitative one-block estimates and hence to the qualitative block
contraction theorem.  Tree~II combines the block results with the locally
\(\TT^2\)-invariant energy--length estimates to obtain iteration on a connected
large-diameter interval.  On a persistent tail it also records the quantitative
quotient-length lower bound, retained-fibration equivalence, transition control,
and elementary circle collapse used to prove convergence of the normalized
quotient length and the full circle limit.  Tree~III takes those
outputs as inputs rather than re-expanding their internal block dependencies;
it records the application of Bamler's alternative on the fixed torus-bundle
cover, then combines cover topology and finite-cover descent with the
Euclidean, Nil, and Sol arguments to
obtain the main theorem.

Arrows run from a dependency to a result that uses it.  When present,
an edge label lists only the named parameters passed from the source vertex to
the target vertex; an arrow carrying no such parameter is left unlabeled.
White boxes denote results whose dependencies are
expanded in the displayed tree.  Light-gray boxes denote inputs carried in
from earlier sections or earlier trees, including external results; the main
theorem is shown in darker gray.  The trees show only the principal
nonredundant dependencies; definitions and remarks are omitted, and a numbered
result appears at most once in each diagram.  Blue headings indicate that the
entire box is an internal hyperlink to the corresponding numbered statement.
\clearpage
\phantomsection
\pdfbookmark[2]{Tree I: Analytic estimates, fibration comparison, and quantitative one-block control}{bookmark:dependency-tree-I}
\noindent\makebox[\textwidth][c]{\DependencyTreeI}%

\clearpage
\phantomsection
\pdfbookmark[2]{Tree II: Componentwise iteration and the persistent-tail circle limit}{bookmark:dependency-tree-II}
\noindent\makebox[\textwidth][c]{\DependencyTreeII}%

\clearpage
\phantomsection
\pdfbookmark[2]{Tree III: Long-time limits and assembly of the main theorem}{bookmark:dependency-tree-III}
\noindent\makebox[\textwidth][c]{\DependencyTreeIII}%
\clearpage
\section{Parameters and persistent notation}
\label{app:symbols}
This appendix is a reference guide for the quantitative data carried through
the proof.  It collects only symbols that persist across major statements or
section boundaries, making the hypotheses and conclusions of the stage and
block constructions easier to track; notation local to a single proof is
deliberately excluded. The appendix is not used as an
additional mathematical input.

The table lists Greek-letter symbols first and Roman-letter symbols second,
in alphabetical order within each group.  Upper- and lower-case variants are
grouped by their base letter; numerical or symbolic subscripts precede
alphabetic subscripts, which are ordered without regard to case.  Related
symbols share an entry under the first symbol shown.  For each entry, the
second column links to the statement, equation, subsection, or anchored
paragraph where the symbol is introduced or fixed, and the third column
describes the role that the symbol continues to play later in the paper.  The
table is restricted to structural constants, maps, metrics, state variables,
and scales that cross statement or section boundaries; local coordinate
variables, tensor indices, cutoff functions, and every symbol used only inside
one proof are omitted.  It is intended both as a reading aid for the constant
dependencies and as a compact reference for the persistent data in the stage
and block constructions.

\begingroup
\scriptsize
\hbadness=10000
\renewcommand{\arraystretch}{1.12}
\setlength{\LTleft}{0pt}
\setlength{\LTright}{0pt}
\begin{longtable}{@{}>{\raggedright\arraybackslash}p{0.20\textwidth}>{\raggedright\arraybackslash}p{0.25\textwidth}>{\raggedright\arraybackslash}p{0.49\textwidth}@{}}
\textbf{Symbol} & \textbf{Defined or fixed in} & \textbf{Meaning}\\ \hline
\endfirsthead
\textbf{Symbol} & \textbf{Defined or fixed in} & \textbf{Meaning}\\ \hline
\endhead
\multicolumn{3}{@{}l}{\textit{Greek-letter parameters and maps}}\\[2pt]
\(\alpha\) & Theorem~\ref{thm:noninvariant-decay} & Prescribed factor \(0<\alpha<1\) contracting the non-invariant part on \([10,100]\) relative to its initial \(C^{N,\sigma}\)-norm.\\
\(\beta_*\) & Lemma~\ref{lem:stage-stability} and the joint proof of Theorem~\ref{thm:one-block-contraction} and Proposition~\ref{prop:one-block-quantitative} & Stage-closeness target in the preparation lemma, chosen in the block proof below the decay, fibration-comparison, and restart thresholds.\\
\(\beta_{\mathrm{fib}}\) & Lemmas~\ref{lem:pointwise-fixed-fiber} and~\ref{lem:retained-fibration-persistence} & Closeness threshold produced by the pointwise comparison and propagated as the stage-closeness input in the retained-fibration lemma.\\
\(\beta_{\mathrm{orb}}\) & Theorem~\ref{thm:noninvariant-decay} & Short-orbit diameter threshold used in the zero-average \(C^0\)-estimate.\\
\(\beta_{\mathrm{pert}}\) & Theorem~\ref{thm:noninvariant-decay} & Independent smallness threshold for the Ricci--DeTurck perturbation in the lifted parabolic estimates.\\
\(\beta_{\mathrm{reg}}\) & Proposition~\ref{prop:positive-time-regular-restart} and Lemma~\ref{lem:stage-stability} & Smallness threshold for positive-time regular restart, carried into stage preparation and the block parameter choice.\\
\(\gamma\) & Lemma~\ref{lem:isotopic-fibers-force-fixed} & Relative \(C^1\)-closeness threshold in fiber-equivalent comparison; chosen from the quotient-length lower bound and the target fixed-fiber diameter, and reused in Lemma~\ref{lem:pointwise-fixed-fiber} and Corollary~\ref{cor:persistent-large-tail-diameter}.\\
\(\Gamma\) & Lemma~\ref{lem:finite-quotient-circle-collapse} and Proposition~\ref{prop:finite-cover-descent} & Fixed finite group acting isometrically on the covering metrics; in finite-cover descent it is the deck group, with \(|\Gamma|=d\). Its limiting action on the circle need not be faithful.\\
\(\delta_0\) & Definition~\ref{def:parabolically-rescaled-large-block} and Theorem~\ref{thm:large-component-iteration} & Lower bound for the normalized diameter on a rescaled large block or connected large-diameter interval; it is at least the cutoff in the fixed-threshold application being used.\\
\(\delta_{\mathrm{flat}}\) & Equations~\eqref{eq:euclidean-flat-entry-tolerance}--\eqref{eq:euclidean-flat-entry-slice} & Normalized-diameter tolerance chosen from \(\varepsilon_{\mathrm{flat}}\); it is independent of both the auxiliary cutoff \(\zeta\) and the fixed global cutoff \(c_{\mathrm{af}}\).\\
\(\delta_{\mathrm{J}}\) & Lemma~\ref{lem:transition-jump} & Upper threshold for the relative \(C^1\)-jump, depending only on the chosen reference energy bound \(E_0\). The component iteration uses \(E_0=3\).\\
\(\eps_{\mathrm{af}}\) & \hyperref[par:almost-flat-threshold]{Almost-flat threshold paragraph} & Dimensional almost-flat threshold used to choose every admissible cutoff \(c_0\), including the fixed global cutoff \(c_{\mathrm{af}}(K)\).\\
\(\varepsilon_{\mathrm{av}}\) & Lemma~\ref{lem:high-order-averaged-slice} & High-order smallness threshold under which averaging preserves bounded lifted geometry.\\
\(\eps_B^{c_0}(t)\) & Proposition~\ref{input:bamler-fixed-threshold} and the \hyperref[par:standing-fixed-threshold-application]{standing fixed-threshold application} & Nonincreasing comparison error attached to the application with cutoff \(c_0\); different cutoffs have separately labeled error functions.\\
\(\eps_{\mathrm{fib}}\) & Lemmas~\ref{lem:pointwise-fixed-fiber} and~\ref{lem:retained-fibration-persistence} & Comparison-metric and comparison-fiber threshold produced by the pointwise lemma and used uniformly throughout a retained-fibration block.\\
\(\varepsilon_{\mathrm{flat}}\) & Proposition~\ref{prop:near-flat-stability-entry} & Analytic smallness threshold for entering the uniform fixed-coordinate flat-stability neighborhood; it is independent of the geometric cutoff used in Bamler's alternative.\\
\(\zeta,T_B(\zeta),\eps_B^{\zeta}\) & Proof of the Euclidean case in Theorem~\ref{thm:euclidean-sol-long-time-limits} & Auxiliary cutoff, with \(K\zeta^2<\eps_{\mathrm{af}}\), and its separately labeled tail data. The liminf argument allows every such \(\zeta\), without changing the fixed global application.\\
\(\eta\) & Lemma~\ref{lem:isotopic-fibers-force-fixed} and Definition~\ref{def:admissible-stage-start} & In fiber comparison, a threshold for comparison-fiber diameter, chosen with \(\gamma\). In the admissible-start definition, the same letter denotes a generic non-invariant error threshold; these are distinct uses.\\
\(\eta_0\) & Theorem~\ref{thm:large-component-iteration} & Initial non-invariant error threshold for iteration on a connected large-diameter interval; it depends only on \(K,\delta_0,N,\sigma\) and is independent of \(D_1\).\\
\(\eta_{\mathrm{blk}}\) & Theorem~\ref{thm:one-block-contraction} & Start-and-restart error threshold for the block contraction.\\
\(\eta_{\mathrm{cd}}\) & Proposition~\ref{prop:uniform-lifted-rdt-stability} & Initial perturbation threshold for uniform Ricci--DeTurck continuous dependence.\\
\(\eta_{\mathrm{st}}\) & Lemma~\ref{lem:stage-stability} and Equation~\eqref{eq:stage-eta-choice} & Stage-start error threshold for the prescribed target \(\beta_*\), also satisfying \(C_1\eta_{\mathrm{st}}\leq\tfrac12\); the block threshold \(\eta_{\mathrm{blk}}\) is chosen below it.\\
\(\theta\) & Lemmas~\ref{lem:pointwise-fixed-fiber} and~\ref{lem:retained-fibration-persistence} & Preliminary fixed-fiber target produced by the pointwise comparison and retained strictly throughout one large block.\\
\(\theta_{\mathrm{blk}}\) & Theorem~\ref{thm:one-block-contraction} and Proposition~\ref{prop:one-block-quantitative} & Fixed-fiber target for the block iteration, reproduced by every completed block.\\
\(\theta_{\max}\) & Lemmas~\ref{lem:pointwise-fixed-fiber} and~\ref{lem:retained-fibration-persistence} & Prescribed upper target for the fixed-fiber diameter; in the block proof it is set to \(\beta_{\mathrm{orb}}^{\mathrm{dec}}/4\).\\
\(\lambda\) & Theorem~\ref{thm:noninvariant-decay} & Positive lower bound for the quotient-circle length in the finite-time decay theorem.\\
\(\Lambda_*\) & Lemma~\ref{lem:stage-stability} and Theorem~\ref{thm:one-block-contraction} & Uniform \(\BG_N\)-bound, chosen as \(\max\{\Lambda_{\mathrm{reg}},\Lambda_{\mathrm{init}}\}\), depending only on \(K,N,\sigma\). It controls the first averaged slice and is reproduced at every restart.\\
\(\lambda_{\mathrm{p}}\) & Lemmas~\ref{lem:pointwise-fixed-fiber} and~\ref{lem:retained-fibration-persistence} & Quotient-length lower threshold in the pointwise comparison and the retained-fibration conclusion; the block construction sets \(\lambda_{\mathrm{p}}=\delta_0/100\).\\
\(\Lambda_{\mathrm{reg}}\) & Proposition~\ref{prop:positive-time-regular-restart} & Uniform \(\BG_N\)-bound for the rescaled locally \(\TT^2\)-invariant endpoint and the averaged restart.\\
\(\pi,\pi_B,\pi_{\mathrm{cyc}},\pi_{\mathrm{fin}}\) & Proposition~\ref{input:bamler-fixed-threshold}, Section~\ref{sec:collapse-inputs}, and the \hyperref[par:metric-gauge-legend]{metric-and-gauge notation} & The retained fibration; the pulled-back Bamler fibration \(\pi_B(\tau)\), denoted \(\pi_t^B\) in physical time; the infinite cyclic covering map; and the fixed finite regular covering map, respectively.\\
\(\rho\) & Propositions~\ref{input:flat-stability} and~\ref{prop:near-flat-stability-entry} & Exponent \(0<\rho<1\) in the little-H\"older space \(h^{2+\rho}\) used for flat stability. It is unrelated to the invariant metric called \(\rho\) in Lemma~\ref{lem:transition-jump}.\\
\(\rho^+,\rho_k^+\) & Proposition~\ref{prop:positive-time-regular-restart} and Theorem~\ref{thm:large-component-iteration} & Rescaled invariant endpoint: \(\rho^+=100^{-1}r(100)\) generically and \(\rho_k^+=100^{-1}r_k(100)\) on stage \(k\).\\
\(\sigma\) & \hyperref[par:standing-regularity-convention]{Standing regularity convention} and Appendix~\ref{app:parabolic-estimates} & H\"older exponent \(0<\sigma<1\), fixed once for the stage and block constructions and used in the spatial and parabolic norms.\\
\(\tau\) & Definition~\ref{def:parabolically-rescaled-large-block} & Parabolically rescaled time: the ambient block \(p\) is defined on \([1/2,101]\), while the stage families \(q\) and \(r\) are used on \([1,100]\).\\
\(\Phi_k\) & Theorem~\ref{thm:large-component-iteration} & Stage-start pullback diffeomorphism, with \(p_k(\tau)=\Phi_k^*(T_k^{-1}g(T_k\tau))\) and \(\Phi_{k+1}=\Phi_k\circ\psi_k(100)\).\\
\(\psi_k,\Psi\) & Theorem~\ref{thm:large-component-iteration} & Stage diffeomorphisms satisfying \(q_k=\psi_k^*p_k\) and \(\psi_k(1)=\Id\), and the assembled map \(\Psi(T_k\tau)=\Phi_k\circ\psi_k(\tau)\); hence \(g^{\prime}=\Psi^*g\).\\[4pt]

\multicolumn{3}{@{}l}{\textit{Roman-letter parameters and notation}}\\[2pt]
\(\mathcal A,\ol U,U^\perp\) & Subsection~\ref{subsec:torus-averaging-collapse} and Lemmas~\ref{lem:averaging-invariant-holder} and~\ref{lem:averaging-parabolic-holder} & Averaging projection for the fixed local \(\TT^2\)-structure, the average \(\ol U=\mathcal A U\), and the non-invariant part \(U^\perp=U-\mathcal A U\); the two lemmas give the fixed-time and parabolic H\"older bounds.\\
\(b,b_k\) & Equation~\eqref{eq:stage-start-data}, Proposition~\ref{prop:one-block-quantitative}, and Theorem~\ref{thm:large-component-iteration} & First numerical admissibility coordinate: the non-invariant start error on one stage and at the start of stage \(k\); the block estimate gives \(b_{k+1}\leq b_k/2\).\\
\(\BG_m(s;\Lambda)\) & Section~\ref{sec:finite-time-decay} & Scale-one bounded geometry on the universal cover: \(\inj_{\widetilde s}\geq\Lambda^{-1}\) and \(|\widetilde\nabla^j\Rm_{\widetilde s}|\leq\Lambda\) for \(0\leq j\leq m+3\), where \(m\geq2\), \(\Lambda\geq1\). This is a separate admissibility condition.\\
\(c_0\) & Proposition~\ref{input:bamler-fixed-threshold} and the \hyperref[par:standing-fixed-threshold-application]{standing fixed-threshold application} & Generic cutoff of the particular Bamler application used throughout one block or component iteration.\\
\(c_{\mathrm{af}}=c_{\mathrm{af}}(K)\) & The fixed global application following Proposition~\ref{input:bamler-fixed-threshold} & Fixed global cutoff used for the global almost-flat/torus-bundle decomposition; in the generic block machinery it is the specialization \(c_0=c_{\mathrm{af}}\).\\
\(C_{\mathrm{cd}}\) & Proposition~\ref{prop:uniform-lifted-rdt-stability} & Linear continuous-dependence constant for one Ricci--DeTurck comparison.\\
\(C_{\mathrm{fib}}\) & Lemma~\ref{lem:different-fibrations-small-diameter} & Constant, equal to \(B+1\) in the proof, controlling total diameter when two short-fiber fibrations are fiber-inequivalent.\\
\(C_{\mathrm{J}}\) & Lemma~\ref{lem:transition-jump} & Constant depending only on the reference energy bound \(E_0\), controlling a one-sided energy increase and the absolute logarithmic length change. The component iteration uses \(E_0=3\).\\
\(C_L\) & Lemma~\ref{lem:finite-quotient-circle-collapse} and Corollary~\ref{cor:persistent-large-tail-diameter} & Metric circle of circumference \(L>0\). In the working-cover limit \(L=\ell_\infty\); its diameter is \(L/2\).\\
\(C_{\mathrm{reg}}\) & Proposition~\ref{prop:positive-time-regular-restart} & Constant in the high-order positive-time smoothing estimate.\\
\(C_{\mathrm{st}}\) & Lemma~\ref{lem:stage-stability} & Stagewise linear stability constant, also controlling the high-order endpoint jump.\\
\(d,d_k\) & Equation~\eqref{eq:stage-start-data}, Theorem~\ref{thm:one-block-contraction}, and Theorem~\ref{thm:large-component-iteration} & Second numerical admissibility coordinate: the retained-fiber ambient diameter at a generic stage start and at the start of stage \(k\). This use of \(d\) is unrelated to the covering degree below.\\
\(d\) (cover degree) & Lemma~\ref{lem:finite-cover-diameter} and Proposition~\ref{prop:finite-cover-descent} & Positive integer degree of the fixed finite cover; for a regular cover, \(d=|\Gamma|\). This is not the stage-start fiber diameter denoted by \(d\).\\
\(D_1\) & Theorem~\ref{thm:large-component-iteration} & Initial upper bound for the normalized diameter at the chosen start of a connected large-diameter interval.\\
\(\mathcal D_t^s\) & Equation~\eqref{eq:metric-compatible-time-derivative} & Metric-compatible covariant time derivative used in the intrinsic parabolic H\"older norms.\\
\(d_{\TT^2}(r)\) & Subsection~\ref{subsec:torus-averaging-collapse} & Supremum of the ambient diameters of the fixed torus fibers in the metric space \((M,r)\).\\
\(E(t),e_k(\tau),E_k\) & Proposition~\ref{input:invariant-energy-length}, Lemma~\ref{lem:scalar-quotient-recursion}, and Theorem~\ref{thm:large-component-iteration} & Maximum vertical energy of an invariant flow, its stagewise profile \(e_k(\tau)\), and the start values \(E_k=e_k(1)\) created after the first restart and propagated by the scalar recursion; on a persistent tail, \(E_k\to2\).\\
\(g(t)\) & Theorem~\ref{thm:main} and the \hyperref[par:working-cover-convention]{working-cover convention} & Immortal Ricci flow in physical time. The same notation is used for its lift while the block construction is carried out on the working cover, denoted simply by \(M\).\\
\(g'(t)\) & Theorem~\ref{thm:large-component-iteration} & Continuous, piecewise smooth DeTurck-gauge pullback of the original physical-time flow on the controlled blocks.\\
\(g_{\mathrm{fin}}(t)\) & Sections~\ref{sec:collapse-inputs}, \ref{sec:finite-cover-descent}, and~\ref{sec:main-theorem-proof} & Pullback of the original Ricci flow to the fixed finite regular working cover.\\
\(g_{\RDT}\) & Theorem~\ref{thm:noninvariant-decay} & Ricci--DeTurck solution relative to an invariant background in the finite-time decay theorem.\\
\(g_t^B\) & Proposition~\ref{input:bamler-fixed-threshold} & Bamler's locally \(\TT^2\)-invariant comparison metric at physical time \(t\).\\
\(H\) & \hyperref[par:working-cover-convention]{Working-cover convention} and Subsection~\ref{subsec:torus-bundle-conventions} & Monodromy matrix on the fixed working cover: \(I\) in the Euclidean case, nontrivial unipotent in the Nil case, and hyperbolic with \(\Tr H>2\) in the Sol case.\\
\(h(t)\) & Theorem~\ref{thm:large-component-iteration} & Right-continuous invariant comparison family: \(h(T_k\tau)=T_kr_k(\tau)\) for \(1\leq\tau<100\), with the averaged restart at \(T_{k+1}\). It is a Ricci flow on each smooth piece, but may jump at switching times.\\
\(h_B(\tau)\) & \hyperref[par:metric-gauge-legend]{Metric-and-gauge notation} and the \hyperref[par:rescaled-bamler-comparison]{rescaled-block comparison paragraph} & Bamler comparison metric on a rescaled block, possibly for a time-dependent comparison fibration.\\
\(I\) & Theorem~\ref{thm:large-component-iteration} & Connected time interval on which the normalized diameter stays above \(\delta_0\); it may be an unbounded tail or a finite component with a terminal partial block.\\
\(i_1,i_2,i_B\) & Lemma~\ref{lem:universal-cover-inj} & Universal-cover injectivity-radius constants for \(B^{-1}h\leq g\leq Bh\), with \(h\) invariant. In Tree~I, \(i_1\) and \(i_2\) are the cases \(B=1\) and \(B=2\), respectively: they supply the invariant-metric input to Proposition~\ref{prop:positive-time-regular-restart} and the comparison-metric input to Lemma~\ref{lem:stage-stability}. For \(|\sec_g|\leq\Lambda\), the lower bound is \(i_B\Lambda^{-1/2}\). The \(i_1\) used locally in the proof of Lemma~\ref{lem:slice-to-parabolic-geometry} is a separate constant with the dependences stated there.\\
\(K\) & Proposition~\ref{input:bamler-typeIII} & Global scale-invariant curvature constant obtained by enlarging Bamler's eventual constant over the compact initial interval.\\
\(K_{\mathrm{st}}\) & Lemma~\ref{lem:stage-stability} & Uniform sectional-curvature constant for the invariant reference flow on one stage.\\
\(L(S^1,r_{S^1})\) & Subsection~\ref{subsec:one-dimensional-estimates} & Length of the quotient circle for an invariant metric.\\
\(\ell_\infty\) & Lemma~\ref{lem:scalar-quotient-recursion} and Corollary~\ref{cor:persistent-large-tail-diameter} & Positive limiting normalized quotient length on a persistent tail: \(\ell_k(\tau)/\sqrt\tau\to\ell_\infty\) uniformly on \([1,100]\). It is the circumference of the working-cover circle limit.\\
\(\ell_k(\tau),\ell_k\) & Theorem~\ref{thm:large-component-iteration} and Lemma~\ref{lem:scalar-quotient-recursion} & Quotient-circle length along stage \(k\) and its start value \(\ell_k=\ell_k(1)\), introduced for \(k\geq1\) after the first restart and then propagated by the scalar recursion; it is not an admissibility coordinate, and on a persistent tail \(\ell_k\) converges.\\
\(m\) (circle quotient) & Lemma~\ref{lem:finite-quotient-circle-collapse} and Proposition~\ref{prop:finite-cover-descent} & Order of the rotation subgroup of the limiting circle action. The quotient has circumference \(L/m\) in the cyclic case and interval length \(L/(2m)\) in the dihedral case; \(m\) divides \(|\Gamma|\).\\
\(M_{\mathrm{cyc}},M_{\mathrm{fin}},\widetilde M\) & Sections~\ref{sec:collapse-inputs}, \ref{sec:finite-time-decay}, and~\ref{sec:finite-cover-descent} & Infinite cyclic cover, fixed finite regular working cover, and universal cover; tensors on the cyclic cover carry the subscript \(\mathrm{cyc}\), while universal-cover lifts carry tildes.\\
\(m_{\mathrm{flat}}\) & Proposition~\ref{prop:near-flat-stability-entry} & Finite curvature-derivative order used in the scale-invariant near-flat entry criterion.\\
\(N\) & \hyperref[par:standing-regularity-convention]{Standing regularity convention} & Integer spatial regularity order \(N\geq2\), fixed once for the stage and block constructions.\\
\(p(\tau)\) & Definition~\ref{def:parabolically-rescaled-large-block} & Pulled-back parabolic rescaling \(\Phi^*(S^{-1}g(S\tau))\) of the original flow.\\
\(p_{\mathrm{st}}\) & Proposition~\ref{prop:near-flat-stability-entry} & Connected finite covering map that absorbs the compactness diffeomorphism and places the diameter-normalized metric in \(\mathcal U_{\mathcal F}\).\\
\(\PBG_m(s;\Lambda;[a,b])\) & Section~\ref{sec:finite-time-decay} and Equation~\eqref{eq:bounded-lifted-geometry} & The \(\BG_m\) injectivity and curvature-derivative bounds at every time, together with a fixed-coordinate uniformly locally finite parabolic atlas whose radius, coefficients, overlaps, and cutoffs are uniformly controlled.\\
\(q(\tau),q_k(\tau)\) & Lemma~\ref{lem:stage-stability} and Theorem~\ref{thm:large-component-iteration} & Ricci--DeTurck metric on one stage and on stage \(k\).\\
\(q^+(1)\) & Proposition~\ref{prop:positive-time-regular-restart} & Rescaled Ricci--DeTurck endpoint \(q^+(1)=100^{-1}q(100)\), which becomes the next stage's initial metric.\\
\(r(\tau),r_k(\tau)\) & Lemma~\ref{lem:stage-stability} and Theorem~\ref{thm:large-component-iteration} & Invariant Ricci flow starting from the averaged stage-start metric.\\
\(r^+(1)\) & Proposition~\ref{prop:positive-time-regular-restart} & Averaged restart metric \(r^+(1)=\overline{q^+(1)}\), used as the next invariant stage start.\\
\(r_{\mathrm D}(\tau)\) & Equation~\eqref{eq:two-deturck-gauges} and Lemma~\ref{lem:stage-stability} & Auxiliary Ricci--DeTurck flow with background \(p\) and initial metric \(r(1)\), used only to construct and control the invariant Ricci flow \(r\).\\
\(S,T_k\) & Definition~\ref{def:parabolically-rescaled-large-block} and Theorem~\ref{thm:large-component-iteration} & Physical start time \(T_0=S\) of the first rescaled block and the geometric sequence \(T_k=100^kS\) of later block starts.\\
\(\TT^3_*,\mathcal F,\mathcal U_{\mathcal F}\) & Propositions~\ref{input:flat-stability} and~\ref{prop:near-flat-stability-entry} & Fixed model torus, compact family of actual flat representatives on it, and the associated uniform fixed-coordinate stability neighborhood for smooth initial metrics in the Euclidean near-flat entry argument.\\
\(T_B(c_0)\) & Proposition~\ref{input:bamler-fixed-threshold} and the \hyperref[par:standing-fixed-threshold-application]{standing fixed-threshold application} & Tail threshold attached to the application with cutoff \(c_0\); for example, the fixed global and Euclidean applications use \(T_B(c_{\mathrm{af}})\) and \(T_B(\zeta)\).\\
\(T_{\mathrm{blk}}\) & Theorem~\ref{thm:one-block-contraction} & Late start threshold depending on \(K,\delta_0,N,\sigma\) and, possibly, the chosen tail data \(T_B(c_0),\eps_B^{c_0}\). The geometric and analytic block thresholds do not acquire this tail-data dependence.\\
\(T_{\mathrm{p}}\) & Theorem~\ref{thm:large-component-iteration} & Late start threshold for component iteration, depending on \(K,\delta_0,N,\sigma\) and, possibly, the selected Bamler tail data, but not on \(D_1\).\\
\(X_s^N[a,b]\) & Equation~\eqref{eq:XY-spaces}, Paragraph~\ref{par:lifted-norm-convention}, and Lemma~\ref{lem:averaging-parabolic-holder} & Intrinsic lifted solution norm \(\mathcal C_s^{N+\sigma,(N+\sigma)/2}\). Averaging is contractive when \(s(t)\) is invariant under the fixed torus structure.\\
\(X_{s,{\mathrm{fc}}}^N[a,b]\) & Lemmas~\ref{lem:intrinsic-fixed-chart-equivalence} and~\ref{lem:averaging-parabolic-holder} & Fixed-chart uniformly local solution norm, uniformly equivalent to \(X_s^N\) under \(\PBG\). For \(s(t)\) invariant under the fixed torus structure and satisfying that bound, averaging is uniformly bounded.\\
\(Y_s^N[a,b]\) & Equation~\eqref{eq:XY-spaces}, Paragraph~\ref{par:lifted-norm-convention}, and Lemma~\ref{lem:averaging-parabolic-holder} & Intrinsic lifted forcing norm \(\mathcal C_s^{N-2+\sigma,(N-2+\sigma)/2}\). Averaging is contractive when \(s(t)\) is invariant under the fixed torus structure.\\
\(Y_{s,{\mathrm{fc}}}^N[a,b]\) & Lemmas~\ref{lem:intrinsic-fixed-chart-equivalence} and~\ref{lem:averaging-parabolic-holder} & Fixed-chart uniformly local forcing norm, uniformly equivalent to \(Y_s^N\) under \(\PBG\). For \(s(t)\) invariant under the fixed torus structure and satisfying that bound, averaging is uniformly bounded.\\
\end{longtable}
\endgroup

\end{document}